\documentclass[11pt,a4paper]{article}

\usepackage[utf8]{inputenc}
\usepackage[T1]{fontenc}
\usepackage[a4paper,margin=1in]{geometry}
\usepackage[final,tracking=true,expansion=true,protrusion=true]{microtype}
\usepackage{amsmath,amssymb,amsthm,mathtools}
\usepackage{mathrsfs} 
\usepackage{newtxtext}\usepackage{newtxmath} 

\usepackage{graphicx}
\usepackage{tikz}
\usepackage{tikz-cd}
\usetikzlibrary{trees,arrows.meta}
\usepackage{float}

\usepackage{enumitem}
\usepackage{booktabs,tabularx}
\usepackage[labelfont=bf]{caption}
\usepackage{subcaption}

\usepackage{enumitem}
\setlist[itemize]{topsep=0.5em, itemsep=0em, parsep=0em, left=1.5em}
\setlist[enumerate]{topsep=0.5em, itemsep=0em, parsep=0em, left=1.5em}

\usepackage[unicode,colorlinks=true,
  linkcolor=blue,citecolor=purple!70!,urlcolor=blue!50!black]{hyperref}

\newtheoremstyle{boldplain}
  {\topsep}   
  {\topsep}   
  {\normalfont}  
  {0pt}       
  {\bfseries} 
  {.}         
  {.5em}      
  {\thmname{#1}\thmnumber{ #2}\thmnote{ (#3)}} 

\newtheoremstyle{bolddefinition}
  {\topsep}   
  {\topsep}   
  {\normalfont} 
  {0pt}         
  {\bfseries}   
  {.}           
  {.5em}        
  {\thmname{#1}\thmnumber{ #2}\thmnote{ (#3)}}

\newtheoremstyle{indentedremark} 
{\topsep} 
{\topsep} 
{\normalfont\leftskip=0.5cm} 
{0pt} 
{\itshape} 
{.} 
{.5em} 
{} 

\theoremstyle{boldplain}
\newtheorem{theorem}{Theorem}[section]
\newtheorem{lemma}[theorem]{Lemma}
\newtheorem{proposition}[theorem]{Proposition}

\theoremstyle{bolddefinition}
\newtheorem{definition}[theorem]{Definition}
\newtheorem{example}[theorem]{Example}

\theoremstyle{indentedremark}
\newtheorem{remark}[theorem]{Remark}

\newcommand{\N}{\mathbb{N}}

\DeclareMathOperator{\Hom}{Hom}
\DeclareMathOperator{\End}{End}
\DeclareMathOperator{\Aut}{Aut}
\DeclareMathOperator{\id}{id}

\begin{document}

\title{A gentle introduction to Algebraic Operads}
\author{Felicia Ferraioli}
\date{} 

\maketitle

\begin{abstract}
This review introduces the theory of algebraic operads, a formalism that unifies many branches of modern algebra. The central idea is that key algebraic structures — associative, commutative, and Lie algebras — can be understood as algebras over specific operads, known as the “three graces” of algebra: $\mathsf{As}$, $\mathsf{Com}$, and $\mathsf{Lie}$.
We present two complementary perspectives on operads. The first is the classical definition, which is shown to be equivalent to the structure of a multicategory with a single object. The second is a constructive approach, where operads are defined by generators and relations, extending familiar notions such as free objects and quotients to the operadic setting.
A main result discussed is the categorical equivalence between the category of algebras over each of the operads $\mathsf{As}$, $\mathsf{Com}$, and $\mathsf{Lie}$ and the corresponding classical algebraic categories. This shows that operad theory provides a unifying language capable of encoding entire algebraic theories within single mathematical objects.
This review is intended for undergraduate students and provides an accessible primer on the language and ideas of operad theory.
\end{abstract}

\tableofcontents

\cleardoublepage 
\markboth{}{}    

\section*{Acknowledgments}

I once reflected on an interesting thought experiment: if I could lay my internal organs out on a table and listen to my heart beat, I would not say, "That is me", but rather, "That is my heart". It is a subtle distinction between what we possess and what we are. And I firmly believe that we are what we love.

\bigskip

\noindent I love Mathematics. I used to think that passion was its only driving force, but that is not entirely accurate: one is also taught to love it. This document is an adaptation of my bachelor's thesis in Mathematics at the Università degli Studi di Salerno and it is a testament to that teaching. For this, I wish to express my deepest gratitude to my supervisors, Professor Chiara Esposito, Professor Luca Vitagliano, and Dr. Marvin Dippell. They guided me through this work with expertise and passion, allowing me to witness firsthand its beauty and sharing with me an authentic view on research. I am thankful for their support, their inspiration, and for contributing to the image of the person and researcher I aspire to become.

\bigskip

\noindent Finally, I wish to extend my heartfelt gratitude to my family and friends. Their unwavering support, encouragement, and patience were the foundation upon which this work was built. This achievement is as much theirs as it is mine.

\cleardoublepage 
\markboth{}{}    

\section{Introduction}
\label{sec:1}
\begin{flushright}
\textit{“Mathematics is the art of giving the same name to different things.”}\\
\smallskip
\textbf{Henri Poincaré}
\end{flushright}
Abstraction lies at the heart of mathematical thought. Across many areas of mathematics, we encounter sets equipped with operations — addition of real numbers, function composition, matrix multiplication — that, despite acting on different objects, share key structural properties: associativity, and often the presence of a neutral element. These shared features motivate the definition of an abstract algebraic structure known as a \emph{monoid}.

Rather than treating each example independently, it is often more fruitful to extract and formalize the common properties and work directly with the abstract structures that embody them. For instance, the associativity axiom ensures that any product $x_1 x_2 \cdots x_n$ in a monoid is well-defined regardless of how it is parenthesized. This kind of abstraction allows us to formulate general theorems that apply across multiple contexts, avoiding redundancy and providing conceptual clarity.

Many familiar algebraic structures — semigroups, monoids, groups, rings, vector spaces — also share deeper patterns. Most are built around associative operations, often accompanied by neutral elements or symmetries. This observation leads naturally to a further level of abstraction: rather than studying each type of algebra separately, we can ask whether there exists a general framework capable of describing whole families of such structures. \emph{Operads} provide precisely this kind of meta-algebraic formalism: they are designed to describe operations (of varying arity) and the rules for composing them \cite{samchuckschnarch}.

This review explores the theory of algebraic operads. A particularly insightful analogy appears in \cite{hilgerponcin}:
\begin{quotation}
“Operads are to algebras, what algebras are to matrices, or, better, to representations.”
\end{quotation}
Just as algebras describe how operations act on a space, operads encode the types of algebraic structures compatible with those operations. When viewed as algebraic theories, operads give rise to algebras — models that realize the abstract operations in concrete terms.

One way to understand operads is as a generalization of the notion of representation. In classical theory, a representation of an algebra is a homomorphism into $\End(V)$, the algebra of linear endomorphisms of a vector space. Similarly, a representation of an operad — also known as an \emph{algebra over an operad} — is a morphism into the collection ${\End}_V = \{ \Hom(V^{\otimes n}, V) \}_{n \geq 0}$, which encodes all multilinear operations on $V$. In this sense, operads extend the representation-theoretic perspective to include $n$-ary operations and their compositional structure.

Although the formal definition of operads first appeared in the 1970s in the work of May \cite{may_geometry}, the ideas trace back to earlier developments, such as Stasheff’s work on associahedra \cite{Stasheff1963} and the theory of operator categories by Boardman and Vogt \cite{boardmanvogt}. May coined the term \emph{operad} — a blend of ``operation'' and ``monad'' — to describe the structure of loop spaces in algebraic topology.

Since the 1990s, operad theory has experienced a remarkable resurgence, expanding far beyond its topological origins. Its influence is particularly profound in mathematical physics, where it provides the fundamental language for describing phenomena in quantum field theory and string theory. A prime example is Kontsevich's celebrated Formality theorem which proves the existence and classification of star products on any Poisson manifolds \cite{Kontsevich2003}. This result relies on the structure of $L_\infty$-algebras, a type of ``Lie algebra up to homotopy'' whose properties are naturally governed by an operad. The conceptual power of this approach was confirmed when Tamarkin provided a groundbreaking new proof of the formality theorem by masterfully employing the language of operads, specifically through the resolution of the little discs operad \cite{Tamarkin1998}. The theory of formality remains today an extremely active field of research, as demonstrated by many subsequent developments, e.g.,  \cite{Calaque2005}, \cite{CalaqueWillwacher2015}, \cite{Dolgushev2005a}, \cite{esposito:2021}, \cite{esposito:2025}, \cite{Willwacher2015}.

Furthermore, operads have become indispensable in string theory. The algebraic structure of string interactions in String Field Theory is governed by homotopy algebras, and the Batalin-Vilkovisky (BV) formalism, an essential framework for quantization, is formally described as the structure of an algebra over the BV operad \cite{Getzler1994, Zwiebach1993}. This operadic perspective has unveiled deep connections, for instance, by showing that the classic Feynman diagrams correspond graphically to the compositions defined by the Homotopy Transfer Theorem, a foundational result in homological algebra with historical roots in the work of Kadeishvili \cite{Kadeishvili1980} and whose modern treatment is central to operad theory \cite{lodayvallette, vallette14}.

The influence of operads extends deeply into algebraic geometry and its applications to topological string theory. The homology of the Deligne-Mumford moduli spaces of stable curves \cite{DeligneMumford1969}, central objects in these theories, is naturally endowed with an operad structure. This discovery has led to profound insights into problems such as the intersection theory on these spaces \cite{Kontsevich1992, Getzler1995, GetzlerKapranov1995}.

Of course, operad theory remains central to its native soil of algebraic topology. It was originally introduced by May to describe the algebraic structure of iterated loop spaces \cite{may_geometry} and now provides the definitive language for ``algebras up to homotopy,'' such as the $A_\infty$-algebras pioneered by Stasheff \cite{Stasheff1963}.

Finally, operads offer a powerful unifying perspective on algebra itself. They allow for the construction of a general cohomology theory that subsumes classical theories like Hochschild, Chevalley-Eilenberg, and Harrison cohomology \cite{lodayvallette}. Tools like Koszul duality, extended to the operadic context by Ginzburg and Kapranov \cite{GinzburgKapranov1994}, have also become fundamental in studying the structure and properties of algebras.

One of the reasons for this renewed interest is the unifying power of operads. The motivation behind this review was to show how a wide variety of algebraic structures — often studied independently — can be elegantly and systematically described using operadic language. In particular, we will see how associative, commutative, and Lie algebras are governed by three fundamental operads, known in the literature as the \emph{“three graces”} of algebra: {$\mathsf{As}$}, {$\mathsf{Com}$}, and {$\mathsf{Lie}$}. This approach not only simplifies many classical proofs, but also offers a more conceptual understanding of the results and suggests natural generalizations. As Leinster puts it:
\begin{quote}
“Is it just language? There would be no shame if it were: language can have the most profound effect. A new language can make new concepts thinkable and render old, seemingly obscure concepts suddenly natural and obvious.”
\end{quote}
At the same time, the elegance of operad theory is known for its technical prerequisites. Many standard references assume prior knowledge of homotopical algebra or higher category theory. This review aims to offer a \emph{self-contained and accessible introduction to algebraic operads} for readers with a solid background in abstract algebra and category theory. The goal is to develop the theory using elementary tools, emphasizing how operads encode algebraic theories and providing a constructive approach via generators and relations.
A key methodological feature of this work is the systematic use of \emph{tree diagrams}. Going beyond their traditional role as simple illustrations, they are formalized here through graph-theoretic language to rigorously express compositional rules.

The main sources of inspiration are the books \emph{Algebraic Operads} by Loday and Vallette \cite{lodayvallette}, and \emph{Operads in Algebra, Topology and Physics} by Markl, Schnider, and Stasheff \cite{markl_shnider_stasheff}, along with more recent lecture notes by Samchuk-Schnarch \cite{samchuckschnarch}, Hilger and Poncin \cite{hilgerponcin}, and Vallette \cite{vallette14}.
The literature on operads is vast, and further important references include, e.g., \cite{krizmay},  \cite{smirnov}, and the works of Fresse \cite{fresse2009,fresse2017}.

The work is structured into five sections. Section 2 introduces the necessary mathematical background on category and representation theory. Section 3 presents the formal definition of operads, develops key examples like As and Com, and relates them to categories of algebras. Section 4 explores equivalent formulations to highlight their flexibility, while Section 5 uses a constructive approach based on generators and relations to define free operads and the Lie operad. Finally, Section 6 concludes with a survey of generalizations (such as colored operads, properads, and PROPs) and their applications.

\section{Preliminary Notions}
\label{chap:preliminaries}

In this section we introduce the mathematical concepts and tools that are prerequisites for the study of the theory of algebraic operads. We begin with a review of category theory, whose formal language is indispensable for defining the notion of an operad. Subsequently, we recall elements of the representation theory of groups and algebras, which provide the context for understanding the motivations of this work.

\subsection{Category Theory}

Category theory offers the language and conceptual foundations to rigorously define operads, the object of study of this review. The central idea of this approach is to abstract the properties of mathematical structures not by focusing on the elements of sets, but on the maps (or morphisms) that preserve their structure.

In this subsection, we  therefore introduce the basic notions of category, object, morphism, and composition. Subsequently, we  extend the discussion to more structured contexts, defining symmetric monoidal categories \cite{maclane98}. This richer environment is the natural setting in which the modern theory of operads develops, allowing for the description of the composition of operations in domains such as vector spaces, which will be of primary interest in the subsequent sections of this work.

\begin{definition}[Category]
A \emph{category} $\mathscr{C}$ consists of the following data:
\begin{enumerate}
    \item A collection (e.g., a set or a proper class) of \emph{objects}, denoted by $\mathrm{Ob}(\mathscr{C})$.
    \item For every pair of objects $x, y \in \mathrm{Ob}(\mathscr{C})$, a collection of \emph{morphisms} from $x$ to $y$, denoted by $\Hom_{\mathscr{C}}(x,y)$ or simply $\Hom(x,y)$ when the category is clear.
    \item For every triple of objects $x, y, z \in \mathrm{Ob}(\mathscr{C})$, a \emph{composition} function $\circ: \Hom(y,z) \times \Hom(x,y) \rightarrow \Hom(x,z)$, which we usually write as $\circ(g,f) =: g \circ f$.
    \item For every object $x \in \mathrm{Ob}(\mathscr{C})$, an element $id_x \in \Hom(x,x)$ called the \emph{identity morphism} for $x$.
\end{enumerate}
Morphisms from $x$ to $y$ are also called morphisms with \emph{domain} $x$ and \emph{codomain} $y$, morphisms with \emph{source} $x$ and \emph{target} $y$, or \emph{arrows} from $x$ to $y$. For any morphism $f \in \Hom(x,y)$, we write $f:x \rightarrow y$ to indicate its source and target. We also require the following axioms to be satisfied:
\begin{itemize}
    \item Composition is associative. That is, for any quadruple of objects $w,x,y,z \in \mathrm{Ob}(\mathscr{C})$ and triple of morphisms $h \in \Hom(w,x)$, $g \in \Hom(x,y)$, and $f \in \Hom(y,z)$, we have $(f \circ g) \circ h = f \circ (g \circ h)$.
    \item Identity morphisms act as neutral elements for composition. That is, for any $x, y \in \mathrm{Ob}(\mathscr{C})$ and $f \in \Hom(x,y)$, we have $id_y \circ f = f = f \circ id_x$.
\end{itemize}
\end{definition}

\begin{example}
\label{ex:category_examples}
We present two important examples of categories, which will play a central role later on.

\begin{itemize}
    \item The \emph{category of sets}, denoted by $\texttt{Set}$. The objects are sets and the morphisms are functions between sets. Composition is the usual composition of functions.
    \item The \emph{category of vector spaces} over a field $\mathbb{K}$, denoted by $\texttt{Vect}$. The objects are vector spaces over $\mathbb{K}$ and the morphisms are $\mathbb{K}$-linear maps. Composition is the usual composition of linear functions.
\end{itemize}
\end{example}

\begin{definition}[Inverse morphism, isomorphism]
Let $\mathscr{C}$ be a category, and $f \in \Hom(x,y)$ any morphism in $\mathscr{C}$. A morphism $g \in \Hom(y,x)$ is called an \emph{inverse morphism} of $f$ if $g \circ f = id_x$ and $f \circ g = id_y$. We write $g = f^{-1}$ for the inverse of $f$. A morphism that has an inverse is called an \emph{isomorphism}.
\end{definition}

\begin{remark}
    Inverse morphisms, when they exist, are unique.
\end{remark}

\begin{samepage}
\begin{definition}[Terminal object, initial object]
Let $\mathscr{C}$ be a category.
\begin{enumerate}
    \item An object $t \in \mathrm{Ob}(\mathscr{C})$ is called a \emph{terminal object} if for every object $x \in \mathrm{Ob}(\mathscr{C})$, there exists a unique morphism from $x$ to $t$.
    \item An object $i \in \mathrm{Ob}(\mathscr{C})$ is called an \emph{initial object} if for every object $x \in \mathrm{Ob}(\mathscr{C})$, there exists a unique morphism from $i$ to $x$.
\end{enumerate}
\end{definition}
\end{samepage}

\begin{samepage}
\begin{example}
\leavevmode
\begin{itemize}
    \item In \texttt{Set}, any singleton set is a terminal object and the empty set is the unique initial object.
    \item In \texttt{Vect}, the zero vector space is a terminal object and the zero vector space is an initial object.
\end{itemize}
\end{example}
\end{samepage}

\begin{definition}[Product category]
Let $\mathscr{C}$ and $\mathscr{D}$ be categories. The \emph{product category} $\mathscr{C} \times \mathscr{D}$ is defined as follows:
\begin{enumerate}
    \item The collection of objects in $\mathscr{C} \times \mathscr{D}$ consists of ordered pairs $(c,d)$, where $c \in \mathrm{Ob}(\mathscr{C})$ and $d \in \mathrm{Ob}(\mathscr{D})$.
    \item For any $(x,x'),(y,y') \in \mathrm{Ob}(\mathscr{C} \times \mathscr{D})$, we define $\Hom\bigl((x,x'),(y,y')\bigr)$ as the set of ordered pairs $(f,a)$, where $f \in \Hom_{\mathscr{C}}(x,y)$ and $a \in \Hom_{\mathscr{D}}(x',y')$.
    \item For any $(f,a) \in \Hom\bigl((x,x'),(y,y')\bigr)$ and $(g,b) \in \Hom\bigl((y,y'),(z,z')\bigr)$, the composition in $\mathscr{C} \times \mathscr{D}$ is defined by: $(g,b) \circ (f,a) \coloneqq (g \circ f, b \circ a)$.
    \item The identity morphisms in the product category are the pairs of identity morphisms from $\mathscr{C}$ and $\mathscr{D}$. That is, $id_{(c,d)} \coloneqq (id_c, id_d)$ for every $(c,d) \in \mathrm{Ob}(\mathscr{C} \times \mathscr{D})$.
\end{enumerate}
\end{definition}

\begin{definition}[Functor]
Let $\mathscr{C}$ and $\mathscr{D}$ be categories. A \emph{functor} $F$ from $\mathscr{C}$ to $\mathscr{D}$, denoted by $F:\mathscr{C} \rightarrow \mathscr{D}$, consists of:
\begin{enumerate}
    \item A component-function of the classes of objects $F_{ob}: \mathrm{Ob}(\mathscr{C}) \rightarrow \mathrm{Ob}(\mathscr{D})$.
    \item For every $x, y \in \mathrm{Ob}(\mathscr{C})$, a component-function of sets of morphisms $F_{x,y}: \Hom_{\mathscr{C}}(x,y) \rightarrow \Hom_{\mathscr{D}}\bigl(F_{ob}(x), F_{ob}(y)\bigr)$.
\end{enumerate}
We usually suppress the subscripts and write $F(x) \coloneqq F_{ob}(x)$ and $F(f) \coloneqq F_{x,y}(f)$. We require that these assignments respect composition $\bigl(F(g \circ f) \coloneqq F(g) \circ F(f)\bigr)$ and identities $\bigl(F(id_x) = id_{F(x)}\bigr)$.
Given two composable functors, $F: \mathscr{C} \to \mathscr{D}$ and $G: \mathscr{D} \to \mathscr{E}$, their \emph{composition} $G \circ F: \mathscr{C} \to \mathscr{E}$ is again a functor, whose action on objects and morphisms is defined by:
\( (G \circ F)(x) \coloneqq G\bigl(F(x)\bigr) \quad \text{and} \quad (G \circ F)(f) \coloneqq G\bigl(F(f)\bigr). \)
\end{definition}

The \emph{identity functor} on $\mathscr{C}$, denoted $id_{\mathscr{C}}: \mathscr{C} \rightarrow \mathscr{C}$, is defined by $id_{\mathscr{C}}(x) \coloneqq x$ and $id_{\mathscr{C}}(f) \coloneqq f$.

\begin{definition}[Isomorphism of categories]
Let $\mathscr{C}$ and $\mathscr{D}$ be categories. A functor $F:\mathscr{C} \rightarrow \mathscr{D}$ is called an \emph{isomorphism of categories} if there exists a functor $G:\mathscr{D} \rightarrow \mathscr{C}$ such that $G \circ F = id_{\mathscr{C}}$ and $F \circ G = id_{\mathscr{D}}$. In this case, $\mathscr{C}$ and $\mathscr{D}$ are called isomorphic categories.
\end{definition}

\begin{definition}[Natural transformation and natural isomorphism]
Let $F, G: \mathscr{C} \rightarrow \mathscr{D}$ be two functors. A \emph{natural transformation} $\alpha$ from $F$ to $G$, denoted $\alpha: F \Rightarrow G$, is a family of morphisms in $\mathscr{D}$, indexed by the objects of $\mathscr{C}$, consisting of a morphism $\alpha_x: F(x) \rightarrow G(x)$ for each object $x \in \mathrm{Ob}(\mathscr{C})$, such that for every morphism $f: x \rightarrow y$ in $\mathscr{C}$, the following diagram commutes:
\begin{equation}
\begin{tikzcd}
F(x) \arrow{r}{F(f)} \arrow[d, "\alpha_x"'] & F(y) \arrow{d}{\alpha_y} \\
G(x) \arrow{r}{G(f)} & G(y)
\end{tikzcd}.
\end{equation}
That is, $\alpha_y \circ F(f) = G(f) \circ \alpha_x$. A natural transformation $\alpha$ such that $\alpha_x$ is an isomorphism for all $x$ is called a \emph{natural isomorphism}.
\end{definition}

\begin{example}[Isomorphism of the bidual]
\label{ex:double_dual}
A classical example of a natural isomorphism arises in the category $\texttt{FDVect}$ of finite-dimensional vector spaces \cite{leinster14}. Consider the following two functors from $\texttt{FDVect}$ to $\texttt{FDVect}$:
\begin{itemize}
    \item The \emph{identity functor}, $\mathrm{id}_{\texttt{FDVect}}$.
    \item The \emph{bidual functor}, $(-)^{**}$. It maps each vector space $V \in \texttt{FDVect}$ to its \emph{bidual space} $V^{**} \coloneqq \Hom\bigl(\Hom(V, \mathbb{K}), \mathbb{K}\bigr)$ and each linear map $f: V \to W$ to the transpose of its transpose, i.e., the map $f^{**}: V^{**} \to W^{**}$.
\end{itemize}
There exists a natural isomorphism $\epsilon: \mathrm{id}_{\texttt{FDVect}} \Rightarrow (-)^{**}$. For each vector space $V$, its component $\epsilon_V: V \to V^{**}$ is the evaluation map defined by:
\begin{equation}
\epsilon_V(v)(f) \coloneqq f(v), \quad \text{for all } v \in V \text{ and } f \in V^*.
\end{equation}
This map is an isomorphism.
\end{example}

\subsection{Monoidal Categories}
\label{subsec:monoidal_categories}
In order to describe algebraic notions such as the composition of operations in an abstract context, the basic structure of a category must be equipped with additional structure. Thus, the concept of a monoidal category is introduced, which endows a category with a notion of a tensor product and a unit object, governed by precise axioms of coherence \cite{maclane98}.

\begin{definition}[Monoidal category]
A \emph{monoidal category} $(\mathscr{C}, \otimes, u, \alpha, \lambda, \rho)$, or simply $(\mathscr{C}, \otimes, u)$, consists of the following data:
\begin{enumerate}
    \item A category $\mathscr{C}$.
    \item A functor $\otimes : \mathscr{C} \times \mathscr{C} \to \mathscr{C}$, called the \emph{tensor product}.
    \item An object $u \in \text{Ob}(\mathscr{C})$, called the \emph{unit object}.
    \item A natural isomorphism $\alpha$, called the \emph{associativity constraint}, with components
        \begin{equation}
         \alpha_{x,y,z} : (x \otimes y) \otimes z \xrightarrow{\cong} x \otimes (y \otimes z), \quad x,y,z \in \text{Ob}(\mathscr{C}).
        \end{equation}
    \item A natural isomorphism $\lambda$, called the \emph{left unit constraint}, with components
        \begin{equation}
         \lambda_x : u \otimes x \xrightarrow{\cong} x, \quad x \in \text{Ob}(\mathscr{C}).
        \end{equation}
    \item A natural isomorphism $\rho$, called the \emph{right unit constraint}, with components
        \begin{equation}
        \rho_x : x \otimes u \xrightarrow{\cong} x, \quad x \in \text{Ob}(\mathscr{C}).
        \end{equation}
\end{enumerate}
This data must satisfy the following axioms:
\begin{itemize}
    \item The \emph{Pentagon Axiom}: For every quadruple of objects $w, x, y, z$, the following diagram commutes:
\begin{equation}
\begin{tikzcd}
    ((w \otimes x) \otimes y) \otimes z \arrow[rr, "\alpha_{w \otimes x, y, z}"] \arrow[d, "\alpha_{w,x,y} \otimes \mathrm{id}_z"'] & & (w \otimes x) \otimes (y \otimes z) \arrow[d, "\alpha_{w,x,y\otimes z}"] \\
    (w \otimes (x \otimes y)) \otimes z \arrow[dr, "\alpha_{w, x \otimes y, z}"'] & & w \otimes (x \otimes (y \otimes z)) \\
    & w \otimes ((x \otimes y) \otimes z) \arrow[ur, "\mathrm{id}_w \otimes \alpha_{x,y,z}"'] &
\end{tikzcd}
    .
\end{equation}

    \item The \emph{Triangle Axiom}: For every pair of objects $x, y$, the following diagram commutes:
    \begin{equation}
     \begin{tikzcd}[column sep=huge]
        (x \otimes u) \otimes y \arrow[rr, "\alpha_{x,u,y}"] \arrow[dr, "\rho_x \otimes \mathrm{id}_y"'] & & x \otimes (u \otimes y) \arrow[dl, "\mathrm{id}_x \otimes \lambda_y"] \\
        & x \otimes y &
     \end{tikzcd}
     .
     \end{equation}
\end{itemize}
\end{definition}
The axioms of a monoidal category handle the re-parenthesizing of products (associativity), but not the reordering of factors. To coherently relate an object like $a \otimes b$ to $b \otimes a$, further structure is needed.

To formally define the notion of symmetry, we first introduce the \emph{swap functor} (or \emph{twist functor}) $\tau: \mathscr{C} \times \mathscr{C} \to \mathscr{C} \times \mathscr{C}$, defined on objects and morphisms as:
\begin{equation}
\tau(x, y) \coloneqq (y, x) \quad \text{and} \quad \tau(f, g) \coloneqq (g, f).
\end{equation}

\begin{definition}[Symmetric monoidal category]
A monoidal category $(\mathscr{C}, \otimes, u)$ is called \emph{braided} if it is equipped with a natural isomorphism
\begin{equation}
B_{x,y} : x \otimes y \xrightarrow{\cong} y \otimes x
\end{equation}
called the \emph{braiding}, which is natural in both $x$ and $y$. This braiding must satisfy two coherence conditions known as the \emph{Hexagon Axioms}. For every triple of objects $x, y, z$, the following two diagrams must commute:
\begin{equation}
\begin{tikzcd}[row sep=huge, column sep=huge]
(x \otimes y) \otimes z \arrow[r, "\alpha_{x,y,z}"] \arrow[d, "B_{x,y} \otimes \mathrm{id}_z"'] & x \otimes (y \otimes z) \arrow[r, "B_{x, y \otimes z}"] & (y \otimes z) \otimes x \arrow[d, "\alpha_{y,z,x}"] \\
(y \otimes x) \otimes z \arrow[r, "\alpha_{y,x,z}"] & y \otimes (x \otimes z) \arrow[r, "\mathrm{id}_y \otimes B_{x,z}"] & y \otimes (z \otimes x)
\end{tikzcd}
\end{equation}
\begin{equation}
\begin{tikzcd}[row sep=huge, column sep=huge]
x \otimes (y \otimes z) \arrow[r, "\alpha^{-1}_{x,y,z}"] \arrow[d, "\mathrm{id}_x \otimes B_{y,z}"'] & (x \otimes y) \otimes z \arrow[r, "B_{x \otimes y, z}"] & z \otimes (x \otimes y) \arrow[d, "\alpha^{-1}_{z,x,y}"] \\
x \otimes (z \otimes y) \arrow[r, "\alpha^{-1}_{x,z,y}"] & (x \otimes z) \otimes y \arrow[r, "B_{x,z} \otimes \mathrm{id}_y"] & (z \otimes x) \otimes y
\end{tikzcd}.
\end{equation}
A braided monoidal category is called \emph{symmetric} if its braiding is a \emph{symmetry}, meaning it satisfies the condition:
\begin{equation}
B_{y,x} \circ B_{x,y} = \mathrm{id}_{x \otimes y} \quad \forall x,y \in \mathrm{Ob}(\mathscr{C}).
\end{equation}
\end{definition}

\begin{remark}
It is worth noting that while the definition of a \emph{braided} category requires checking both hexagon axioms, the situation simplifies in the \emph{symmetric} case. The symmetry condition, $B_{y,x} \circ B_{x,y} = \mathrm{id}_{x \otimes y}$, implies that one of the two hexagon diagrams commutes if and only if the other one does as well \cite[Chap. XI]{maclane98}. Therefore, to verify that a symmetric braiding satisfies the coherence axioms, it is sufficient to check just one of the two hexagons.
\end{remark}

\begin{example}
Fundamental examples of symmetric monoidal categories, which will be central to this work, include:
\begin{itemize}
    \item The category \texttt{Set} with the Cartesian product $\times$ as the tensor product and any singleton set as the unit object. The braiding is given by the swap isomorphism.
    \item The category of vector spaces over a field $\mathbb{K}$, \texttt{Vect}, with the standard tensor product ($\otimes_{\mathbb{K}}$) and the field $\mathbb{K}$ itself as the unit object. The braiding is given by the swap isomorphism.
\end{itemize}
\end{example}

\begin{theorem}[Mac Lane's Coherence Theorem]
In a monoidal category, any diagram built from the associativity constraint, left and right unit constraint, identities, and the tensor product commutes. If the monoidal category is symmetric, any \emph{linear} diagram (where each variable appears at most once in each object) that also involves the braiding commutes.
\end{theorem}

\begin{proof}
See \cite[Chapter VII, Theorem 1]{maclane98}.
\end{proof}

Within a monoidal category, it is possible to define algebraic structures in a purely diagrammatic way, generalizing the classical notions of semigroup and monoid from sets to objects of an abstract category.

\begin{definition}[Semigroup object]
\label{def:semigroup_object}
Let $\mathscr{C}$ be a monoidal category. A \emph{semigroup object} in $\mathscr{C}$ is a pair $(O, m)$, where $O \in \mathrm{Ob}(\mathscr{C})$ and $m: O \otimes O \rightarrow O$ is a morphism (called \emph{multiplication}) that satisfies the associativity axiom, i.e., it makes the following diagram commute:
\begin{equation} 
\label{eq:monoid_associativity_diagram}
\begin{tikzcd}
(O \otimes O) \otimes O \arrow[r, "m \otimes \mathrm{id}_O"] \arrow[d, "\alpha_{O,O,O}"'] & O \otimes O \arrow[dd, "m"] \\
O \otimes (O \otimes O) \arrow[d, "\mathrm{id}_O \otimes m"'] & \\
O \otimes O \arrow[r, "m"'] & O
\end{tikzcd}
.
\end{equation}
\end{definition}

\begin{definition}[Monoid object]
\label{def:monoid_object}
A \emph{monoid object} in $\mathscr{C}$ is a triple $(O, m, I)$, where $(O, m)$ is a semigroup object and $I: u \rightarrow O$ is a morphism (called \emph{unit}) such that the following diagrams commute:
\begin{equation}
\label{eq:monoid_unit_diagrams}
\begin{tikzcd}
u \otimes O \arrow[rr, "I \otimes \mathrm{id}_O"] \arrow[dr, "\lambda_O"'] & & O \otimes O \arrow[dl, "m"] \\
& O &
\end{tikzcd}
\quad \text{and} \quad
\begin{tikzcd}
O \otimes u \arrow[rr, "\mathrm{id}_O \otimes I"] \arrow[dr, "\rho_O"'] & & O \otimes O \arrow[dl, "m"] \\
& O &
\end{tikzcd}
.
\end{equation}
\end{definition}

\begin{definition}[Commutative semigroup object, commutative monoid object]
\label{def:commutative_semigroup}
Let $\mathscr{C}$ be a {symmetric} monoidal category. A semigroup (or monoid) object $(O, m)$ in $\mathscr{C}$ is called \emph{commutative} if its multiplication $m$ makes the following diagram commute:
\begin{equation}
\begin{tikzcd}
O \otimes O \arrow[rr, "B_{O,O}"] \arrow[dr, "m"'] & & O \otimes O \arrow[dl, "m"] \\
& O &
\end{tikzcd}
\end{equation}
where $B_{O,O}$ is the component of the braiding on $(O,O)$. The axiom thus requires $m \circ B_{O,O} = m$.
\end{definition}

\begin{example}
The abstract definitions of semigroup and monoid objects (Definitions \ref{def:semigroup_object}-\ref{def:commutative_semigroup}) capture familiar structures in concrete categories:
\begin{itemize}
    \item A \emph{semigroup object} corresponds to a classical \emph{semigroup} in \texttt{Set}, and to an \emph{associative algebra} in \texttt{Vect}.
    \item A \emph{monoid object} corresponds to a classical \emph{monoid} in \texttt{Set}, and to a \emph{unital associative algebra} in \texttt{Vect}.
    \item A \emph{commutative semigroup object} corresponds to a \emph{commutative semigroup} in \texttt{Set}, and to a \emph{commutative (associative) algebra} in \texttt{Vect}.
    \item A \emph{commutative monoid object} corresponds to a \emph{commutative monoid} in \texttt{Set}, and to a \emph{unital commutative (associative) algebra} in \texttt{Vect}.
\end{itemize}
\end{example}

\begin{definition}[Morphism of semigroup objects and monoid objects]
\label{def:morphism_monoid_objects}
Let $\mathscr{C}$ be a monoidal category and let $(O,m)$ and $(P,n)$ be two semigroup objects in $\mathscr{C}$.
\begin{enumerate}
    \item A \emph{morphism of semigroup objects} is a morphism $f: O \rightarrow P$ in $\mathscr{C}$ that preserves multiplication, i.e., the following diagram commutes:
    \begin{equation}
    \begin{tikzcd}
    O \otimes O \arrow[r, "m"] \arrow[d, "f \otimes f"'] & O \arrow[d, "f"] \\
    P \otimes P \arrow[r, "n"] & P
    \end{tikzcd}
    .
    \end{equation}
    \item A \emph{morphism of monoid objects} between $(O, m, I)$ and $(P, n, J)$ is a morphism of semigroup objects that also preserves the unit, i.e., such that $f \circ I = J$.
\end{enumerate}
\end{definition}

\begin{definition}[Categories of algebraic objects]
Let $\mathscr{C}$ be a monoidal category. The notions of object and morphism from Definitions \ref{def:semigroup_object}, \ref{def:monoid_object}, and \ref{def:morphism_monoid_objects} allow for the definition of several categories:
\begin{enumerate}
    \item The \emph{category of semigroup objects} in $\mathscr{C}$, denoted by $\texttt{Semi}_{\mathscr{C}}$.
    \item The \emph{category of monoid objects} in $\mathscr{C}$, denoted by $\texttt{Mon}_{\mathscr{C}}$.
    \item Their commutative counterparts, denoted by $\texttt{CSemi}_{\mathscr{C}}$ and $\texttt{CMon}_{\mathscr{C}}$, in the case where $\mathscr{C}$ is symmetric.
\end{enumerate}
\end{definition}

\subsection{Representation Theory}

In this subsection, we recall some notions from representation theory. This branch of algebra studies algebraic structures, such as groups, by representing their elements as linear transformations of vector spaces. This approach allows one to reframe problems from group theory in terms of linear algebra, thereby providing a more tractable framework.

In particular, in this review, the representation theory of the symmetric group $\mathbb{S}_n$ is of central importance. It provides the necessary tools to define and understand the action of permutations on the operations of an operad, as will be discussed in Subsection \ref{sec: s operad}.

\subsubsection{Group Representations}

We begin by recalling the basic definitions and results concerning the representation theory of groups.

Throughout this paper, we work over a field $\mathbb{K}$ of characteristic 0, unless otherwise specified. All algebraic objects such as vector spaces, tensor products, and algebras are defined over this field.

\begin{definition}[Group representation]
\label{def:group_representation}
Let $G$ be a group. A \emph{left representation} of $G$ on a vector space $V$ is a group homomorphism:
\begin{equation}
    \rho: G \rightarrow \Aut(V),
\end{equation}
where $\Aut(V)$ is the group of linear automorphisms of $V$.
The space $V$ is called the \emph{representation space}, or simply a \emph{representation} of the group $G$, or a \emph{left $G$-module}.
The homomorphism $\rho$ defines a \emph{left linear action} by setting $g \cdot v \coloneqq \rho(g)(v)$. The homomorphism property $\rho(gg') = \rho(g)\rho(g')$ ensures compatibility with the group operation:
\begin{equation}
    (gg') \cdot v = g \cdot (g' \cdot v)
\end{equation}
for all $g, g' \in G$ and $v \in V$.
\smallskip
\\Alternatively, a \emph{right representation} of $G$ on $V$ is defined by a group anti-homomorphism
\begin{equation}
    \rho: G \rightarrow \Aut(V).
\end{equation}
The latter defines a \emph{right linear action} by setting $v \cdot g \coloneqq \rho(g)(v)$. The anti-homomorphism condition $\rho(gg') = \rho(g')\rho(g)$ ensures that:
\begin{equation}
    v \cdot (gg') = (v \cdot g) \cdot g'
\end{equation}
for all $g, g' \in G$ and $v \in V$. In this case, $V$ is called a \emph{right $G$-module}.

In both cases, the identity element $e \in G$ acts as the identity: $\rho(e) = \mathrm{id}_V$.
\end{definition}

In operad theory, particularly for the action of the symmetric group, the convention of right actions is prevalent. This choice is motivated by the manner in which permutations act on the composition of operations, as will be seen in Subsection \ref{sec: s operad}. For consistency with the standard literature in the field, this review adopts the convention of right actions; consequently, we will henceforth use ``representation'', ``action'', and ``$G$-module'' to mean their right-sided counterparts, unless otherwise specified.

\begin{definition}[Morphism of representations]
\label{def:morphism_of_representations}
Let $G$ be a group. A \emph{morphism} between two representations $(V, \rho_V)$ and $(W, \rho_W)$ of $G$ is a linear map $\varphi: V \rightarrow W$ such that
\begin{equation}
\label{eq:morphism_of_reps}
\varphi(v \cdot g) = \varphi(v) \cdot g
\end{equation}
for all $g \in G$ and $v \in V$. Such a map is also called a \emph{$G$-morphism} or a \emph{$G$-linear map}. The space of $G$-morphisms between the representations $V$ and $W$ is denoted by $\Hom_G(V, W)$.
\end{definition}

The property \eqref{eq:morphism_of_reps} is also called \emph{equivariance} of $\varphi$.

\begin{definition}[Category of representations]
\label{def:category_of_representations}
The representations of a group $G$ form a category, denoted by \texttt{Rep(G)}.
\begin{enumerate}
    \item The \emph{objects} are the representations of $G$.
    \item The \emph{morphisms} between two representations $(V, \rho_V)$ and $(W, \rho_W)$ are the $G$-morphisms $\varphi \in \Hom_G(V, W)$.
    \item The \emph{composition} of morphisms is the standard composition of linear maps.
    \item The \emph{identity morphism} on an object $(V, \rho)$ is the identity map $\mathrm{id}_V$.
\end{enumerate}
\end{definition}

\begin{samepage}
\begin{example}
\label{ex:representations}
Some fundamental examples of representations:
\begin{itemize}
    \item The \emph{trivial representation}, where $V = \mathbb{K}$ and the action of $G$ is defined by $v \cdot g = v$ for all $v \in V$ and $g \in G$.
    \item The \emph{regular representation}. Consider a set $\{e_g\}_{g \in G}$ indexed by the elements of the group $G$. The representation space $V$ is defined as the set of all linear combinations of these elements with coefficients in $\mathbb{K}$:
    \begin{equation}
    \label{eq:group_algebra_space}
    V = \mathbb{K}[G]:= \left\{ \sum_{g \in G} k_g e_g \mid k_g \in \mathbb{K} \text{, almost all zero} \right\}.
    \end{equation}
    This space is also known as the \emph{group algebra} of $G$. The action of an element $g' \in G$ on a vector in $V$ is given by
    \begin{equation}
        \left( \sum_{g \in G} k_g e_g \right) \cdot g' = \sum_{g \in G} k_g e_{gg'}.
    \end{equation}
    \item If $G = S_n$, the \emph{sign representation} is given by the one-dimensional vector space $V = \mathbb{K}$ with the action $v \cdot g = \mathrm{sign}(g)v$.
\end{itemize}
\end{example}
\end{samepage}

\subsubsection{Algebras}
\label{subsec:algebras}

The concepts related to algebras are well-known, so this subsection is only intended to fix the notation and terminology that will be used in what follows.

\begin{definition}[Associative algebra]
An \emph{associative algebra} is a vector space $A$ equipped with a bilinear and associative product $\mu: A \otimes A \to A$. If the algebra has a multiplicative identity element, called a \emph{unit}, the algebra is said to be \emph{unital}. A \emph{morphism} between (unital) associative algebras is a linear map that preserves the product (and the unit).
The category of (not necessarily unital) associative algebras over $\mathbb{K}$ is denoted by $\texttt{Alg}_{\mathbb{K}}$, while that of unital associative algebras is denoted by $\texttt{uAlg}_{\mathbb{K}}$.
\end{definition}

\begin{definition}[Commutative algebra]
An associative algebra $(A, \mu)$ is called \emph{commutative} if its product satisfies the relation $\mu(x,y) = \mu(y,x)$\footnote{For a bilinear map such as $\mu: A \otimes A \to A$, we often use the convenient notation $\mu(x,y)$ to denote the image of the tensor $x \otimes y$.} for all $x,y \in A$. A commutative algebra that is also unital is called a \emph{unital commutative algebra}. Morphisms between (unital) commutative algebras are the same as for associative algebras.
The category of (not necessarily unital) commutative associative algebras is denoted by $\texttt{CommAlg}_{\mathbb{K}}$, while that of unital commutative associative algebras is denoted by $\texttt{uCommAlg}_{\mathbb{K}}$.
\end{definition}

\begin{definition}[Lie algebra]
A \emph{Lie algebra} is a vector space $\mathfrak{g}$ equipped with a bilinear product, called the \emph{Lie bracket} $[-,-]: \mathfrak{g} \otimes \mathfrak{g} \to \mathfrak{g}$, which satisfies the following properties for all $x, y, z \in \mathfrak{g}$:
\begin{enumerate}
    \item \emph{Alternating:} $[x,x] = 0$.
    \item \emph{Jacobi identity:} $[x,[y,z]] + [y,[z,x]] + [z,[x,y]] = 0$.
\end{enumerate}
A \emph{morphism} between Lie algebras is a linear map that preserves the Lie bracket. The category of Lie algebras is denoted by $\texttt{LieAlg}_{\mathbb{K}}$.
\end{definition}

\begin{remark}
    Note that if the field $\mathbb{K}$ has a characteristic not equal to 2 (i.e., $\mathrm{char}(\mathbb{K}) \neq 2$), the alternating property is equivalent to \emph{anticommutativity}, i.e.:
    \begin{equation}
      [x,y] = -[y,x], \quad \forall x,y \in \mathfrak{g}.
    \end{equation}
\end{remark}

\begin{remark}[From Associative to Lie Algebras and back]
Associative algebras and Lie algebras are deeply related. Any associative algebra $(A, \mu)$ can be given the structure of a Lie algebra by defining the Lie bracket as the \emph{commutator}, $[x,y] \coloneqq \mu(x,y) - \mu(y,x)$. This process defines a functor from the category of associative algebras to the category of Lie algebras.

Conversely, there is a functor in the opposite direction which constructs, for any Lie algebra $\mathfrak{g}$, its \emph{universal enveloping algebra} $U(\mathfrak{g})$, which is an associative algebra. These two functors form an adjoint pair, where the universal enveloping algebra functor is left adjoint to the commutator functor.

This relationship can be elegantly captured within operad theory itself. There exists a morphism of operads from the operad $\texttt{Lie}$ to the operad $\texttt{As}$, which precisely encodes the transformation from an associative product to its commutator. While a full description of this adjoint pair requires a more general language like that of PROPs (see \cite{AppelLaredo2018}), the underlying connection is a fundamental example of how operads relate different algebraic theories.
\end{remark}

\subsubsection{Representations of Algebras}
\label{subsec:algebra_representations}

Drawing a parallel with group theory, the representation of an algebra offers a powerful method for analyzing its structure. This is achieved by reframing its abstract properties in the more concrete language of linear maps.

\begin{definition}[Representation of an algebra]
\label{def:algebra_rep}
Let $A$ be a (unital) associative algebra and let $V$ be a vector space. A \emph{representation} of $A$ on $V$ is a (unital) algebra morphism
\begin{equation}
\rho: A \longrightarrow \End(V),
\end{equation}
where $\End(V)$ is the algebra of linear endomorphisms of $V$, in which the product is given by composition and the unit is the identity map $\mathrm{id}_V$.
\end{definition}

\begin{definition}[Morphism of algebra representations]
\label{def:morphism_algebra_reps}
Let $(V, \rho_V)$ and $(W, \rho_W)$ be two representations of an algebra $A$. A \emph{morphism of representations} (or \emph{$A$-morphism}) is a linear map $\varphi: V \to W$ that commutes with the action of $A$, i.e., such that the following diagram commutes for all $a \in A$:
\begin{equation}
\begin{tikzcd}
V \arrow[r, "\rho_V(a)"] \arrow[d, "\varphi"'] & V \arrow[d, "\varphi"] \\
W \arrow[r, "\rho_W(a)"'] & W
\end{tikzcd}
,
\end{equation}
explicitly $\varphi \circ \rho_V(a) = \rho_W(a) \circ \varphi$. The vector space of such morphisms is denoted by $\Hom_A(V, W)$.
\end{definition}

\begin{definition}[Category of representations of an algebra]
\label{def:category_algebra_reps}
The representations of an algebra $A$ form a category, denoted by $A\texttt{-Mod}$.
\begin{enumerate}
    \item The \emph{objects} are the representations of $A$, i.e., the pairs $(V, \rho)$.
    \item The \emph{morphisms} between two objects $(V, \rho_V)$ and $(W, \rho_W)$ are the morphisms of representations $\varphi \in \Hom_A(V, W)$.
    \item The \emph{composition} is the standard composition of linear maps.
    \item The \emph{identity morphism} on an object $(V, \rho)$ is the identity map $\mathrm{id}_V$.
\end{enumerate}
\end{definition}

\subsubsection{Group Algebra}

The group algebra $\mathbb{K}[G]$ is an important concept that allows reformulating results from group representation theory in terms of representations of unital associative algebras.

\begin{definition}[Group algebra]
The \emph{group algebra} $\mathbb{K}[G]$ of a group $G$ over a field $\mathbb{K}$ is the vector space $\mathbb{K}[G]$ from \eqref{eq:group_algebra_space}, equipped with a bilinear multiplication defined on the basis vectors by $e_g \cdot e_{g'} = e_{gg'}.$
\end{definition}

The following proposition establishes a fundamental link between the representations of a group and those of its group algebra.

\begin{proposition}
\label{prop:rep_group_algebra}
The representations of the group $G$ and the representations of the group algebra $\mathbb{K}[G]$ are in one-to-one correspondence.
\end{proposition}

\begin{proof}
Given a group representation $\rho: G \rightarrow \Aut(V)$, one can define an algebra representation $\tilde{\rho}: \mathbb{K}[G] \rightarrow \End(V)$ by extending $\rho$ by linearity:
\begin{equation}
\tilde{\rho}\left(\sum_{g\in G} k_g e_g\right) \coloneqq \sum_{g\in G} k_g \rho(g).
\end{equation}
Using the fact that $\rho$ is a group homomorphism, it is easy to verify that $\tilde{\rho}$ is a morphism of unital associative algebras. Conversely, given an algebra representation $\tilde{\rho}: \mathbb{K}[G] \rightarrow \End(V)$, one can define a group representation $\rho: G \rightarrow \Aut(V)$ by restricting it to the basis vectors:
\begin{equation}
\rho(g) \coloneqq \tilde{\rho}(e_g).
\end{equation}
Again, it is immediate to verify that $\rho$ is a group homomorphism. It is also necessary to show that the image of $\rho$ is indeed contained in $\Aut(V)$, i.e., that $\tilde{\rho}(e_g)$ is invertible. Its inverse is given by $\tilde{\rho}(e_{g^{-1}})$, in fact:
\begin{equation}
\tilde{\rho}(e_g) \circ \tilde{\rho}(e_{g^{-1}}) = \tilde{\rho}(e_g e_{g^{-1}}) = \tilde{\rho}(e_{gg^{-1}}) = \tilde{\rho}(e_e) = \tilde{\rho}(1) = \mathrm{id}_V.
\end{equation}
Similarly, it can be shown that $\tilde{\rho}(e_{g^{-1}}) \circ \tilde{\rho}(e_g) = \mathrm{id}_V$.
\end{proof}

\subsubsection{Motivating the Theory}
\label{subsec:motivations}

The notions of algebra and representation discussed so far serve to motivate the introduction of operads. The central idea is to generalize classical representation theory to a context that can accommodate operations with an arbitrary number of inputs.

Classical representation theory studies algebras through their actions on vector spaces. A representation of a unital associative algebra $A$ is an algebra morphism $\rho: A \to \End(V)$ (Definition \ref{def:algebra_rep}). This means that we associate to each abstract element $a \in A$ a concrete operation, the endomorphism $\rho(a): V \to V$. In this sense, the algebra $A$ can be seen as a model that encodes a type of \emph{unary operation} (one input, one output) and its compositions \cite[Section 2.1]{vallette14}.

A natural question thus arises: how can we extend this powerful paradigm to describe operations with multiple inputs, i.e., \emph{multilinear operations}?
The archetypal structure for such a theory is no longer the single space $\End(V)$, but the entire collection of spaces of multilinear maps:
\begin{equation}
\End_V := \left\{ \Hom(V^{\otimes n}, V) \right\}_{n \ge 0}.
\end{equation}
This collection is endowed with a much richer algebraic structure than an algebra alone possesses. For example, it is possible to compose a $k$-ary operation with $k$ other operations of arbitrary arities, obtaining a new operation whose arity is the sum of the arities of the inserted operations.

The mathematical objects designed to capture exactly this rich structure are \emph{operads}. An operad $\mathcal{P}$ can be thought of as a collection of spaces $\mathcal{P}(n)$, whose elements are ``abstract $n$-ary operations'', equipped with composition laws that generalize the composition of multilinear functions, and a unit operation $1 \in \mathcal{P}(1)$ \cite[Section 2.2]{vallette14}.

In this framework, a \emph{representation of an operad} $\mathcal{P}$ on a vector space $V$ (also called a \emph{$\mathcal{P}$-algebra}) is a morphism of operads:
\begin{equation}
\Phi: \mathcal{P} \longrightarrow \End_V.
\end{equation}
This notion generalizes classical representation theory: operads are to multilinear operations what associative algebras are to unary operations.

\section{The Classical Definition of an Operad}

In this section, we present the classical definition of an operad, beginning with the non-symmetric case, which offers a more intuitive version of these structures. We  consider them from a categorical perspective, defining them as multicategories with a single object. After formalizing the notions of a multicategory and a morphism of multicategories, we proceed to explicitly define non-symmetric operads, highlighting the one-to-one correspondence with single-object multicategories. Subsequently, the focus shifts to operads defined in the category of vector spaces. We  introduce the notions of a non-symmetric operad on \texttt{Vect}, a morphism of operads, and a representation of an operad, analyzing in detail the endomorphism operad $\End_V$ as a canonical example. The section concludes with an analysis of two foundational examples of non-symmetric operads: $\mathsf{As}$ and $\mathsf{uAs}$. It will be shown that algebras over $\mathsf{As}$ are in one-to-one correspondence with associative $\mathbb{K}$-algebras, while algebras over $\mathsf{uAs}$ correspond to unital associative $\mathbb{K}$-algebras. Finally, we introduce symmetric operads, a generalization that incorporates an action of the permutation group. For this symmetric case, the definitions of a morphism and a representation are also provided. The discussion will then focus on the operads $\mathsf{Com}$ and $\mathsf{uCom}$, which model commutative and unital commutative algebras, respectively.

\subsection{Multicategories}

There are several approaches to the notion of an \emph{operad}, each emphasizing different aspects of the theory. One possibility is to define operads as collections of operations endowed with a composition law and a unit element, subject to appropriate relations. Alternatively, operads can be considered as special cases of \emph{multicategories} – structures where morphisms can have a source consisting of a tuple of objects.

In this section we follow the latter approach. This perspective offers a natural motivation for operadic composition and is particularly well-suited for representation via \emph{planar trees}. These graphical tools are used to model the compositional structure in a context without symmetries.

Let us begin with the definition of a multicategory.

\begin{definition}[Multicategory]
\label{def:multicategory}
A \emph{multicategory} $\mathscr{C}$ consists of the following data:
\begin{enumerate}
    \item A collection of objects, denoted by $\mathrm{Ob}(\mathscr{C})$;
    \item For each $n \in \mathbb{N}_0$ and $a_1, \ldots, a_n, a \in \mathrm{Ob}(\mathscr{C})$, a set $\Hom(a_1, \ldots, a_n; a)$ of morphisms, called \emph{morphisms from $(a_1, \ldots, a_n)$ to $a$}; if $\theta \in \Hom(a_1,\ldots,a_n;a)$, the string $(a_1,\ldots,a_n)$ is called the \emph{source} of $\theta$ and $a$ is called the \emph{target};

    \item For any choice of objects $a_i, a_{ij}, a \in \mathrm{Ob}(\mathscr{C}), \ i=1,\ldots,n, \ k_i \in \mathbb{N}, \ j=1,\ldots,k_i$, a composition map
    \begin{equation}
    (\theta; \theta_1, \ldots, \theta_n) \mapsto \theta \circ (\theta_1, \ldots, \theta_n), 
    \end{equation}
    where $\theta \in \Hom(a_1, \ldots, a_n; a)$ and $\theta_i \in \Hom(a_{i1}, \ldots, a_{ik_i}; a_i)$ for each $i$, yielding a composite morphism in $\Hom(a_{11}, \ldots, a_{nk_n}; a)$;

    \item For each object $a \in \mathrm{Ob}(\mathscr{C})$, a \emph{unit} element $1_a \in \Hom(a; a)$, such that the compositions satisfy the following laws:
    \begin{itemize}
        \item \emph{Associativity:} given $\theta \in \Hom(a_1,\ldots,a_n; a)$, $\theta_i \in \Hom(a_{i1},\ldots,a_{ik_i}; a_i)$, and $\theta_{ij} \in \Hom(a_{ij1},\ldots,a_{ijk_{ij}}; a_{ij})$, we have
                \begin{equation}
                \begin{split}
                    & \theta \circ \bigl(\theta_1 \circ (\theta_{11}, \ldots, \theta_{1k_1}), \ldots, \theta_n \circ (\theta_{n1}, \ldots, \theta_{nk_n})\bigr) = \\
                    & = \bigl(\theta \circ (\theta_1, \ldots, \theta_n)\bigr) \circ (\theta_{11}, \ldots, \theta_{nk_n}).
                \end{split}
                \end{equation}
        \item \emph{Unit:} for any $\theta \in \Hom(a_1,\ldots,a_n; a)$ the following identities hold
                \begin{equation}
                    \theta \circ (1_{a_1}, \ldots, 1_{a_n}) = \theta \qquad \text{and} \qquad
                    1_a \circ (\theta) = \theta.
                \end{equation}
    \end{itemize}
\end{enumerate}
\end{definition}

\begin{remark}[Foundational Aspects]
In the definition above, we assume that the morphisms between any given source and target objects form a \emph{set}. A multicategory with this property is typically called \emph{locally small}.

Regarding the objects, while $\mathrm{Ob}(\mathscr{C})$ can be a proper class in more general formulations, this work focuses on cases where the collection of objects is also a set. Such a multicategory is known as a \emph{small multicategory}. These conventions are adopted to streamline the presentation and avoid advanced set-theoretic considerations.
\end{remark}

\begin{remark}
In the definition of a multicategory, the sets \(\Hom(a_1, \ldots, a_n; a)\) can also include the case \(n=0\), i.e., morphisms with an empty domain. The composition
$\theta \circ (\theta_1, \ldots, \theta_n)$
is well-defined even when some of the operations \(\theta_i\) have arity zero, understood as morphisms \(\theta_i \in \Hom(\ \cdot \ ; a_i)\).
\end{remark}

\noindent It can be helpful to imagine the composition in terms of trees.

\begin{figure}[ht]
\centering
\resizebox{\textwidth}{!}{
\begin{tikzpicture}[every node/.style={font=\scriptsize}, thick]

\node (a) at (0,-6.5) {$a$};
\node (theta) at (0,-5) {$\theta$};
\draw (a) -- (theta);

\node (a1) at (-4.5,-3.2) {$a_1$};
\node (a2) at (-1.5,-3.2) {$a_2$};
\node (an) at (3,-3.2) {$a_n$};
\node at (1,-3.2) {$\cdots$};

\draw (theta) -- (a1);
\draw (theta) -- (a2);
\draw (theta) -- (an);

\node (theta1) at (-4.5,-1.5) {$\theta_1$};
\node (theta2) at (-1.5,-1.5) {$\theta_2$};
\node (thetan) at (3,-1.5) {$\theta_n$};

\draw (a1) -- (theta1);
\draw (a2) -- (theta2);
\draw (an) -- (thetan);

\node at (-6,0.4) {$a_{11}$};
\node at (-5.2,0.4) {$a_{12}$};
\node at (-4.5,0.4) {$\cdots$};
\node at (-3.8,0.4) {$a_{1k_1}$};

\draw (theta1) -- (-6,0);
\draw (theta1) -- (-5.2,0);
\draw (theta1) -- (-4.5,0);
\draw (theta1) -- (-3.8,0);

\node at (-2.7,0.4) {$a_{21}$};
\node at (-1.9,0.4) {$a_{22}$};
\node at (-1.2,0.4) {$\cdots$};
\node at (-0.5,0.4) {$a_{2k_2}$};

\draw (theta2) -- (-2.7,0);
\draw (theta2) -- (-1.9,0);
\draw (theta2) -- (-1.2,0);
\draw (theta2) -- (-0.5,0);

\node at (1.5,0.4) {$a_{n1}$};
\node at (2.2,0.4) {$a_{n2}$};
\node at (3,0.4) {$\cdots$};
\node at (3.8,0.4) {$a_{nk_n}$};

\draw (thetan) -- (1.5,0);
\draw (thetan) -- (2.2,0);
\draw (thetan) -- (3,0);
\draw (thetan) -- (3.8,0);

\draw[->, thick] (4.6,-3) -- (6.1,-3);

\begin{scope}[xshift=8cm, yshift=-3cm, grow'=up, level distance=1.2cm,
  sibling distance=0.7cm, every node/.style={font=\scriptsize},
  edge from parent/.style={draw, thick}]

\node (res) {$\theta \circ (\theta_1, \ldots,\theta_n)$}
  child {node {$a_{11}$}}
  child {node {$a_{12}$}}
  child {node {$\cdots$}}
  child {node {$a_{nk_n}$}};
  
\node (out) at (0,-1.3) {$a$};
\draw[thick] (out) -- (res);

\end{scope}
\end{tikzpicture}
}
\caption{Representation of the composition $\theta \circ (\theta_1, \ldots, \theta_n)$ in a multicategory.}
\label{fig:multicat-comp}
\end{figure}
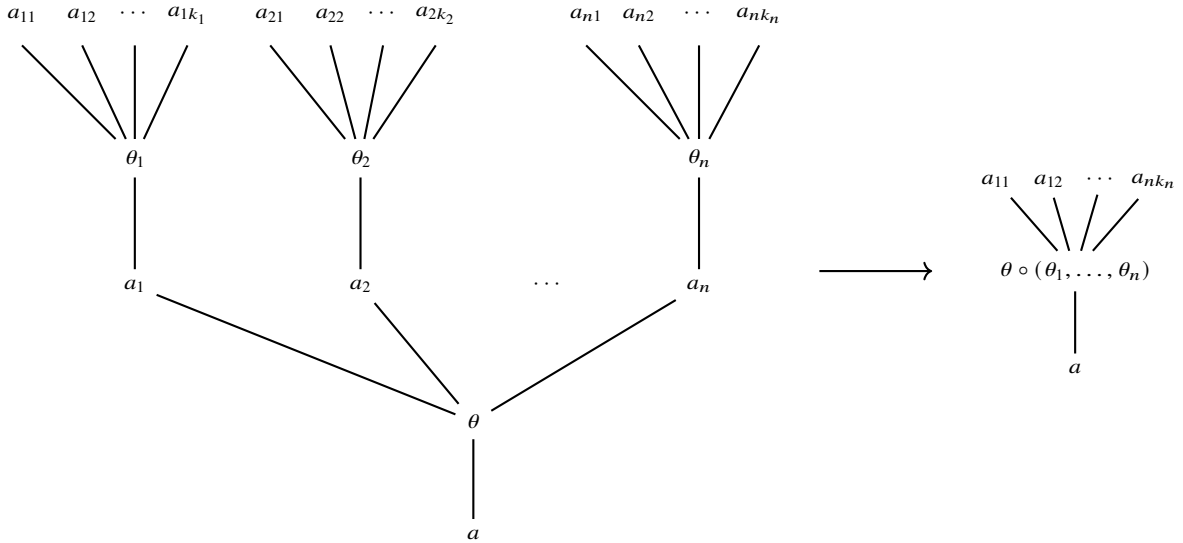
\noindent Figure \ref{fig:multicat-comp} represents the composition of an operation $\theta \in \Hom(a_1, \ldots, a_n; a)$ with a family of operations
$\theta_i \in \Hom(a_{i1}, \ldots, a_{ik_i}; a_i) \ \text{for } i = 1, \ldots, n,$ for which we obtain a new operation
$\theta \circ (\theta_1, \ldots, \theta_n) \in \Hom(a_{11}, \ldots, a_{nk_n}; a).$

The left side of the figure depicts the pre-composition process:
\begin{itemize}
    \item The node labeled with $a$ at the base of the tree represents the final target object of the composite operation.
    \item The node immediately above, labeled $\theta$, represents the main operation, whose source is the tuple of objects $(a_1, \ldots, a_n)$ and whose target is the object $a$.
    \item Each $a_i$ is an object of the multicategory and serves as the target of the operation $\theta_i$.
    \item Each operation $\theta_i$ has the tuple $(a_{i1}, \ldots, a_{ik_i})$ as its source.
    \item The nodes $a_{ij}$ represent the source objects for the initial layer of operations; they are the primitive inputs of the entire composition.
\end{itemize}

To the right of the arrow, the resulting tree corresponds to the composite operation:
\begin{itemize}
    \item The node $\theta \circ (\theta_1, \ldots, \theta_n)$ represents the composed morphism.
    \item The leaves, labeled $a_{11}, \ldots, a_{nk_n}$, form the source of this new morphism, created by the ordered concatenation (according to lexicographical order) of the source objects of each $\theta_i$.
    \item The composition has the object $a$ as its target, conventionally drawn as the root of the tree.
\end{itemize}

In summary, the objects $a_i$ and $a_{ij}$ specify the sources and targets for the operators (the morphisms $\theta$ and $\theta_i$), while $\theta$ and $\theta_i$ are the actual morphisms acting between these objects. The number of inputs of each morphism is called its \emph{arity}.

From this point forward, \emph{trees} will serve as a fundamental tool for the rigorous description and representation of composition in multicategories. While they have been employed so far as intuitive and schematic aids, we will later introduce a formal definition of an \emph{operadic tree} (in the combinatorial sense appropriate for our context), which will allow us to formalize composition operations precisely and to handle more advanced notions.

We now present the definition of a morphism of multicategories.
\begin{definition}
A \emph{morphism of multicategories} $F : \mathrm{Ob}(\mathscr{M}) \to \mathrm{Ob}(\mathscr{N})$ between two multicategories $\mathscr{M}$ and $\mathscr{N}$ consists of:
\begin{enumerate}
  \item A function between the sets of objects
        $ F_0 : \mathrm{Ob}(\mathscr{M}) \to \mathrm{Ob}(\mathscr{N});$
  \item A function
         $F : \Hom_{\mathscr{M}}(a_1, \ldots, a_n; a) \to \Hom_{\mathscr{N}}\bigl(F_0(a_1), \ldots, F_0(a_n); F_0(a)\bigr)$ for every $n\geq0$, $a_1, \ldots, a_n, a \in \mathrm{Ob}(\mathscr{M})$
\noindent such that:
  \item Composition is preserved:
  \begin{equation}
        F\bigl(\theta \circ (\theta_1, \ldots, \theta_n)\bigr) = F(\theta) \circ \bigl(F(\theta_1), \ldots, F(\theta_n)\bigr);
  \end{equation}
  \item Units are preserved:
  \begin{equation}
      F(\mathrm{1}_a) = \mathrm{1}_{F_0(a)}.
  \end{equation}
  \end{enumerate}
  The composition of morphisms of multicategories is the obvious one.
\end{definition}

\begin{remark}
Multicategories form a category, commonly denoted by \texttt{MultiCat}. The \emph{objects} are multicategories and the \emph{morphisms} are morphisms of multicategories.
\end{remark}

\subsection{Non-symmetric Operads}
\label{sec:non_symmetric_operads}

An operad can be defined concisely as a multicategory with only one object. By restricting a multicategory to a single object, the entire structure simplifies: all operations must necessarily act on this single type, and the composition laws are constrained accordingly. This ``monotypical'' perspective is fundamental to the theory and motivates the explicit, axiomatic definition that follows. A more detailed analysis of this correspondence is presented in Remark \ref{re:mult_1_obj}.

\begin{definition}[Non-symmetric operad]
\label{def: ns operad}
A \emph{non-symmetric operad} (also called an \emph{operad without symmetry}) \( \mathcal{P} \) consists of:

\begin{enumerate}
    \item A family \( \bigl(\mathcal{P}(n)\bigr)_{n \in \mathbb{N}_0} \) of sets, whose elements are called the \emph{abstract n-ary operations} of \( \mathcal{P} \).
    
    \item For each \( n, k_1, \dots, k_n \in \mathbb{N}_0\), a composition map
        \begin{equation}
        \begin{split}
        \gamma_{k_1,\dots ,k_n}\colon\;
        \mathcal{P}(n)\times\mathcal{P}(k_1)\times\cdots\times\mathcal{P}(k_n)
        &\;\longrightarrow\;
        \mathcal{P}(k_1+\cdots+k_n) \\[2pt] 
        (\theta;\,\theta_1,\dots,\theta_n)
        &\;\longmapsto\;
        \theta\circ(\theta_1,\dots,\theta_n)
        \label{eq:gamma}
        \end{split}.
        \end{equation}
    
    \item An element \( 1 \in \mathcal{P}(1) \), called the \emph{unit}.
\end{enumerate}
Moreover, for all $\theta \in \Hom(a_1,\ldots,a_n; a)$, $\theta_i \in \Hom(a_{i1},\ldots,a_{ik_i}; a_i)$, $\theta_{ij} \in \Hom(a_{ij1},\ldots,a_{ijk_{ij}}; a_{ij})$, the following \emph{associativity laws}
\begin{equation}
\begin{split}
    &\theta \circ \big( \theta_1 \circ (\theta_{1,1}, \dots, \theta_{1,k_1}), \dots, \theta_n \circ (\theta_{n,1}, \dots, \theta_{n,k_n}) \big) = \\
    & = \bigl(\theta \circ (\theta_1, \dots, \theta_n)\bigr) \circ (\theta_{1,1}, \dots, \theta_{1,k_1}, \dots, \theta_{n,1}, \dots, \theta_{n,k_n})
\label{eq:associativita-ns}
\end{split}
\end{equation}
and \emph{unit laws}
\begin{equation}
    \theta \circ (1, \dots, 1) = \theta = 1 \circ \theta
    \label{eq: leggi di unità ns}
\end{equation}
must be satisfied.
\end{definition}

\begin{remark}[Non-symmetric operads as single-object multicategories]
\label{re:mult_1_obj}
The definition of a non-symmetric operad coincides with the definition of a multicategory having only one object. This correspondence merits further elaboration.

Let $\mathscr{M}$ be a multicategory with a single object $a$. Given the uniqueness of this object, every morphism in the multicategory have the form $\theta \in \Hom(a, \ldots, a; a)$, where $n$ copies of $a$ appear in the source. We may define a collection of sets $\mathcal{P} = \{ \mathcal{P}(n) \}_{n \in \mathbb{N}_0}$ by setting:
\begin{equation}
    \mathcal{P}(n) \coloneqq \Hom(\underbrace{a, \ldots, a}_n; a).
\end{equation}
The composition law of the multicategory
\begin{equation}
    \gamma_{k_1, \ldots, k_n} : \Hom(a, \ldots, a; a) \times \prod_{i=1}^{n} \Hom(\underbrace{a, \ldots, a}_{k_i}; a) \longrightarrow \Hom(\underbrace{a, \ldots, a}_{k_1 + \cdots + k_n}; a)
\end{equation}
naturally induces a composition law for the collection $\mathcal{P}$:
\begin{equation}
    \gamma_{k_1, \dots, k_n} : \mathcal{P}(n) \times \mathcal{P}(k_1) \times \cdots \times \mathcal{P}(k_n) \longrightarrow \mathcal{P}(k_1 + \cdots + k_n).
\end{equation}
This is precisely the composition specified in Definition \ref{def: ns operad}. The unit $1_a \in \Hom(a; a)$ of the multicategory corresponds to the unit element $1 \in \mathcal{P}(1)$ of the operad. Consequently, the associativity and unit axioms for a multicategory specialize to the corresponding laws for an operad, requiring that iterated composition is well-defined regardless of how the operations are associated, and that the element $1 \in \mathcal{P}(1)$ acts as the compositional identity.

This correspondence is not merely a technical formality; it is conceptually significant. An operad describes a ``monotypical'' algebraic theory — one where all operations act on a single type of object. This restriction simplifies the multicategorical framework, as there is no need to distinguish between different object types. The lack of distinction between the source and target types of internal operations is what makes their representation via \emph{planar trees} so natural: since all internal edges of a tree correspond to the same object, they do not need to be labeled, allowing the graphical representation to focus purely on the compositional structure.

In summary:
\begin{itemize}
    \item Every non-symmetric operad can be viewed as a multicategory with a single object.
    \item Every single-object multicategory gives rise, by restricting to its set of morphism, to a non-symmetric operad.
    \item The correspondence is one-to-one and preserves all structures: composition, unit, arity.
\end{itemize}
\end{remark}

\begin{remark}
\label{rem:cat_nonsym_op}
The collection of all non-symmetric operads forms a \emph{full subcategory} of the category of multicategories, {\texttt{MultiCat}}.
To distinguish this category from its symmetric counterpart, which we will introduce in Section \ref{sec: s operad}, we will denote the category of non-symmetric operads by ${\texttt{NonSymOp}}$.
\end{remark}

\begin{remark}
    The definition of an operad can be generalized by requiring that the collections of morphisms $\mathcal{P}(n)$ are not merely sets, but objects in another category (e.g., modules over a commutative unital ring \( R \), vector spaces over a field \( \mathbb{K} \), or more generally objects of a symmetric monoidal category \( \mathscr{C} \)). This leads to the concept of an \emph{enriched operad}, a topic for which we refer the reader to \cite{Kelly1982}. A more detailed treatment of operads defined on a general symmetric monoidal category is provided in Section \ref{sec:operads-monoidal-category}. In such cases, the composition maps must be R-multilinear, $\mathbb{K}$-multilinear, or, more generally, morphisms in \( \mathscr{C} \), where the Cartesian product is replaced by the tensor product defined by the monoidal structure.
\end{remark}

\begin{remark}
In a (symmetric or non-symmetric) operad, the elements of $\mathcal{P}(0)$ are called \emph{constant operations}, as they represent operations that require no input but still produce an output. Formally, each $\theta \in \mathcal{P}(0)$ is a ``0-ary operation''; that is, a morphism $\theta \in \Hom(; a)$ in the corresponding multicategory, where the source is empty and the target is the single object $a$.

The decision to include or exclude such operations often depends on the specific theoretical framework. In some formulations, for instance, it is desirable for every operation to be constructed from at least one input, leading to the convention $\mathcal{P}(0) = \varnothing$. This is typical in contexts such as:
\begin{itemize}
    \item \emph{Theories based on rooted planar trees}, where each internal node is required to have at least one input.
    \item \emph{Algebraic models without units}, such as semigroups, where the absence of a neutral element is reflected by excluding constant operations.
\end{itemize}
Conversely, zero-arity operations are essential in other frameworks:
\begin{itemize}
    \item \emph{Algebraic theories with constants}, where structures like monoids include distinguished elements that must be modeled by operads with $\mathcal{P}(0) \neq \varnothing$.
    \item \emph{Topological or categorical operads}, where constant operations can represent fundamental structures like base points or units.
\end{itemize}
In the definition of an operad within a symmetric monoidal category $(\mathscr{C}, \otimes, u)$, it is therefore customary to set $\mathcal{P}(0) \coloneqq u$, where $u$ is the unit object. This convention ensures formal compatibility with composition. For any representation of the operad on an object $X \in \mathrm{Ob}(\mathscr{C})$, the component $\mathcal{P}(0)$ corresponds to morphisms of the type $u \to X$, i.e., to the choice of a ``constant'' element in $X$. In the category of vector spaces, for example, this unit object is the field $\mathbb{K}$, and a map $\mathbb{K} \to V$ is uniquely determined by the image of $1 \in \mathbb{K}$, which corresponds to selecting an element of $V$. This convention facilitates the natural treatment of units in algebras over operads.
\end{remark}

\subsection{Linear Operads}
\label{subsec:linear_operads}

While the definitions provided so far are general, the aim of this review is to study operads that encode algebraic structures built upon vector spaces, such as associative, commutative, and Lie algebras. In this context, the abstract operations of an operad must be interpreted as multilinear transformations.

For this reason, the focus will now shift to operads defined within the category of vector spaces. An operad in the symmetric monoidal category $(\texttt{Vect}, \otimes, \mathbb{K})$ is called a \emph{linear operad}. For clarity, we will now restate the definition of a non-symmetric operad in this specific setting.

\begin{definition}[Non-symmetric operad on \texttt{Vect}]
A non-symmetric operad $\mathcal{P}$ in the category of vector spaces consists of:
\begin{enumerate}
    \item A family of vector spaces \( \{ \mathcal{P}(n) \}_{n \in \mathbb{N}_0} \), whose elements are called \emph{abstract n-ary operations}.

    \item For each \( n, k_1, \dots, k_n \in \mathbb{N}_0\), a multilinear composition map
\begin{equation}
\begin{split}
    \gamma_{k_1, \dots, k_n} \colon 
    \mathcal{P}(n) \otimes \mathcal{P}(k_1) \otimes \cdots \otimes \mathcal{P}(k_n) 
    & \ \longrightarrow \ \mathcal{P}(k_1 + \cdots + k_n) \\ 
    \theta \otimes \theta_1 \otimes \cdots \otimes \theta_n 
    & \ \longmapsto \ \theta \circ (\theta_1 \otimes \cdots \otimes \theta_n).
\end{split}
\end{equation}

    \item An element \( 1 \in \mathcal{P}(1) \), called the \emph{unit}.
\end{enumerate}

Furthermore, the usual \emph{associativity} and \emph{unit laws} must be satisfied.

\end{definition}

We use the notation \( \theta \circ (\theta_1 \otimes \dots \otimes \theta_n) \) to emphasize that the composition is a multilinear map in the setting of the category \texttt{Vect}. When there is no ambiguity, we will omit the \( \otimes \) symbol and use the more common notation \( \theta \circ (\theta_1, \dots, \theta_n) \).

The primary significance of operads lies in their representations, which are also known as \emph{algebras over an operad}. Drawing a parallel with classical algebra, where a representation of an algebra $A$ is a structure-preserving map into the algebra of endomorphisms $\End(V)$, a representation of an operad $\mathcal{P}$ is conceived as a map that realizes its abstract operations as concrete multilinear operations on a vector space $V$.

The canonical structure that collects all such multilinear operations is the endomorphism operad $\End_V$, introduced in Example \ref{es: end}. Consequently, the notion of a $\mathcal{P}$-algebra is formalized precisely as a \emph{morphism of operads} from the abstract operad $\mathcal{P}$ to the concrete operad $\End_V$. To make this definition rigorous, we must therefore first establish what constitutes a morphism between operads.

\begin{definition}[Morphism of non-symmetric operads]
Let \( \mathcal{P} \) and \( \mathcal{Q} \) be two non-symmetric operads on \texttt{Vect}. A \emph{morphism of non-symmetric operads} is a morphism $\varphi: \mathcal{P \to \mathcal{Q}}$ of multicategories such that for all $n \in \mathbb{N}_0$, $\varphi=:\varphi_n:\mathcal{P}(n) \to \mathcal{Q}(n)$ is a linear map. Thus, \( \varphi : \mathcal{P} \to \mathcal{Q} \) is a family of linear maps
\begin{equation}
\varphi_n : \mathcal{P}(n) \longrightarrow \mathcal{Q}(n)
\end{equation}
such that:
\begin{enumerate}
    \item \emph{Compatibility with composition:} for any \( \theta \in \mathcal{P}(n) \), \( \theta_i \in \mathcal{P}(k_i) \),
    \begin{equation}
            \varphi_{k_1 + \cdots + k_n} \big( \theta \circ (\theta_1, \dots, \theta_n) \big)
    = \varphi_n(\theta) \circ \big( \varphi_{k_1}(\theta_1), \dots, \varphi_{k_n}(\theta_n) \big)
    \end{equation}
    or, if $\gamma$ is the composition of $\mathcal{P}$ and $\xi$ is the composition of $\mathcal{Q}$, we can write
    \begin{equation}
             \xi_{k1,\ldots,k_n} \circ \bigl(\varphi_n \circ (\varphi_{k_1}, \ldots, \varphi_{k_n})\bigr)= \varphi_{k_1+ \cdots +k_n} \circ \gamma_{k1,\ldots,k_n};
    \end{equation}

    \item \emph{Compatibility with the unit:}
    \begin{equation}
        \varphi_1(1_\mathcal{P}) = 1_\mathcal{Q}.
    \end{equation}
\end{enumerate}
A morphism of operads $\varphi$ is an \emph{isomorphism} if there exists an inverse morphism $\psi: \mathcal{Q} \to \mathcal{P}$ such that $\psi \circ \varphi = \mathrm{id}_{\mathcal{P}}$ and $\varphi \circ \psi = \mathrm{id}_{\mathcal{Q}}$. Equivalently, $\varphi$ is an isomorphism if and only if each of its components $\varphi_n$ is an isomorphism of vector spaces.

\label{def: morfismo di operad ns}
\end{definition}

\begin{example}[Endomorphism structure]
\label{es: end}
Let \( V \) be a vector space. We define the \emph{endomorphism structure} of \( V \), denoted \( \End_V \), as the collection of vector spaces
  \begin{equation}
  \End_V(n) \coloneqq \Hom(V^{\otimes n}, V), \quad  \text{for each} \ n \in \mathbb{N}_0, 
  \end{equation}
  where by convention \( V^{\otimes 0} \coloneqq \mathbb{K} \).
  This structure is equipped with the following operations:
\begin{itemize}
  \item \emph{Composition}: given \( f \in \End_V(n) \) and $f_i \in \End_V(k_i)$, $i=0,1,\ldots,n$, we define:
  \begin{equation}
        f \circ (f_1 \otimes \cdots \otimes f_n) \in \End_V(k_1 + \cdots + k_n)
  \end{equation}
  as the multilinear map obtained from the usual composition of multilinear applications, i.e.:
  \begin{equation}
  \begin{split}
  & \bigl(f \circ (f_1 \otimes \cdots \otimes f_n)\bigr)(v_1 \otimes \cdots \otimes v_{k_1 + \cdots + k_n}) = \\
  & = f\left( f_1(v_1 \otimes \cdots \otimes v_{k_1}) \otimes \cdots \otimes f_n(v_{k_1 + \cdots + k_{n-1} + 1} \otimes \cdots \otimes v_{k_1 + \cdots + k_n}) \right).
  \label{eq:def_comp_end}
  \end{split}
  \end{equation}
  \item \emph{Unit}: the {unit} is the identity map \( \mathrm{id}_V \in \End_V(1) \).
\end{itemize}
\end{example}

\begin{lemma}
Let $V$ be a vector space.
The {endomorphism structure}, with the operations defined above, is a well-defined non-symmetric operad.
\label{le: end}
\end{lemma}

\begin{proof}
We verify that $\End_V$ satisfies axioms \eqref{eq:associativita-ns} and \eqref{eq: leggi di unità ns}. An element $f \in \End_V(n)$ is a linear map $f: V^{\otimes n} \to V$. The composition $\theta \circ (\theta_1, \dots, \theta_n)$ is defined such that the output of each ``inner'' function $\theta_i$ becomes the $i$-th input of the ``outer'' function $\theta$. This corresponds precisely to the composition of multilinear functions: one first evaluates the internal functions, and their results are then used as arguments for the external one, following a compositional hierarchy.

The associativity law for a non-symmetric operad is given by the equation:
\begin{equation} \label{eq:operad_associativity}
\begin{split}
    &\theta \circ \bigl( \theta_1 \circ (\theta_{1,1}, \dots, \theta_{1,k_1}), \dots, \theta_n \circ (\theta_{n,1}, \dots, \theta_{n,k_n}) \bigr) = \\
    &= \bigl(\theta \circ (\theta_1, \dots, \theta_n)\bigr) \circ (\theta_{1,1}, \dots, \theta_{1,k_1}, \dots, \theta_{n,1}, \dots, \theta_{n,k_n}).
\end{split}
\end{equation}
To establish \eqref{eq:operad_associativity}, we consider an arbitrary tensor $v \in V^{\otimes K}$, where $K = \sum_{i=1}^n \sum_{j=1}^{k_i} k_{i,j}$ (and where $k_i$ is the arity of $\theta_i$ and $k_{i,j}$ is the arity of $\theta_{i,j}$). We split $v$ into blocks:
\begin{equation}
v = v_{1,1} \otimes \cdots \otimes v_{1,k_1} \otimes v_{2,1} \otimes \cdots \otimes v_{2,k_2} \otimes \cdots \otimes v_{n,1} \otimes \cdots \otimes v_{n,k_n}.
\end{equation}
We apply the definition of composition \eqref{eq:def_comp_end} to both sides of \eqref{eq:operad_associativity}.
For the left-hand side (LHS):
\begin{equation}
\begin{split}
& \left( \theta \circ \big( \theta_1 \circ (\theta_{1,1}, \dots, \theta_{1,k_1}), \dots, \theta_n \circ (\theta_{n,1}, \dots, \theta_{n,k_n}) \big) \right)(v)= \\
&= \theta \bigl( \left( \theta_1 \circ (\theta_{1,1}, \dots, \theta_{1,k_1}) \right)(v_{1,1} \otimes \cdots \otimes v_{1,k_1}) \otimes \cdots \\ 
& \qquad \otimes \left( \theta_n \circ (\theta_{n,1}, \dots, \theta_{n,k_n}) \right)(v_{n,1} \otimes \cdots \otimes v_{n,k_n}) \bigr) = \\
&= \theta \bigl( \theta_1(\theta_{1,1}(v_{1,1}) \otimes \cdots \otimes \theta_{1,k_1}(v_{1,k_1})) \otimes \cdots \\ 
& \qquad \otimes \theta_n(\theta_{n,1}(v_{n,1}) \otimes \cdots \otimes \theta_{n,k_n}(v_{n,k_n})) \bigr).
\end{split}
\end{equation}
For the right-hand side (RHS):
let $\Psi \coloneqq \theta \circ (\theta_1, \dots, \theta_n)$; this is an operation with arity $k_1 + \dots + k_n$.
\begin{equation}
\begin{split}
& \bigl( (\theta \circ (\theta_1, \dots, \theta_n)) \circ (\theta_{1,1}, \dots, \theta_{n,k_n}) \bigr)(v)= \\
& = \Psi \bigl( \theta_{1,1}(v_{1,1}) \otimes \cdots \otimes \theta_{1,k_1}(v_{1,k_1}) \otimes \cdots \bigr. \\
& \qquad \bigl. \otimes \, \theta_{n,1}(v_{n,1}) \otimes \cdots \otimes \theta_{n,k_n}(v_{n,k_n}) \bigr) = \\
& = \theta \bigl( \theta_1(\theta_{1,1}(v_{1,1}) \otimes \cdots \otimes \theta_{1,k_1}(v_{1,k_1})) \otimes \cdots \bigr. \\
& \qquad \bigl. \otimes \, \theta_n(\theta_{n,1}(v_{n,1}) \otimes \cdots \otimes \theta_{n,k_n}(v_{n,k_n})) \bigr).
\end{split}
\end{equation}
Since LHS = RHS for every $v \in V^{\otimes K}$, the equality is verified.
The unit laws require that:
\begin{equation}
\theta \circ (1, \dots, 1) = \theta
\label{eq:right_unit_law}
\end{equation}
and
\begin{equation}
    1 \circ \theta = \theta
    \label{eq:left_unit_law}
\end{equation}
where $1 = \mathrm{id}_V \in \End_V(1)$ is the identity map on $V$. Let $\theta \in \End_V (k)$.

\begin{itemize}
\item We first verify the right unit law \eqref{eq:right_unit_law}. For every $v_1 \otimes \cdots \otimes v_n \in V^{\otimes n}$:
\begin{equation}
\begin{split}
& \bigl(\theta \circ (\mathrm{id}_V, \dots, \mathrm{id}_V)\bigr)(v_1 \otimes \cdots \otimes v_n) = \\
& \qquad = \theta\bigl(\mathrm{id}_V(v_1) \otimes \cdots \otimes \mathrm{id}_V(v_n)\bigr) = \theta(v_1 \otimes \cdots \otimes v_n).
\end{split}
\end{equation}
This shows that composition with the unit on the right does not change the operation.

\item Next, we verify the left unit law \eqref{eq:left_unit_law}. For every $v_1 \otimes \cdots \otimes v_k \in V^{\otimes k}$:
\begin{equation}
(\mathrm{id}_V \circ \theta)(v_1 \otimes \cdots \otimes v_k) = \mathrm{id}_V\bigl(\theta(v_1 \otimes \cdots \otimes v_k)\bigr) = \theta(v_1 \otimes \cdots \otimes v_k).
\end{equation}
This shows that composition with the unit on the left does not change the operation.
\end{itemize}
\end{proof}

\begin{remark}
    Lemma \ref{le: end} shows that \( \End_V \) is a non-symmetric operad in the category \( \texttt{Vect} \). This operad is also called \emph{tautological}.
\end{remark}


\begin{definition}[Representation of a non-symmetric operad]
\label{def: P-algebra}
Let \(\mathcal{P}\) be a non-symmetric linear operad. A \emph{representation} of $\mathcal{P}$ (or $\mathcal{P}$-\emph{algebra}) is a pair $(V, \Phi)$, consisting of a vector space $V$ and a family of linear maps

\begin{equation}
\Phi = \{\Phi_n : \mathcal{P}(n) \to \Hom(V^{\otimes n}, V)\}_{n \in \mathbb{N}_0},
\end{equation}
such that, \(n, k_1, \dots, k_n \in \mathbb{N}_0\), for every $\theta \in \mathcal{P}(n), \theta_i \in \mathcal{P}(k_i),$
compatibility with composition is satisfied
\begin{equation}
    \Phi_{k_1 + \cdots + k_n} \big( \theta \circ (\theta_1, \dots, \theta_n) \big)
    = \Phi_n(\theta) \circ \big( \Phi_{k_1}(\theta_1) \otimes \cdots \otimes \Phi_{k_n}(\theta_n) \big).
\end{equation}
Here the right-hand side is the usual composition of multilinear maps:
\begin{equation}
    V^{\otimes (k_1 + \cdots + k_n)} \cong V^{\otimes k_1} \otimes \cdots \otimes V^{\otimes k_n}
    \xrightarrow{\Phi_{k_1}(\theta_1) \otimes \cdots \otimes \Phi_{k_n}(\theta_n)} V^{\otimes n} \xrightarrow{\Phi_n(\theta)} V.
\end{equation}
Furthermore, compatibility with the unit must be satisfied
\begin{equation}
    \Phi_1(1_{\mathcal{P}}) = \mathrm{id}_V.
\end{equation}
In other words, a representation is a morphism of non-symmetric operads
$\Phi : \mathcal{P} \to \End_V,$
where $\End_V(n)$ is the endomorphism operad of $V$.
\end{definition}

\begin{definition}[Morphism of $\mathcal{P}$-algebras]
\label{def: morfismo di P-algebre}
Let $\mathcal{P}$ be a non-symmetric operad on \texttt{Vect}.
Given two $\mathcal{P}$-algebras $\bigl( V,\Phi^V: \mathcal{P} \to \End_V \bigr)$ and $\bigl( W, \Phi^W: \mathcal{P} \to \End_W \bigr)$, a \emph{morphism of $\mathcal{P}$-algebras} is a $\mathbb{K}$-linear map $f: V \to W$ such that the following diagram commutes for every operation $\mu_n \in \mathcal{P}(n)$ and for every $n \in \mathbb{N}_0$:
\begin{equation}
    \begin{tikzcd}[column sep=large]
        V^{\otimes n} \arrow[d, "f^{\otimes n}"'] \arrow[r, "\Phi^V_n(\mu_n)"] & V \arrow[d, "f"] \\
        W^{\otimes n} \arrow[r, "\Phi^W_n(\mu_n)"] & W
    \end{tikzcd}
\end{equation}
    i.e., in explicit form:
    \begin{equation}
    \label{eq: condizione morf di P-algebre}
    f \circ \Phi^V_n(\mu_n) = \Phi^W_n(\mu_n) \circ f^{\otimes n}.
    \end{equation}
    This condition requires that the linear map $f$ is compatible with all $n$-ary operations induced by $\mathcal{P}$.
\end{definition}

\begin{definition}[Category of $\mathcal{P}$-algebras]
Let $\mathcal{P}$ be a non-symmetric operad. The \emph{category of $\mathcal{P}$-algebras}, denoted by $\texttt{Alg}(\mathcal{P})$, is defined as follows:

\begin{enumerate}
    \item The objects of $\texttt{Alg} (\mathcal{P})$ are the \emph{representations (or algebras) of the operad $\mathcal{P}$ on a vector space $V$}, in the sense of Definition \ref{def: P-algebra}.    
    \item  The morphisms of $\texttt{Alg} (\mathcal{P})$ are the morphisms of $\mathcal{P}$-algebras, in the sense of Definition \ref{def: morfismo di P-algebre}.

    \item The composition of morphisms in $\texttt{Alg}(\mathcal{P})$ is the usual composition of linear maps. {If $f: V \to W$ and $g: W \to Z$} are morphisms of $\mathcal{P}$-algebras, then $g \circ f: V \to Z$ is a morphism of $\mathcal{P}$-algebras, as can be easily verified.

    \item For each object $\bigl( V, \Phi^V : \mathcal{P} \to \End_V \bigr)$, the identity morphism is the identity $\mathrm{id}_V: V \to V$.
\end{enumerate}
\end{definition}

\subsubsection{The Operad $\mathsf{As}$}
\label{subsec: as}

One of the fundamental examples of a non-symmetric operad is the associative operad, commonly denoted by $\mathsf{As}$. It encodes, in operadic form, the structure of associative algebras: intuitively, an algebra over $\mathsf{As}$ is simply an associative algebra.
The interest in this operad stems from several factors: besides providing a first paradigmatic example of an operad, $\mathsf{As}$ plays a central role in several areas of mathematics, including algebraic topology \cite{may_geometry, boardmanvogt}, deformation theory \cite{Kontsevich2003, Gerstenhaber1964}, and the homotopy theory of algebras \cite{Stasheff1963, Kadeishvili1980}. Indeed, it serves as the foundation for defining more flexible structures, such as $A_\infty$-algebras, in which the associative law holds only up to homotopy.

In the following, we give an explicit definition of $\mathsf{As}$ as a non-symmetric operad on $\texttt{Vect}$, describing its components, composition operations, and identity, and we prove that $\mathsf{As}$-algebras are in one-to-one correspondence with associative $\mathbb{K}$-algebras.

We define $\mathsf{As}$ as a non-symmetric operad such that for each $n \geq 1$, the vector space $\mathsf{As}(n)$ is generated by a single element, which represents the unique $n$-ary operation obtained from an associative composition of $n$ variables. The composition laws of the operad reflect the associativity of the operations.

In detail:
\begin{definition}
We define $\mathsf{As}$ as follows:
\begin{enumerate}
  \item $\mathsf{As}(0) \coloneqq 0$;
  \item For each $n \in \mathbb{N}$, $\mathsf{As}(n) \coloneqq \mathbb{K} \{ \mu_n \} \footnote{Here, $\mathbb{K}\{\mu_n\}$ denotes the one-dimensional $\mathbb{K}$-vector space spanned by the formal basis element $\mu_n$.} \cong \mathbb{K}
    $;
  \item The composition
    \begin{equation}
      \mathsf{As}(n) \otimes \mathsf{As}(k_1) \otimes \cdots \otimes \mathsf{As}(k_n) \longrightarrow \mathsf{As}(k_1 + \cdots + k_n)
    \end{equation}
    is the canonical isomorphism
    \begin{equation}
    \underbrace{\mathbb{K} \otimes \cdots \otimes \mathbb{K}}_{n+1} \longrightarrow \mathbb{K}, \qquad a_1 \otimes \cdots\otimes a_n \mapsto a_1 \cdots a_n.
    \end{equation}
\end{enumerate}
\label{def:as}
\end{definition}

\begin{lemma}
The operad $\mathsf{As}$ of Definition \ref{def:as} is well-defined.
\end{lemma}

\begin{proof}
It follows directly from the definition that:
\begin{itemize}
    \item for each $n \in \mathbb{N}_0$, $\mathsf{As}(n)$ is a vector space;
    \item there exists an identity element $1 \in \mathsf{As}(1)$ (the unit of the field $\mathbb{K}$), such that
    \begin{equation}
    1 \circ \mu = 1 \cdot \mu = \mu, \qquad
    \mu \circ (1 \otimes \cdots \otimes 1) = \mu \cdot 1 \cdots 1 = \mu
    \end{equation}
    for every $\mu \in \mathsf{As}(n)$ and for every $n \in \mathbb{N}_0$;
    \item the composition maps are trivially associative, since the product in the field $\mathbb{K}$ is.
\end{itemize}
\end{proof}

The framework allows us to establish the main result of this section: an algebra over the operad $\mathsf{As}$ is precisely a vector space $V \in \texttt{Vect}$ equipped with an associative bilinear map $m_V: V \otimes V \to V$. In other words, algebras over $\mathsf{As}$ are exactly (not necessarily unital) associative algebras.

We now establish the isomorphism between the category of $\mathsf{As}$-algebras and the category of associative algebras. The proof consists of constructing the two functors that realize this isomorphism in the following propositions.

\begin{proposition}
\label{prop:functor_F_As}
The correspondence $F : \emph{\texttt{Alg}}(\mathsf{As}) \to \emph{\texttt{Alg}}_\mathbb{K}$, defined on objects by 
\begin{equation}
F(V,\Phi) \coloneqq \bigr(V, m_V \coloneqq \Phi_2(\mu_2)\bigl)
\end{equation}
and on morphisms by 
\begin{equation}
F(f) \coloneqq f,
\end{equation}
is a well-defined functor.
\end{proposition}

\begin{proof}
We need to verify that F is well-defined on objects and morphisms, and that it preserves identities and composition.
\begin{enumerate}
    \item \emph{Well-definedness on objects.}
    We have to show that for any $\mathsf{As}$-algebra $(V,\Phi)$, the pair $(V,m_V)$ is indeed an object in $\texttt{Alg}_{\mathbb{K}}$, i.e., a (not necessarily unital) associative $\mathbb{K}$-algebra. By construction, $m_V \in \End_V(2)=\Hom(V^{\otimes2},V)$, so it is a bilinear map. It remains to verify associativity.
    In the operad $\mathsf{As}$, the composition is associative, which implies the relation $\mu_2 \circ (\mu_2, \mu_1) = \mu_2 \circ (\mu_1, \mu_2)$, where $\mu_1=\mathrm{id}$. Since $\Phi$ is a morphism of operads, it preserves composition. Therefore:
    \begin{equation}
     \Phi_2(\mu_2) \circ \bigl(\Phi_2(\mu_2) \otimes \Phi_1(\mu_1)\bigr)
      \,=\,
      \Phi_2(\mu_2) \circ \bigl( \Phi_1(\mu_1) \otimes \Phi_2(\mu_2) \bigr).
    \end{equation}
    Since $\Phi_1(\mu_1)=\id_V$ and $\Phi_2(\mu_2)=m_V$, we obtain:
    \begin{equation}
      m_V \circ (m_V \otimes \id_V)
      =
      m_V \circ (\id_V \otimes m_V),
    \end{equation}
    which is precisely the associativity axiom for $m_V$. Thus, $(V,m_V)$ is an associative algebra and $F$ is well-defined on objects.
    \item \emph{Well-definedness on morphisms.}
    We must show that if $f$ is a morphism of $\mathsf{As}$-algebras, then $F(f)=f$ is a morphism of associative $\mathbb{K}$-algebras. The defining condition for a morphism of $\mathsf{As}$-algebras, from Equation \eqref{eq: condizione morf di P-algebre} for $n=2$, is:
    \begin{equation}
      f \circ \Phi^V_2(\mu_2)\;=\;\Phi^W_2(\mu_2)\circ f^{\otimes2}.
      \label{eq:cond_n_2}
    \end{equation}
    Since $m_V = \Phi^V_2(\mu_2)$ and $m_W = \Phi^W_2(\mu_2)$, \eqref{eq:cond_n_2} becomes $f \circ m_V = m_W \circ (f \otimes f)$, which is exactly the condition for $f$ to be a morphism of associative algebras. Thus, $F$ is well-defined on morphisms.
    \item \emph{Functorial properties.}
    Finally, we verify the functorial axioms. $F$ preserves identity morphisms, since for any object $(V,\Phi)$, $F(\id_V) = \id_V = \id_{F(V,\Phi)}$. It also preserves composition, since for any composable morphisms $f$ and $g$, we have $F(g \circ f) = g \circ f = F(g) \circ F(f)$.
\end{enumerate}

Since all conditions hold, $F$ is a well-defined functor.
\end{proof}

\begin{remark}[Structure of representations of the operad \(\mathsf{As}\)]
\label{re: as representation structure}
Let \((V, \Phi)\) be an \(\mathsf{As}\)--algebra, i.e., a representation of the operad \(\mathsf{As}\) on a vector space \(V\), with $
\Phi : \mathsf{As} \longrightarrow \End_V$ and in particular $\Phi_n : \mathsf{As}(n) \longrightarrow \End_V(n) = \Hom(V^{\otimes n}, V).$

The operad \(\mathsf{As}\) is such that for each \(n \ge 1\) there exists a unique element \(\mu_n \in \mathsf{As}(n)\), which represents the \emph{unique} fully associative \(n\)-ary operation. For this reason, every representation \(\Phi\) is completely determined by the collection of images \(\Phi_n(\mu_n)\). Let us now look at the structure of this representation in detail:

\begin{itemize}
 \item \emph{$n=0$}: \(\mathsf{As}(0) = 0\) is the zero vector space. It follows that the unique linear map \(\Phi_0 : 0 \to \Hom(\mathbb{K}, V)\) is the zero map. Thus, no constant operations (or zero-arity elements) are generated by \(\mathsf{As}(0)\).

 \item \emph{$n=1$}: There is a unique generator \(\mu_1 \in \mathsf{As}(1)\), which represents the identity (by definition of the operad's unit). The unit law imposes that
 \begin{equation}
 \Phi_1(\mu_1) = \id_V.
 \end{equation}

 \item \emph{$n=2$}: As discussed in Proposition \ref{prop:functor_F_As}, the unique binary generator \(\mu_2 \in \mathsf{As}(2)\) is sent via \(\Phi_2\) to a bilinear map
 \begin{equation}
 m_V := \Phi_2(\mu_2) : V \otimes V \longrightarrow V,
 \end{equation}
 which defines the composition law of the algebra.

 \item \emph{$n \geq 3$}: The \(n\)-ary operations \(\mu_n \in \mathsf{As}(n)\) are constructed as iterated compositions of \(\mu_2\):
\begin{equation}
\mu_n \circ (\mu_{k_1} \otimes \cdots \otimes \mu_{k_n}) = \mu_{k_1 + \cdots + k_n}.
\end{equation}
 Since \(\Phi\) is a morphism of operads, the following compatibility condition must hold:
 \begin{equation}
 \Phi_{k_1 + \cdots + k_n}(\mu_{k_1 + \cdots + k_n}) = \Phi_n(\mu_n) \circ \bigl( \Phi_{k_1}(\mu_{k_1}) \otimes \cdots \otimes \Phi_{k_n}(\mu_{k_n}) \bigr).
 \end{equation}
 Consequently, every map \(\Phi_n(\mu_n)\) is recursively determined by \(m_V = \Phi_2(\mu_2)\). For example:
 \begin{equation}
 \Phi_3(\mu_3) = m_V \circ (m_V \otimes \id_V) = m_V \circ (\id_V \otimes m_V),
 \end{equation}
 which is the associativity condition. Similarly, every \(\Phi_n(\mu_n)\) is an iterated composition of \(m_V\), according to the $\mathsf{As}$-algebra structure induced by the operad.

\end{itemize}
In conclusion, the structure of an \(\mathsf{As}\)-algebra on \(V\) is uniquely determined by the choice of the associative binary operation \(m_V = \Phi_2(\mu_2)\). This single binary operation is therefore sufficient to determine the entire morphism \(\Phi\).
\end{remark}

\begin{proposition}
\label{prop:functor_G_As}
The correspondence $G : \emph{\texttt{Alg}}_\mathbb{K} \to \emph{\texttt{Alg}}(\mathsf{As})$, defined by 
\begin{equation}
G(A, m_A) \coloneqq (A, \Psi)   
\end{equation}
and 
\begin{equation}
G(g) \coloneqq g,
\end{equation}
where $\Psi$ is the canonical $\mathsf{As}$-algebra structure induced by $m_A$, is a well-defined functor.
\end{proposition}

\begin{proof}
We must show that the functor $G$ is well-defined on both objects and morphisms, and that it respects identity and composition.
\begin{enumerate}
    \item \emph{Well-definedness on objects.}
    First, we show that for any associative algebra $(A, m_A)$, the pair $(A, \Psi)$ is an $\mathsf{As}$-algebra. The morphism of operads $\Psi = \{\Psi_n : \mathsf{As}(n) \to \End_A(n)\}_{n \geq 0}$ is defined as follows:
    \begin{itemize}
        \item $\Psi_0$ is the zero map, since $\mathsf{As}(0)=0$.
        \item $\Psi_1(\mu_1) \coloneqq \mathrm{id}_A$.
        \item $\Psi_2(\mu_2) \coloneqq m_A$.
        \item For $n \geq 3$, $\Psi_n(\mu_n)$ is defined recursively as the iterated composition of $m_A$:
        \begin{equation}
        \Psi_n(\mu_n) := \Psi_2(\mu_2) \circ \bigl(\Psi_{n-1}(\mu_{n-1}) \otimes \Psi_1(\mu_1) \bigr).
        \end{equation}
    \end{itemize}
    This definition is well-defined because $\mathsf{As}(0) = 0$ is the zero vector space, so $\Psi_0$ is necessarily the zero map. The unit element $\mu_1$ of the operad induces $\Psi_1(\mu_1) = \mathrm{id}_A$. The multiplication $m_A$ defines $\Psi_2(\mu_2)$. For $n \geq 3$, the recursive definition constructs the $n$-ary operations as iterated compositions of $m_A$. Compatibility with the composition of the operad $\mathsf{As}$ is guaranteed by the associativity of $m_A$. Therefore, $(A, \Psi)$ is a well-defined $\mathsf{As}$-algebra.
    \item \emph{Well-definedness on morphisms.}
    Next, we show that for any morphism of associative algebras $g : (A, m_A) \to (B, m_B)$, the map $G(g)=g$ is a morphism of $\mathsf{As}$-algebras. We must prove that for every $n \geq 0$, the following identity holds:
    \begin{equation}
    g \circ \Psi^A_n(\mu_n) = \Psi^B_n(\mu_n) \circ (g^{\otimes n}).
    \end{equation}
    We proceed by induction on $n$.
    
        \emph{Base case ($n=1$):} Since $\Psi^A_1(\mu_1) = \mathrm{id}_A$ and $\Psi^B_1(\mu_1) = \mathrm{id}_B$, the identity $g \circ \mathrm{id}_A = \mathrm{id}_B \circ g$ holds trivially.
        
        \smallskip
        \noindent \emph{Inductive step:} Assume the property holds for every $k < n$. From the definition of $\Psi^A_n(\mu_n)$, for every $a_1, \dots, a_n \in A$, we have:
        \begin{equation}
        \Psi^A_n(\mu_n)(a_1 \otimes \cdots \otimes a_n) =\big( m_A( \Psi^A_{n-1}(\mu_{n-1})(a_1 \otimes \cdots \otimes a_{n-1}), a_n ) \big).
        \end{equation}
        Applying $g$, we get
        \begin{equation}
        \begin{split}
        g \circ \Psi^A_n(\mu_n)(a_1 \otimes \cdots \otimes a_n)
        &= g \big( m_A( \Psi^A_{n-1}(\mu_{n-1})(a_1 \otimes \cdots \otimes a_{n-1}), a_n ) \big) \\
        &\overset{(1)}{=} m_B \big( g( \Psi^A_{n-1}(\mu_{n-1})(a_1 \otimes \cdots \otimes a_{n-1}) ), g(a_n) \big) \\
        &\overset{(2)}{=} m_B \big( \Psi^B_{n-1}(\mu_{n-1})(g(a_1) \otimes \cdots \otimes g(a_{n-1})), g(a_n) \big) \\
        &\overset{(3)}{=} \Psi^B_n(\mu_n)\bigl(g(a_1) \otimes \cdots \otimes g(a_n)\bigr) \\
        &= \Psi^B_n(\mu_n) \circ (g^{\otimes n})(a_1 \otimes \cdots \otimes a_n),
        \end{split}
        \end{equation}
        where:
        \begin{itemize}
            \item[(1)]\quad follows from the compatibility of $g$ with multiplication: \(g \circ m_A = m_B \circ (g \otimes g)\),
            \item[(2)]\quad follows from the inductive hypothesis on $\Psi_{n-1}$,
            \item[(3)]\quad follows from the definition of $\Psi^B_n(\mu_n)$.
        \end{itemize}
    This completes the induction, showing that $G$ is well-defined on morphisms.
    \item \emph{Functorial properties.}
    Finally, $G$ clearly preserves identities and composition because its action on morphisms is trivial: $G(\id_A) = \id_A = \id_{G(A, m_A)}$, and $G(g \circ f) = g \circ f = G(g) \circ G(f)$.
\end{enumerate}
Therefore, $G$ is a well-defined functor.
\end{proof}

\begin{theorem}
\label{th:as_isomorphism}
The category of $\mathsf{As}$‑algebras is isomorphic to the category of (not necessarily unital) associative algebras. This isomorphism of categories is established by the functors introduced in Propositions~\ref{prop:functor_F_As} and~\ref{prop:functor_G_As}.
\end{theorem}

\begin{proof}
Let
\(
F : \texttt{Alg}(\mathsf{As}) \longrightarrow \texttt{Alg}_\mathbb{K},
\
G : \texttt{Alg}_\mathbb{K} \longrightarrow \texttt{Alg}(\mathsf{As})
\)
be as in the preceding Propositions. We prove that
\(
F \circ G = \mathrm{Id}_{\texttt{Alg}_\mathbb{K}}
\)
and
\(
G \circ F = \mathrm{Id}_{\texttt{Alg}(\mathsf{As})}.
\)

\smallskip
\noindent\textit{(i) \ $F \circ G = \mathrm{Id}_{\texttt{Alg}_\mathbb{K}}$.}
\begin{itemize}
\item \emph{Objects.}
For $(A,m_A)\in\texttt{Alg}_\mathbb{K}$, \; $G$ yields $(A,\Psi)$ with $\Psi_2(\mu_2)=m_A$; and applying $F$ yields
\(
F\bigl(G(A,m_A)\bigr) = \bigl(A,\Psi_2(\mu_2)\bigr) = (A,m_A).
\)
\item \emph{Morphisms.}
For $g:(A,m_A)\to(B,m_B)$ in $\texttt{Alg}_\mathbb{K}$, we have $F\bigl(G(g)\bigr)=g$ by definition of $F$ and $G$.
\end{itemize}
Thus $F\circ G$ coincides with the identity of $\texttt{Alg}_\mathbb{K}$.

\smallskip
\noindent\textit{(ii) \ $G \circ F = \mathrm{Id}_{\texttt{Alg}(\mathsf{As})}$.}
\begin{itemize}
\item \emph{Objects.}
For $(V,\Phi)\in\texttt{Alg}(\mathsf{As})$, \; $F$ yields $(V,m_V)$ with $m_V=\Phi_2(\mu_2)$; and applying $G$ yields $(V,\widetilde{\Phi})$, where each
\(
\widetilde{\Phi}_n(\mu_n)
\)
is defined recursively from $m_V$, as in Proposition \ref{prop:functor_G_As}.
Since $\Phi$ is a morphism of operads and $m_V=\Phi_2(\mu_2)$, it follows directly from compatibility with composition that
\(
\widetilde{\Phi}_n(\mu_n)=\Phi_n(\mu_n)
\)
for every $n\ge 1$, so $G\bigl(F(V,\Phi)\bigr)=(V,\Phi)$.
\item \emph{Morphisms.}
For $f:(V,\Phi^V)\to(W,\Phi^W)$ in $\texttt{Alg}(\mathsf{As})$, we have $G\bigl(F(f)\bigr)=f$ (as shown in Proposition~\ref{prop:functor_G_As}).
\end{itemize}
Therefore $G\circ F$ coincides with the identity of $\texttt{Alg}(\mathsf{As})$.

\smallskip
$F$ and $G$ are inverses of each other, so they establish an isomorphism of categories:
$\texttt{Alg}(\mathsf{As}) \;\cong\; \texttt{Alg}_\mathbb{K}.$

\end{proof}

\subsubsection{The Operad \(\mathsf{uAs}\)}

A natural extension of the operad \(\mathsf{As}\) is the operad $\mathsf{uAs}$, which models associative algebras \emph{equipped with a unit element}. Analogously to \(\mathsf{As}\), the interest in \(\mathsf{uAs}\) lies in the fact that it provides an operadic description of an important category of algebraic structures, namely unital associative algebras.

\begin{remark}
\label{re: as0}
When modeling \emph{unital algebras} through operadic representations, the definition of the component $\mathcal{P}(0)$ plays a crucial role.
We recall that for the endomorphism operad $\End_V(0) = \Hom(\mathbb{K}, V)$. Each such map is uniquely determined by the image of $1 \in \mathbb{K}$. Consequently, $\Hom(\mathbb{K}, V)$ can be canonically identified with $V$ (each $v \in V$ corresponds to the map $\lambda \mapsto \lambda v$). To ensure that the representations of $\mathcal{P}$ produce algebras with a unit, it is sufficient to impose that the operad $\mathcal{P}$ contains a ``unit-generating'' 0-ary operation. This is accomplished by setting $\mathcal{P}(0) \coloneqq \mathbb{K}.$ Under this definition, a morphism of operads $\Phi: \mathcal{P} \rightarrow \End_V$ induces a specific map for $n=0$:
\begin{equation}
\Phi_0 : \mathcal{P}(0) \longrightarrow \End_V(0)
\end{equation}
i.e.
\begin{equation}
\Phi_0 : \mathbb{K} \longrightarrow \Hom(\mathbb{K}, V).
\end{equation}
Since $\Phi_0$ is a linear map, the image of the scalar unit $1 \in \mathbb{K}$ is a well-defined element in $\Hom(\mathbb{K}, V)$. Identifying $\Hom(\mathbb{K}, V)$ with $V$, the image $\Phi_0(1)$ corresponds to an element $e \in V$.
This element $e$ is precisely the unit element of the induced algebra structure on $V$. For a binary operation $\mu_2 \in \mathcal{P}(2)$ (which is mapped to a product in $V$, denoted by $\cdot$), the properties of the unit (such as $e \cdot x = x$) derive from the compatibility of the operad morphism with composition.
In summary, setting $\mathcal{P}(0) \coloneqq \mathbb{K}$ implicitly endows every $\mathcal{P}$-algebra with a unit element. This is the standard approach used in operad theory to model algebraic structures that possess a unit element.
\end{remark}

We can explicitly construct \(\mathsf{uAs}\) as a non-symmetric operad that extends \(\mathsf{As}\) by including the operation of arity zero that represents the unit.

More precisely:
\begin{definition}
    We define \(\mathsf{uAs}\) as follows:
\begin{enumerate}
  \item \(\mathsf{uAs}(0) \coloneqq \mathbb{K} \{ \eta \} \cong \mathbb{K}\);
  \item For each \(n \geq 1\), we set \(\mathsf{uAs}(n) \coloneqq \mathbb{K} \{ \mu_n \} \cong \mathbb{K}\);
  \item The compositions
    \begin{equation}
    \mathsf{uAs}(n) \otimes \mathsf{uAs}(k_1) \otimes \cdots \otimes \mathsf{uAs}(k_n) \longrightarrow \mathsf{uAs}(k_1 + \cdots + k_n)
    \end{equation}
    are given (fixing the normalization \(\mu_n = 1\) and \(\eta = 1\)) by the usual product of \(n+1\) elements in \(\mathbb{K}\), as in Definition \ref{def:as}.
\end{enumerate}
\label{def: uas}
\end{definition}

\begin{lemma}
\label{le: uas ben definito}
    The operad $\mathsf{uAs}$ of Definition \ref{def: uas} is well-defined.
\end{lemma}

\begin{proof}
It follows directly from the definition that:
  \begin{itemize}
    \item for each \(n \in \mathbb{N}_0\), \(\mathsf{uAs}(n)\) is a vector space by construction;
    \item there exists an identity element \(1 \in \mathsf{uAs}(1)\) such that, for any \(\mu \in \mathsf{uAs}(n)\)
    \begin{equation}
    1 \circ \mu = \mu, \qquad \mu \circ (1 \otimes \ldots \otimes 1) = \mu;
    \end{equation}
    \item the composition maps are associative, since the product in the field \(\mathbb{K}\) is.
\end{itemize}  
\end{proof}

\begin{remark}
\label{rem:uAs_remark}
The operad $\mathsf{uAs}$ can be viewed as a ``unitalization'' of $\mathsf{As}$, obtained by formally adding a nullary operation compatible with the binary multiplication.
\end{remark}

The previous observation can be generalized into a formal categorical framework.

We begin by defining two distinct classes of operads based on their spaces of nullary operations.

\begin{definition}[Unitary and non-unitary operads]
A non-symmetric linear operad $\mathcal{P}$ is said to be \emph{non-unitary} if its space of nullary operations is the zero vector space, i.e., $\mathcal{P}(0) = 0$. It is called \emph{unitary} if $\mathcal{P}(0) = \mathbb{K}$.\footnote{It is worth noting that the terminology in this area is not entirely standardized. Some authors use the term ``operad'' to refer exclusively to what we call unitary operads, while others use ``augmented operad'' for the unitary case and reserve ``operad'' for the non-unitary one. The terminology can be subtle; for instance, May in \cite{MayZhangZou2025} uses the term ``unital'', whereas we follow Fresse et al. \cite{FresseTurchinWillwacher18} who use ``unitary''.}
These two classes form distinct categories: the category of non-symmetric unitary operads, denoted by $\texttt{NonSymOp}_+$, and the category of non-symmetric non-unitary operads, denoted by $\texttt{NonSymOp}_\varnothing$. The morphisms in $\texttt{NonSymOp}_+$ are required to preserve the distinguished nullary operation\footnote{The generator of the one-dimensional vector space $\mathcal{P}(0) \cong \mathbb{K}$.}.
\end{definition}

There is a canonical procedure to construct a unitary operad from a non-unitary one. This process defines a standard \emph{unitalization functor}, denoted by $(-)_+$:
\begin{equation}
    (-)_+: \texttt{NonSymOp}_\varnothing \longrightarrow \texttt{NonSymOp}_+. 
\end{equation}
Given an operad $\mathcal{P} \in \texttt{NonSymOp}_\varnothing$, the unitary operad $\mathcal{P}_+$ is defined by setting $\mathcal{P}_+(0) \coloneqq \mathbb{K}$ and $\mathcal{P}_+(n) \coloneqq \mathcal{P}(n)$ for all $n > 0$, with a natural extension of the composition laws. This construction, and its property of being a left adjoint to the forgetful functor from $\texttt{NonSymOp}_+$ to $\texttt{NonSymOp}_\varnothing$, is a well-established result in the literature, as discussed for instance in \cite{FresseTurchinWillwacher18}.

\begin{example}[Unitalization of $\mathsf{As}$]
Applying the unitalization functor to the operad $\mathsf{As} \in \texttt{NonSymOp}_\varnothing$ yields a new operad, $(\mathsf{As})_+$, in the category $\texttt{NonSymOp}_+$. Its spaces of operations are:
\begin{equation}
(\mathsf{As})_+(n) = 
\begin{cases}
    \mathbb{K} & \text{if } n=0, \\
    \mathsf{As}(n) \cong \mathbb{K} & \text{if } n > 0.
\end{cases}
\end{equation}
This resulting operad, where every space of operations is one-dimensional, is precisely the operad $\mathsf{uAs}$ (see Definition \ref{def: uas}). This functorial construction therefore provides a rigorous foundation for the intuitive notion of ``unitalization'' mentioned in Remark \ref{rem:uAs_remark}, formalizing the process of turning the theory of non-unital associative algebras into the theory of unital ones.
\end{example}

An algebra over \(\mathsf{uAs}\) is precisely an object \(V \in \texttt{Vect}\) equipped with:
\begin{itemize}
  \item a bilinear map \(m_V : V \otimes V \to V\);
  \item a unit element \(1_V \in V\),
\end{itemize}
such that the usual relations hold:
\begin{equation}
m_V\bigl(m_V(x, y), z\bigr) = m_V\bigl(x, m_V(y, z)\bigr) \quad \text{and} \quad m_V(1_V, x) = x = m_V(x, 1_V).
\end{equation}
In other words, algebras over \(\mathsf{uAs}\) are exactly the \emph{unital} associative \(\mathbb{K}\)-algebras, as the following Theorem states.

\begin{theorem}
\label{th:uas}
The category of $\mathsf{uAs}$-algebras is isomorphic, as a category, to the category of unital associative $\mathbb{K}$-algebras. The isomorphism of categories is established by two functors constructed as in Propositions \ref{prop:functor_F_As} and \ref{prop:functor_G_As}.
\end{theorem}

\begin{proof}
A rigorous proof can be done by replicating \emph{mutatis mutandis} the proof of Theorem \ref{th:as_isomorphism}; for completeness, however, we briefly outline the main arguments.\\
The isomorphism is realized by two functors
\begin{equation}
F: \texttt{Alg}(\mathsf{uAs}) \longrightarrow \texttt{Alg}_{\mathbb{K}}, \qquad
G: \texttt{Alg}_{\mathbb{K}} \longrightarrow \texttt{Alg}(\mathsf{uAs}),
\end{equation}
defined as in Propositions \ref{prop:functor_F_As} and \ref{prop:functor_G_As}:
\begin{itemize}
  \item Given an object $(V,\Phi^V)$ in $\mathrm{Alg}(\mathsf{uAs})$, we set
  \begin{equation}
  F\bigl((V,\Phi^V)\bigr) \coloneqq \left(V, m \coloneqq \Phi^V_2(\mu_2),\ \mathbf{1} \coloneqq \Phi^V_0(\mu_0)\right);
  \end{equation}
  \item Given $(V, m, \mathbf{1}) \in \texttt{Alg}_{\mathbb{K}}$, we define $G\bigr((V, m, \mathbf{1})\bigl) \coloneqq (V, \Psi^V)$, where $\Psi^V: \mathsf{uAs} \to \End_V$ is the functor such that
  \begin{equation}
  \Psi^V(\mu_n)(v_1, \dots, v_n) \coloneqq v_1 \cdots v_n \quad \text{(iterated product)}, \quad \Psi^V(\mu_0)(-) \coloneqq \mathbf{1}.
  \end{equation}
\end{itemize}
As in the case of $\mathsf{As}$, the structure of $\mathsf{uAs}$ ensures that every $\mathsf{uAs}$-algebra $\Phi^V$ determines a bilinear multiplication $m \coloneqq \Phi^V_2(\mu_2)$ on $V$, which is associative by compatibility with composition, as well as an element $\mathbf{1} \coloneqq \Phi^V_0(\mu_0) \in V$, which is neutral for multiplication.

In particular, the relation 
\begin{equation}
\mu_2 \circ (\mu_0 \otimes \mathrm{id}) = \mu_1 = \mu_2 \circ (\mathrm{id} \otimes \mu_0)
\end{equation}
represents the property of unitality. Applying $\Phi^V$ yields:
\begin{equation}
\Phi^V(\mu_2) \circ \bigl(\Phi^V(\mu_0) \otimes \mathrm{id}_V\bigr) = \Phi^V(\mu_1) = \mathrm{id}_V = \Phi^V(\mu_2) \circ \bigl(\mathrm{id}_V \otimes \Phi^V(\mu_0)\bigr).
\end{equation}
Explicitly, for every $v \in V$:
\begin{equation}
m(\mathbf{1}, v) = v = m(v, \mathbf{1}).
\end{equation}
Conversely, if $(V, m, \mathbf{1})$ is a unital associative $\mathbb{K}$-algebra, then $\bigl( V, \Psi^V: \mathsf{uAs} \to \End_V \bigr)$ is a $\mathsf{uAs}$-algebra. Indeed, all identities between compositions in the operad $\mathsf{uAs}$ are satisfied as a consequence of the associativity and the unit law of $m$ and $\mathbf{1}$.

Finally, it is immediate to verify that the functors $F$ and $G$ are inverses of each other.
\end{proof}

\subsection{Visual Representation of Operadic Operations via Trees - Symmetric Case} \label{subsec: rappr operad non simm}

The operations of an operad and their compositions can be visualized using special graphs called \emph{operadic trees}. To define these structures rigorously, we  recall the fundamental notions of graph theory, based on the terminology from \cite{diestel17}. Subsequently, we  adapt these classical definitions to introduce a more specific notion of a tree, suitable for representing operadic operations and their composition rules.

\subsubsection{Introduction to Graph Theory}

Given a set \(V\) and an integer \(k\), with the symbol \([V]^k\) we denote the set of all subsets of \(V\) with \(k\) elements.
\begin{definition}[Graph] \label{def:grafo_base}
A \emph{graph} is a pair \(G=(V,E)\) of sets such that \(E \subseteq [V]^2\); the elements of \(V\) are called \emph{vertices} and the elements of \(E\) are called \emph{edges}. The two vertices contained in an edge are called its \emph{endpoints}. 
\end{definition}
Let \(G=(V,E)\) be a graph.
\begin{itemize}
    \item The sets of vertices and edges of \(G\) are denoted by \(V(G)\) and \(E(G)\), respectively.
    \item The \emph{order} of \(G\), denoted by \(|G|\), is the number of its vertices. Its \emph{size}, denoted by \(\|G\|\), is the number of its edges.
    \item A graph is \emph{finite} if its set of vertices is finite. Unless otherwise stated, in this review we  only consider finite graphs.
    \item The graph \((\varnothing, \varnothing)\) is called the \emph{empty graph}. A graph is called \emph{trivial} if its order is 0 or 1.
\end{itemize}

\begin{remark}[Simple and undirected graphs]
The definition of a graph given above, where edges are 2-element subsets of vertices ($E \subseteq [V]^2$), implicitly defines the type of graphs used throughout this review. Specifically:
\begin{itemize}
    \item The graphs are \textit{undirected}, as the edges are sets (e.g., $\{v_1, v_2\}$ is the same as $\{v_2, v_1\}$) and not ordered pairs.
    \item The graphs are \textit{simple}, as this formalism does not allow for loops (an edge from a vertex to itself, which would not be a 2-element set) or multiple edges between the same two vertices (since $E$ is a set of unique subsets).
\end{itemize}
This is the standard convention in this context.
\end{remark}

\begin{definition}[Adjacency and incidence] \label{def:adiacenza_incidenza}
Let \(G=(V,E)\) be a graph.
\begin{enumerate}
    \item A vertex \(v \in V\) is \emph{incident} to an edge \(e \in E\) if \(v \in e\).
    \item Two vertices \(x, y \in V\) are \emph{adjacent} or \emph{neighbors} if \(\{x,y\}\) is an edge of \(G\). The set of neighbors of a vertex \(v\) is denoted by \(N(v)\).
    \item Two edges \(e \neq f\) are \emph{adjacent} if they have a common endpoint.
\end{enumerate}
\end{definition}

\begin{remark}
A graph is an abstract combinatorial structure. Its visual representation, obtained by drawing vertices as points and edges as lines connecting them, serves as an intuitive aid. The geometric position of points and lines is irrelevant to the structure of the graph itself: the only information that matters is which pairs of vertices form an edge.
\end{remark}

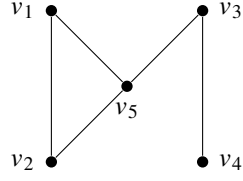
\begin{figure}[H]
\centering
\begin{tikzpicture}[
    every node/.style={font=\small},
    vertex/.style={circle, fill=black, inner sep=1.5pt, minimum size=4pt}
]
\node[vertex, label=left:{$v_1$}]  (v1) at (0, 2) {};
\node[vertex, label=left:{$v_2$}]  (v2) at (0, 0) {};
\node[vertex, label=right:{$v_3$}] (v3) at (2, 2) {};
\node[vertex, label=right:{$v_4$}] (v4) at (2, 0) {};
\node[vertex, label=below:{$v_5$}] (v5) at (1, 1) {};
\draw (v1) -- (v2) (v1) -- (v5) (v2) -- (v5) (v3) -- (v5) (v4) -- (v3);
\end{tikzpicture}
\caption{The figure shows an example of a simple, undirected graph \(G=(V,E)\) with \(V=\{v_1, \dots, v_5\}\) and \(E=\{\{v_1,v_2\}, \{v_1,v_5\},\{v_2,v_5\},\{v_3,v_4\},\{v_4,v_5\}\}\).}
\label{fig:esempio-grafo}
\end{figure}

\begin{definition}[Fundamental Concepts in Graph Theory]
\label{def:graph_theory_concepts}
Let $G=(V,E)$ and $G'=(V',E')$ be graphs.
\begin{enumerate}
    \item A graph $G'$ is a \emph{subgraph} of $G$, written $G' \subseteq G$, if $V' \subseteq V$ and $E' \subseteq E$.

    \item Two graphs $G$ and $G'$ are \emph{isomorphic}, written $G \simeq G'$, if there exists a bijection $\varphi\colon V(G) \to V(G')$ such that
    \begin{equation}
    \{x,y\} \in E(G) \iff \{\varphi(x), \varphi(y)\} \in E(G') \quad \forall x,y \in V(G).
    \end{equation}
    Such a map $\varphi$ is called an \emph{isomorphism}. In the following, we will not distinguish isomorphic graphs and will commonly write $G=G'$.

    \item The \emph{degree} $d(v)$ of a vertex $v$ is the number of edges incident to it, i.e., $d(v) = |N(v)|$. A vertex of degree 1 is called a \emph{leaf}\footnote{Except for the \emph{root} of a rooted tree (Definition \ref{def:albero-radicato-ordinato}) which is never called a leaf, even if it has degree 1.}.
    \item A \emph{path} in $G$ is a subgraph of $G$ with distinct vertices $v_0, v_k$ and edges $\{v_0, v_1\}, \dots, \{v_{k-1}, v_k\}$. The \emph{length} of a path is the number of its edges. The vertices $v_0$ and $v_k$ are the \emph{endpoints} of the path, while $v_1, \dots, v_{k-1}$ are its \emph{internal vertices}. A path of vertices $v_0,\dots,v_k$ is often denoted simply by the sequence of its vertices, written $v_0v_1\ldots v_k$\footnote{More precisely, $v_0v_1\ldots v_k$ and $v_kv_{k-1}\ldots v_0$ denote the same path. By fixing one of these two orderings, we may refer to the \emph{first vertex} of a path, with a slight abuse of notation.}. If a path has endpoints $v_0,\dots,v_k$ it is also said to be a path \emph{from} $v_0$ \emph{to} $v_k$ or a path \emph{between} $v_0$ and $v_k$.
        
    \item A \emph{k-cycle} (or simply a \emph{cycle}) in $G$ is a subgraph of $G$ with $k \ge 3$ vertices $v_0, \dots, v_{k-1}$ and edges $\{v_0,v_1\}, \dots, \{v_{k-1},v_0\}$.

    \item A graph $G$ is \emph{connected} if there exists a path in $G$ between every pair of distinct vertices.
    
    \item A \emph{tree} is a connected and acyclic (i.e., one without cycles) graph.
\end{enumerate}
\end{definition}

\begin{theorem}[Characterizations of trees] \label{thm:caratt_alberi}
The following statements are equivalent for a non-trivial graph \(T\):
\begin{enumerate}
    \item \(T\) is a tree.
    \item Every pair of vertices in \(T\) is connected by a unique path.
\end{enumerate}
\end{theorem}

\begin{proof}
    See \cite[Proposition 1.5.1]{diestel17}.
\end{proof}

\begin{definition}[Rooted and ordered tree]
\label{def:albero-radicato-ordinato}
 A \emph{rooted tree} is a pair \((T,r)\), where \(T\) is a tree and \(r \in V(T)\) is a distinguished vertex called the \emph{root}. The choice of a root allows the definition of parent-child relationships among the vertices:
 \begin{enumerate}
 \item The \emph{parent} of a vertex \(v \neq r\) is the unique vertex adjacent to \(v\) that lies on the unique path (guaranteed by Theorem \ref{thm:caratt_alberi}) between the root \(r\) and \(v\).
 \item A vertex \(u\) is a \emph{child} of \(v\) if \(v\) is the parent of \(u\).
 \item The \emph{level} of a vertex \(v\) is its distance from the root \(r\), i.e., the length of the unique path connecting them.
 \end{enumerate}
A rooted tree is called \emph{ordered} (or \emph{planar}) if the set of children of each vertex is totally ordered.
\end{definition}

\subsubsection{Graphs with Half-Edges and Operadic Trees}

Operadic composition is most effectively formalized by means of a graphical language. We next introduce its basic elements, designed to encode maps with several inputs and a single output.

\begin{definition}[Graph with half-edges]
A \emph{graph with half-edges} is a quadruple \(G = (V, E, F, s)\) where:
\begin{enumerate}
    \item \((V,E)\) is a graph. The elements of \(V\) are called \emph{internal vertices}.
    \item \(F\) is a finite set whose elements are called \emph{half-edges}.
    \item \(s : F \to V\) is the \emph{incidence function}, which assigns to each half-edge the unique internal vertex to which it is attached. If \(f \in F\) and \(s(f)=v\), we say that \(f\) is \emph{incident to} the vertex \(v\).
\end{enumerate}
\end{definition}

\begin{definition}[Associated classical graph]
\label{def:associated_graph}
Let ${G}=(V,E,F,s)$ be a graph with half-edges. Its \emph{associated classical graph}, denoted by $T(G)$, is the graph defined by the set of vertices $V \cup F$ and the set of edges $E \cup \{\{f, s(f)\} \mid f \in F\}$. In symbols:
\begin{equation}
T(G) := \bigl(V \cup F, E \cup \{\{f, s(f)\} \mid f \in F\}\bigr).
\end{equation}
\end{definition}

\begin{remark}[Degree of half-edges] \label{oss:grado_semi_archi}
From the definition of the associated classical graph \(T({G})\), we immediately obtain a fundamental structural property. Every half-edge \(f \in F\), which corresponds to a vertex in \(T({G})\), is by construction adjacent to exactly one other vertex, namely its incident vertex \(s(f) \in V\).
Consequently, in the associated classical graph \(T({G})\), all half-edges are vertices of degree 1.
\end{remark}

\begin{definition}[Operadic tree]
\label{def:albero_per_operad}
An \emph{operadic tree} is a graph with half-edges ${G}=(V,E,F,s)$ that satisfies two conditions:
\begin{enumerate}
    \item Its associated classical graph $T(G)$ is a tree (Definition~\ref{def:graph_theory_concepts}).
    \item The set of half-edges $F$ is partitioned into two disjoint subsets:
    \begin{itemize}
        \item A singleton $\{r\}$, whose element $r$ is called the \textit{root} (or output).
        \item A set $L$, whose elements are called \textit{leaves} (or inputs).
    \end{itemize}
\end{enumerate}
\end{definition}

\begin{remark}
Since an operadic tree \({G}\) is defined by the condition that its associated graph \(T({G})\) is a classical tree, we can extend several of the standard terminologies and properties of trees to \({G}\).
In particular, Theorem \ref{thm:caratt_alberi} guarantees that in \(T({G})\) there is a unique path between every pair of vertices. This allows us to define parent–child relations for an operadic tree as well, simply by referring to the corresponding relations in its associated graph.
\end{remark}

By convention, an operadic tree \({G}=(V,E,F,s)\) is represented by drawing its associated classical graph \(T({G})\), albeit with a key graphical simplification. While the internal vertices (the elements of \(V\)) are explicitly drawn as points, those of \(T({G})\) corresponding to half-edges (the elements of \(F\)) are not. As a result, edges of the type \(\{f, s(f)\}\), $f \in F$, appear in the drawing simply as edges incident on an internal vertex. The convention adopted throughout this review is to draw the root at the bottom and the leaves at the top.

\begin{definition}[Corolla]
\label{def:corolla}
For each \(n \ge 0\), the \emph{\(n\)-corolla} is the operadic tree \({C_n}=(V,E,F,s)\) with the following properties:
\begin{enumerate}
    \item It has a single internal vertex, \(V=\{v\}\).
    \item It has no internal edges, \(E=\varnothing\).
    \item It has \(n+1\) half-edges, \(F=\{r, l_1, \dots, l_n\}\), where \(r\) is the root and the \(l_i\) are the leaves.
    \item The incidence function $s: F \to V$ maps every half-edge to the unique internal vertex, i.e., $s(f) = v$ for all $f \in F$.
\end{enumerate}
\end{definition}

\begin{example}[Stump]
The $0$-corolla is called a \emph{stump}.
\end{example}

\begin{definition}[Trivial tree]
\label{def:albero_banale}
The \emph{trivial tree} is a special operadic tree that does not strictly adhere to all conditions of Definition~\ref{def:albero_per_operad}. It is defined as a graph with:
\begin{itemize}
    \item No internal vertices ($V = \emptyset$).
    \item A single half-edge that serves as both its root and its only leaf.
\end{itemize}
By convention, it is the unique tree with one input and one output that has no internal vertices.
\end{definition}

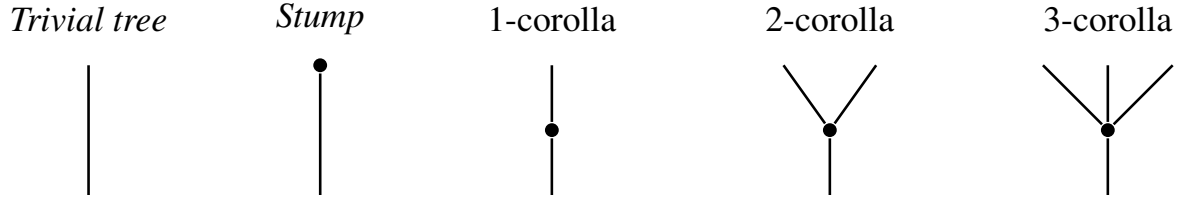
\begin{figure}[H]
    \centering
    \resizebox{\textwidth}{!}{
        \begin{tikzpicture}[thick]
            \node at (-8, 1.2) {\emph{Trivial tree}};
            \coordinate (trivial_leaf) at (-8,0.7);
            \coordinate (trivial_root) at (-8,-0.7);
            \draw (trivial_leaf) -- (trivial_root);

            \node at (-5.5, 1.2) {\emph{Stump}};
            \node[circle, fill=black, inner sep=1.5pt] (stump) at (-5.5,0.7) {};
            \coordinate (stump_root) at (-5.5,-0.7);
            \draw (stump) -- (stump_root);

            \node at (-3, 1.2) {1-corolla};
            \node[circle, fill=black, inner sep=1.5pt] (op1) at (-3,0) {};
            \coordinate (leaf1) at (-3,0.7);
            \coordinate (root1) at (-3,-0.7);
            \draw (leaf1) -- (op1);
            \draw (op1) -- (root1);

            \node at (0, 1.2) {2-corolla};
            \node[circle, fill=black, inner sep=1.5pt] (op2) at (0,0) {};
            \coordinate (leaf2a) at (-0.5,0.7);
            \coordinate (leaf2b) at (0.5,0.7);
            \coordinate (root2) at (0,-0.7);
            \draw (leaf2a) -- (op2);
            \draw (leaf2b) -- (op2);
            \draw (op2) -- (root2);

            \node at (3, 1.2) {3-corolla};
            \node[circle, fill=black, inner sep=1.5pt] (op3) at (3,0) {};
            \coordinate (leaf3a) at (2.3,0.7);
            \coordinate (leaf3b) at (3,0.7);
            \coordinate (leaf3c) at (3.7,0.7);
            \coordinate (root3) at (3,-0.7);
            \draw (leaf3a) -- (op3);
            \draw (leaf3b) -- (op3);
            \draw (leaf3c) -- (op3);
            \draw (op3) -- (root3);
        \end{tikzpicture}
    } 
    \caption{Visual representation of the trivial tree, the stump, and some corollas.}
    \label{fig:corolle}
\end{figure}

\begin{remark}[Trivial tree vs. 0-corolla]
It is important to distinguish the \emph{trivial tree} from the \emph{0-corolla} (or stump). The 0-corolla has one internal vertex and represents a 0-ary operation (a constant). In contrast, the trivial tree has no internal vertices and represents the 1-ary identity operation.
\end{remark}

\begin{definition}[Composition by grafting]
Let \(G_0 = (V_0, E_0, F_0, s_0)\) be an operadic tree with \(n\) leaves \(L_0 = (l_1, \dots, l_n) \subseteq F_0\) for which an ordering has been fixed. Let \((G_1, \dots, G_n)\) be sequence of \(n\) operadic trees, where each \(G_i = (V_i, E_i, F_i, s_i)\) has root \(r_i\). The \emph{composition} (or \emph{grafting}) of \((G_1, \dots, G_n)\) onto \(G_0\), denoted \(\gamma(G_0; G_1, \dots, G_n)\), is the new operadic tree \(G=(V,E,F,s)\) defined as follows:
\begin{enumerate}
    \item The set of internal vertices is the union of all vertices:
    \begin{equation}
    V = V_0 \cup V_1 \cup \dots \cup V_n.
    \end{equation}
    \item The set of internal edges is
    \begin{equation}
    E = E_0 \cup E_1 \cup \dots \cup E_n \cup \bigcup_{i=1}^{n} \{\{s_0(l_i), s_i(r_i)\}\}.
    \end{equation}
    \item The set of external half-edges is 
    \begin{equation}
    F = \{r_0\} \cup L_1 \cup \dots \cup L_n, 
    \end{equation}
    where \(r_0\) is the root of \(G_0\) and \(L_i\) is the set of leaves of \(G_i\).
    \item The incidence function \(s: F \to V\) is defined by:
    \begin{equation}
    s(f) = 
    \begin{cases}
        s_i(f) & \text{if } f \in L_i \text{ for some } i \in \{1,\dots,n\} \\
        s_0(r_0) & \text{if } f = r_0.
    \end{cases}
    \end{equation}
\end{enumerate}
The root of \(G\) is \(r_0\). The set of leaves of \(G\) is given by the union of the sets \(L_1, \dots, L_n\).
\end{definition}

\begin{figure}[H]
\centering
\begin{tikzpicture}[
    scale=1, 
    every node/.style={transform shape},
    v/.style={fill=white, inner sep=1.5pt}, 
    leaf/.style={inner sep=0pt}, 
]

\node[v] (v0) at (0,0) {$v_0$};
\node at (0, -1.7) {$G_0$};

\node[leaf] (r0) at (0, -1.2) {}; 
\node[leaf] (l1) at (-1.5, 1.2) {}; 
\node[leaf] (l2) at (0, 1.2) {};   
\node[leaf] (l3) at (1.5, 1.2) {};  

\draw (v0) -- (r0) node[midway, right=1pt] {$r_0$};
\draw (v0) -- (l1) node[midway, left=1pt] {$l_1$};
\draw (v0) -- (l2) node[midway, right=1pt] {$l_2$};
\draw (v0) -- (l3) node[midway, right=1pt] {$l_3$};

\begin{scope}[xshift=-1.5cm, yshift=1.7cm]
    \node at (0, 0) {$G_1$};
    \node[v] (v1) at (0,1) {$v_1$};
    \node[leaf] (r1) at (0, 0.4) {}; 
    \node[leaf] (f11) at (-0.5, 1.8) {}; 
    \node[leaf] (f12) at (0.5, 1.8) {};
    \draw (v1) -- (r1) node[midway, right=1pt] {$r_1$};
    \draw (v1) -- (f11);
    \draw (v1) -- (f12);
\end{scope}

\begin{scope}[xshift=0cm, yshift=1.7cm]
    \node at (0, 0) {$G_2$};
    \node[v] (v2) at (0,1) {$v_2$};
    \node[leaf] (r2) at (0, 0.4) {}; 
    \node[leaf] (f21) at (0, 1.8) {}; 
    \draw (v2) -- (r2) node[midway, right=1pt] {$r_2$};
    \draw (v2) -- (f21);
\end{scope}

\begin{scope}[xshift=1.5cm, yshift=1.7cm]
    \node at (0, 0) {$G_3$};
    \node[v] (v3) at (0,1) {$v_3$};
    \node[leaf] (r3) at (0, 0.4) {}; 
    \node[leaf] (f31) at (-0.5, 1.8) {}; 
    \node[leaf] (f32) at (0.5, 1.8) {};
    \draw (v3) -- (r3) node[midway, right=1pt] {$r_3$};
    \draw (v3) -- (f31);
    \draw (v3) -- (f32);
\end{scope}

\node at (3.5, 1) {\huge$\overset{\gamma}{\longrightarrow}$};

\begin{scope}[xshift=7cm, yshift=0cm]
    \node at (0, -1.7) {$\gamma(G_0; G_1, G_2, G_3)$};

    \node[v] (v0_new) at (0,0) {$v_0$};
    \node[v] (v1_new) at (-1.5, 1.5) {$v_1$};
    \node[v] (v2_new) at (0, 1.5) {$v_2$};
    \node[v] (v3_new) at (1.5, 1.5) {$v_3$};

    \draw (v0_new) -- (0, -1.2); 
    \draw (v0_new) -- (v1_new);
    \draw (v0_new) -- (v2_new);
    \draw (v0_new) -- (v3_new);

    \draw (v1_new) -- ++(-0.5, 0.8);
    \draw (v1_new) -- ++(0.5, 0.8);
    
    \draw (v2_new) -- ++(0, 0.8);

    \draw (v3_new) -- ++(-0.5, 0.8);
    \draw (v3_new) -- ++(0.5, 0.8);
\end{scope}

\end{tikzpicture}
\caption{Graphical representation of tree composition via grafting.}
\label{fig:innesto_alberi_invertito}
\end{figure}
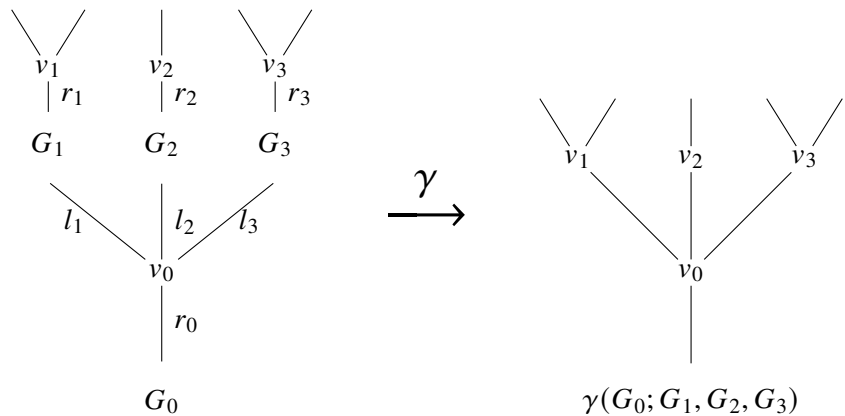

\begin{remark}
The trivial tree acts as a neutral element for composition. If a trivial tree is grafted onto a leaf \(l_i\), no internal vertices are added, no new internal edges are created, and the leaf $l_i$ is replaced by the leaf of the trivial tree, leaving the overall structure of the graph unchanged.
\end{remark}

\subsubsection{Graphical Representation of Abstract Operations and Composition}

Now that we have formally defined operadic trees and their composition, we can explain how they are used to represent the abstract operations of an operad.

An abstract operation of arity \(n\) is usually represented graphically by an \emph{\(n\)-corolla} \(C_n\) (Definition \,\ref{def:corolla}) where:
\begin{itemize}
  \item the \emph{leaves} represent the \(n\) \emph{inputs} of the operation;
  \item the \emph{root} represents the single \emph{output};
  \item the internal vertex \(v\) represents the abstract operation itself.
\end{itemize}
In particular, the stump represents a 0-ary operation, which can be regarded as a constant, while the trivial tree, on the other hand, models the identity operation.

Composition by grafting trees corresponds to the composition of the abstract operations they represent.

\begin{figure}[H]
\centering
\begin{tikzpicture}[thick, scale=1.3, every node/.style={circle, fill, inner sep=1.5pt}, baseline={(current bounding box.center)}]

\node[draw=none, fill=none] at (-4.3, 0.2) {$\gamma\Biggl($};

\node (a0) at (-3.5, 0.2) {};
\coordinate (a1) at (-3.7, 0.5);
\coordinate (a3) at (-3.3, 0.5);
\draw (a0) -- (a1);
\draw (a0) -- (a3);
\coordinate (o) at (-3.5, -0.1);
\draw (o) -- (a0);

\node[draw=none, fill=none] at (-3.0, 0.2) {$;$};

\node (b0) at (-2.2, 0.2) {};
\coordinate (b1) at (-2.4, 0.5);
\coordinate (b2) at (-2.2, 0.5);
\coordinate (b3) at (-2.0, 0.5);
\draw (b0) -- (b1);
\draw (b0) -- (b2);
\draw (b0) -- (b3);
\coordinate (o2) at (-2.2, -0.1);
\draw (o2) -- (b0);

\node[draw=none, fill=none] at (-1.6, 0.2) {$,$};

\node (c0) at (-0.9, 0.2) {};
\coordinate (c1) at (-1.1, 0.5);
\coordinate (c3) at (-0.7, 0.5);
\draw (c0) -- (c1);
\draw (c0) -- (c3);
\coordinate (o3) at (-0.9, -0.1);
\draw (o3) -- (c0);

\node[draw=none, fill=none] at (-0.2, 0.2) {$\Biggr)$};

\node[draw=none, fill=none] at (0.6, 0.2) {$=$};

\node (d0) at (1.8, -0.2) {}; 
\node (d1) at (1.3, 0.3) {};  
\node (d2) at (2.3, 0.3) {};  

\coordinate (d1a) at (0.9, 0.8);
\coordinate (d1b) at (1.3, 0.8);
\coordinate (d1c) at (1.7, 0.8);
\draw (d1) -- (d1a);
\draw (d1) -- (d1b);
\draw (d1) -- (d1c);
\draw (d0) -- (d1);
\draw (d0) -- (d2);

\coordinate (d2a) at (2.1, 0.8);
\coordinate (d2b) at (2.5, 0.8);
\draw (d2) -- (d2a);
\draw (d2) -- (d2b);

\draw (d0) -- +(0,-0.3);

\end{tikzpicture}
\caption{Example of composition of binary and ternary operations.}
\end{figure}
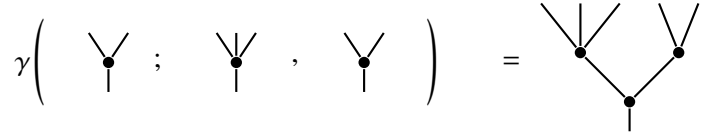

\begin{remark}
In general, one may need to distinguish between different abstract operations having the same number of inputs. To distinguish between such operations, one can label the vertex of the corolla with the symbol of the operation being considered. For example, let us consider three different binary operations, indicated respectively by \(\mu\), \(\nu_1\), and \(\nu_2\). The trees corresponding to these operations are the 2-corollas with the vertex labeled by each of these operations.
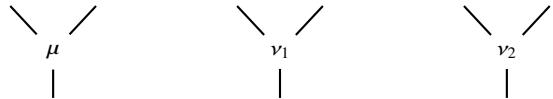
\begin{figure}[H]
\centering
\begin{tikzpicture}[every node/.style={font=\scriptsize}, thick]

\node (ou) at (0,-1.5) {};
\node (theta) at (0,-0.75) {$\mu$};
\draw (ou) -- (theta);

\node (a1) at (-0.7,0) {};
\node (a2) at (+0.7,0) {};

\draw (theta) -- (a1);
\draw (theta) -- (a2);

\node (ou3) at (6,-1.5) {};
\node (theta3) at (6,-0.75) {$\nu_2$};
\draw (ou3) -- (theta3);

\node (a5) at (5.3,0) {};
\node (a6) at (+6.7,0) {};

\draw (theta3) -- (a5);
\draw (theta3) -- (a6);

\node (ou2) at (3,-1.5) {};
\node (theta2) at (3,-0.75) {$\nu_1$};
\draw (ou2) -- (theta2);

\node (a3) at (2.3,0) {};
\node (a4) at (+3.7,0) {};

\draw (theta2) -- (a3);
\draw (theta2) -- (a4);

\end{tikzpicture}
\caption{Representation of binary operations using labeled 2-corollas.}
\label{fig:corolle etichettate}
\end{figure}
The composition of these operations yields a new operation whose arity is the sum of the arities of the operations being composed. If this new operation is indicated by \(\lambda\), we can write
\begin{equation}
    \label{eq: innesto}
    \gamma(\mu; \nu_1, \nu_2) = \lambda,
\end{equation}
where \(\gamma\) represents composition by grafting. Below is the tree representation of \eqref{eq: innesto}.
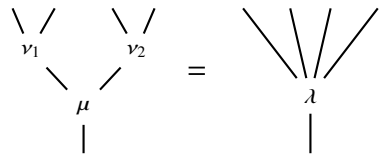
\begin{figure}[H]
\centering
\begin{tikzpicture}[every node/.style={font=\scriptsize}, thick]

\node (ou) at (0,-1.5) {};
\node (theta) at (0,-0.75) {$\mu$};
\draw (ou) -- (theta);

\node (a1) at (-0.7,0) {$\nu_1$};
\node (a2) at (0.7,0) {$\nu_2$};

\draw (theta) -- (a1);
\draw (theta) -- (a2);

\node (a11) at (-1,0.7) {};
\node (a12) at (-0.3,0.7) {};
\node (a21) at (0.3,0.7) {};
\node (a22) at (1,0.7) {};

\draw (a1) -- (a11);
\draw (a1) -- (a12);
\draw (a2) -- (a21);
\draw (a2) -- (a22);

\node at (1.5, -0.3) {\large $=$};

\node (ou2) at (3, -1.5) {};
\node (theta2) at (3, -0.6) {$\lambda$};
\draw (ou2) -- (theta2);

\node (b1) at (2,0.7) {};
\node (b2) at (2.7,0.7) {};
\node (b3) at (3.3,0.7) {};
\node (b4) at (4,0.7) {};

\draw (theta2) -- (b1);
\draw (theta2) -- (b2);
\draw (theta2) -- (b3);
\draw (theta2) -- (b4);

\end{tikzpicture}
\caption{Tree representation of the composition $\gamma(\mu; \nu_1, \nu_2) = \lambda$.}
\end{figure}
\end{remark}

\subsubsection{Planar and Non-Planar Trees}

The distinction between non-symmetric operads and symmetric operads, which we will formally define in Section~\ref{sec: s operad}, is reflected in the structure of the trees used to represent their operations.

\begin{definition}[Planar and non-planar operadic tree]
\label{def:albero_planare_e_non}
A \emph{planar operadic tree} is an operadic tree \(G=(V,E,F,s)\) equipped with a total order on its set of leaves $L$.
A \emph{non-planar operadic tree} is simply an operadic tree.
\end{definition}

\begin{remark}[Planar Trees, embeddings, and labeling conventions]
A key distinction is made between planar and non-planar trees. A \emph{planar tree} is equipped with an intrinsic total order on its leaves, which is tipically represented by a left-to-right arrangement in diagrams. This induces a global ordering that makes labeling the leaves superfluous.

In contrast, a \emph{non-planar tree} has no intrinsic ordering of its leaves. Any drawing of such a tree on a plane, known as a \emph{planar embedding}, requires an arbitrary choice of ordering. In order to distinguish between the different possible embeddings of a non-planar tree with $n$ leaves, its leaves are labeled (typically with integers). All possible embeddings can be generated from a single one by the action of the symmetric group $\mathbb{S}_n$, which permutes these labels. For instance, a non-planar 3-corolla has $3! = 6$ distinct planar embeddings, each corresponding to a unique permutation in $\mathbb{S}_3$. In particular, if \(G\) is a non-planar \(n\)-corolla and we have fixed a planar embedding with labels \((1,\dots,n)\), every other arrangement corresponds to the action of a permutation \(\sigma\in\mathbb{S}_n\) that renumbers leaf \(i\) to \(\sigma(i)\).

Henceforth, when we refer to a ``planar tree'' or a ``non-planar tree'', we  implicitly mean its planar embedding — the graphical arrangement that respects the intrinsic order (for planar trees) or a chosen labeled order (for non-planar trees).
\end{remark}

\begin{example}[Planar embedding of a planar and non-planar 3-corolla]
Figure \ref{fig:planar_embeddings_example} illustrates the unique planar embedding of a planar 3-corolla and the six possible planar embeddings (choices of leaf ordering) for a non-planar 3-corolla.

\begin{figure}[H]
\centering
\resizebox{\textwidth}{!}{
\begin{tikzpicture}[every node/.style={font=\scriptsize}, thick]

\node (ou) at (0,-1.5) {};
\node (theta) at (0,-0.75) {$\theta$};
\draw (ou) -- (theta);

\node (a1) at (-0.7,0) {1};
\node (a3) at (0.7,0) {3};
\node (a2) at (0,0) {2};

\draw (theta) -- (a1);
\draw (theta) -- (a2);
\draw (theta) -- (a3);

\node at (1.2, -1) {\large ,};

\node (ou2) at (2.5, -1.5) {};
\node (theta2) at (2.5, -0.75) {$\theta$};
\draw (ou2) -- (theta2);

\node (b1) at (1.8,0) {1};
\node (b3) at (2.5,0) {3};
\node (b4) at (3.2,0) {2};

\draw (theta2) -- (b1);
\draw (theta2) -- (b3);
\draw (theta2) -- (b4);

\node at (3.7, -1) {\large ,};

\node (ou3) at (5,-1.5) {};
\node (theta3) at (5,-0.75) {$\theta$};
\draw (ou3) -- (theta3);

\node (c1) at (4.3,0) {2};
\node (c2) at (5,0) {1};
\node (c3) at (5.7,0) {3};

\draw (theta3) -- (c1);
\draw (theta3) -- (c2);
\draw (theta3) -- (c3);

\node at (6.2, -1) {\large ,};

\node (ou4) at (7.5,-1.5) {};
\node (theta4) at (7.5,-0.75) {$\theta$};
\draw (ou4) -- (theta4);

\node (d1) at (6.8,0) {2};
\node (d2) at (7.5,0) {3};
\node (d3) at (8.2,0) {1};

\draw (theta4) -- (d1);
\draw (theta4) -- (d2);
\draw (theta4) -- (d3);

\node at (8.7, -1) {\large ,};

\node (ou5) at (10,-1.5) {};
\node (theta5) at (10,-0.75) {$\theta$};
\draw (ou5) -- (theta5);

\node (e1) at (9.3,0) {3};
\node (e2) at (10,0) {1};
\node (e3) at (10.7,0) {2};

\draw (theta5) -- (e1);
\draw (theta5) -- (e2);
\draw (theta5) -- (e3);

\node at (11.2, -1) {\large ,};

\node (ou6) at (12.5,-1.5) {};
\node (theta6) at (12.5,-0.75) {$\theta$};
\draw (ou6) -- (theta6);

\node (f1) at (11.8,0) {3};
\node (f2) at (12.5,0) {2};
\node (f3) at (13.2,0) {1};

\draw (theta6) -- (f1);
\draw (theta6) -- (f2);
\draw (theta6) -- (f3);

\node at (13.7, -1) {\large .};

\end{tikzpicture}
}
\caption{Possible \emph{planar embeddings} of a \emph{non-planar} 3-corolla.}
\label{fig:planar_embeddings_example}
\end{figure}
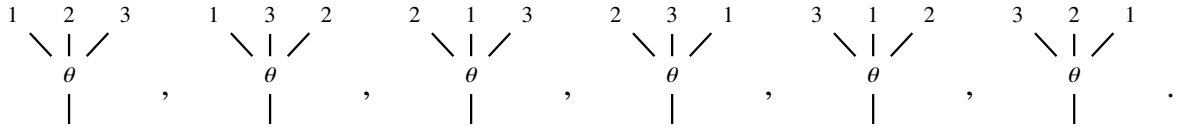
\end{example}

In \emph{non-symmetric operads}, planar trees are used. In a non-symmetric operad, the order of the inputs is fixed and is an integral part of the operation, meaning that altering the order results in a distinct operation.
In Section \ref{sec: s operad}, we will see that non-planar trees are used to represent abstract operations in \emph{symmetric operads}.

\subsection{Examples of Non-Symmetric Operads}

\begin{example}[Symmetries operad]
Consider the non-symmetric operad of sets \( \mathsf{S} \), called the \emph{symmetries operad}, defined by setting
\begin{equation}
\mathsf{S}(n) \coloneqq \mathbb{S}_n,
\end{equation}
i.e., the symmetric group on $n$ elements.

A permutation \( \sigma \in \mathbb{S}_n \) can be visualized as a \emph{string diagram} where the strings go from left to right: each string connects the starting position \( i \) to the ending position \( \sigma(i) \). For example, the permutation
\begin{equation}
\sigma = \begin{pmatrix} 1 & 2 & 3 \\ 2 & 3 & 1 \end{pmatrix} \in \mathbb{S}_3
\end{equation}
can be represented graphically as follows:

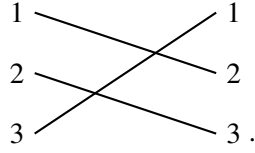
\begin{center}
\begin{tikzpicture}[scale=0.8, baseline=(current bounding box.center)]
    \foreach \i/\j in {1/2, 2/3, 3/1} {
        \draw[thick] (0,4-\i) -- (3,4-\j);
    }

    \foreach \i in {1,2,3} {
        \node[left] at (0,4-\i) {\i};
    }
    \foreach \i in {1,2,3} {
        \node[right] at (3,4-\i) {\i};
    }
    \node at (3.6, 0.9) {.};

\end{tikzpicture}
\end{center}
Given $\sigma \in \mathbb{S}_k, \tau_1 \in \mathbb{S}_{\ell_1}, \dots, \tau_k \in \mathbb{S}_{\ell_k},$ the composition
$\sigma \circ (\tau_1, \dots, \tau_k) \in \mathbb{S}_{\ell_1 + \cdots + \ell_k}$
is defined as the permutation obtained by:
\begin{itemize}
    \item applying each \( \tau_i \) within the \( i \)-th block (composed of \( \ell_i \) elements),
    \item and then reordering the blocks according to \( \sigma \).
\end{itemize}

For example, let
\[
\sigma = \begin{pmatrix} 1 & 2 & 3 \\ 2 & 3 & 1 \end{pmatrix} \in \mathbb{S}_3, \quad
\tau_1 = \begin{pmatrix} 1 & 2 \\ 1 & 2 \end{pmatrix} \in \mathbb{S}_2, \quad
\tau_2 = \begin{pmatrix} 1 & 2 & 3 \\ 3 & 1 & 2 \end{pmatrix} \in \mathbb{S}_3, \quad
\tau_3 = \begin{pmatrix} 1 & 2 \\ 2 & 1 \end{pmatrix} \in \mathbb{S}_2.
\]

Consider the composition \( \sigma \circ (\tau_1, \tau_2, \tau_3) \in \mathbb{S}_7 \). The blocks are:
\begin{itemize}
    \item Block 1: positions 1, 2 (on which \( \tau_1 \) acts),
    \item Block 2: positions 3, 4, 5 (on which \( \tau_2 \) acts),
    \item Block 3: positions 6, 7 (on which \( \tau_3 \) acts),
\end{itemize}
and the permutation \( \sigma \) reorders them as block 2, then 3, then 1.

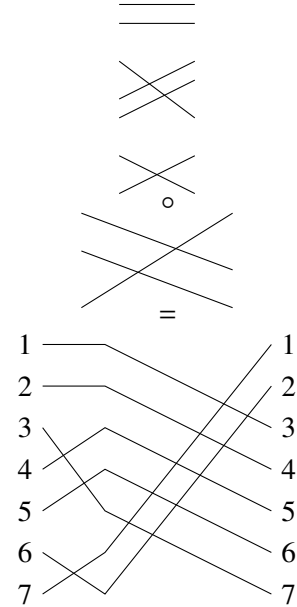
\begin{center}

\begin{tikzpicture}[scale=0.5, baseline={([yshift=-.5ex]current bounding box.center)}]

    \draw[thin] (0, 4) -- (2, 4);
    \draw[thin] (0, 3.5) -- (2, 3.5);

    \draw[thin] (0, 2.5) -- (2, 1);
    \draw[thin] (0, 1.5) -- (2, 2.5);
    \draw[thin] (0, 1) -- (2, 2);

    \draw[thin] (0, 0) -- (2, -1);
    \draw[thin] (0, -1) -- (2, 0);
\end{tikzpicture}

\quad $\circ$ \quad

\begin{tikzpicture}[scale=0.5, baseline={([yshift=-.5ex]current bounding box.center)}]

    \draw[thin] (0, 1) -- (4, -0.5);
    \draw[thin] (0, 0) -- (4, -1.5);
    \draw[thin] (0, -1.5) -- (4, 1);
\end{tikzpicture}

\quad $=$ \quad

\begin{tikzpicture}[scale=0.55, baseline={([yshift=-.5ex]current bounding box.center)}]

    \foreach \N in {1,...,7} {
        \node[left] at (0, -\N) {\N};
        \node[right] at (5.5, -\N) {\N};
    }

    \draw[thin] (0, -1) -- (1.5, -1) -- (5.5, -3); 
    \draw[thin] (0, -2) -- (1.5, -2) -- (5.5, -4); 
    \draw[thin] (0, -3) -- (1.5, -5) -- (5.5, -7); 
    \draw[thin] (0, -4) -- (1.5, -3) -- (5.5, -5); 
    \draw[thin] (0, -5) -- (1.5, -4) -- (5.5, -6); 
    \draw[thin] (0, -6) -- (1.5, -7) -- (5.5, -2); 
    \draw[thin] (0, -7) -- (1.5, -6) -- (5.5, -1); 
\end{tikzpicture}
\end{center}
We thus obtain the permutation:
\begin{equation}
\sigma \circ (\tau_1, \tau_2, \tau_3) = \begin{pmatrix}
1 & 2 & 3 & 4 & 5 & 6 & 7 \\
3 & 4 & 7 & 5 & 6 & 2 & 1
\end{pmatrix} \in \mathbb{S}_7.
\end{equation}
The unit of the operad is the trivial permutation of one element
\begin{equation}
\mathrm{id} = \begin{pmatrix} 1 \\ 1 \end{pmatrix} \in \mathsf{S}(1).
\end{equation}
\end{example}

Showing that $\mathsf{S}$ defines an operad requires a non-trivial proof. In particular, it is necessary to check that the composition defined above is associative and compatible with the identity, which involves a detailed analysis of block permutations and their iterated compositions. A rigorous treatment of these aspects can be found, for example, in \cite{lodayvallette}.

\begin{example}[Non-symmetric operad of trees]
The non-symmetric operad of trees \( \mathsf{T}_{\mathrm{ns}} \) is the operad defined as follows:
\begin{enumerate}
    \item  \( \forall n \in \mathbb{N}_0 \ \ \mathsf{T}_{\mathrm{ns}}(n) \) consists of all \textit{rooted}, \textit{planar} trees with \( n \) leaves.
    \item  The composition maps in \( \mathsf{T}_{\mathrm{ns}} \) are the functions:
    \begin{equation}
    \gamma : \mathsf{T}_{\mathrm{ns}}(n) \times \mathsf{T}_{\mathrm{ns}}(k_1) \times \dots \times \mathsf{T}_{\mathrm{ns}}(k_n) \to \mathsf{T}_{\mathrm{ns}}(k_1 + \dots + k_n)
    \end{equation}
    defined as follows: given a tree \( T \in \mathsf{T}_{\mathrm{ns}}(n) \), and trees \( T_1, \dots, T_n \), with \( T_i \in \mathsf{T}_{\mathrm{ns}}(k_i) \), the composition
    \begin{equation}
    \gamma(T; T_1, \dots, T_n)
    \end{equation}
    is the tree obtained by grafting the root of each \( T_i \) onto the \( i \)-th leaf of \( T \), for \( i = 1, \dots, n \). After grafting, the leaves of the resulting tree are re-labeled from \( 1 \) to \( k_1 + \dots + k_n \), respecting the induced (lexicographical) order.
    \item The unit is the trivial tree \( \mathrm{\mathbf{1}} \in T_{\mathrm{ns}}(1) \).
\end{enumerate}
\end{example}

Verifying that \(\mathsf{T}_{\mathrm{ns}}\) indeed defines a non-symmetric operad structure requires a detailed check of the associativity laws of grafting. In particular, it must be shown that the result of grafting onto two distinct leaves is independent of the order in which the operations are performed, and that the \emph{trivial tree} acts as a neutral element. A rigorous treatment of these aspects can be found in \cite{markl_shnider_stasheff}.

\subsection{Symmetric Operads} \label{sec: s operad}
In the context of non-symmetric operads, each n-ary operation is uniquely associated with a given ordered sequence of n arguments. The order of the inputs is, therefore, an integral part of the structure: two operations that differ by a permutation of the arguments are considered distinct, and a priori there is no relationship between them. However, in many algebraic situations, the operations of interest possess symmetries with respect to the permutation of arguments. A classic example is that of commutative algebras, where multiplication satisfies $x \cdot y = y \cdot x$ for every pair of elements $x,y$. To model such structures, it is natural to introduce a more general notion of an operad, which takes into account the action of the symmetric group $\mathbb{S}_n$ on the n-ary operations.

A \emph{symmetric operad} is therefore an operad $\mathcal{P}$ in which, for each $n \in \mathbb{N}_0$, the set $\mathcal{P}(n)$ is endowed with a (right) action of the group $\mathbb{S}_n$, and the composition laws are compatible with this action.

In the following, we will provide a formal definition of a symmetric operad, as well as those of a morphism and representation. We will then present two fundamental examples: the operad $\mathsf{Com}$, which encodes commutative algebras, and its unital counterpart $\mathsf{uCom}$.

\begin{remark}[Action of \(\mathbb{S}_n\) on $V^{\otimes n}$]
\label{rem:azione_sn}
Let \(V\) be a vector space. We define a left and a right action of the symmetric group \(\mathbb S_n\) on \(V^{\otimes n}\).

\begin{itemize}
    \item \emph{Left Action:} 
    $\sigma \cdot (v_1 \otimes \dots \otimes v_n)
    \coloneqq v_{\sigma^{-1}(1)} \otimes \dots \otimes v_{\sigma^{-1}(n)}.$

    \item \emph{Right Action:}
    $(v_1 \otimes \dots \otimes v_n) \cdot \sigma
    \coloneqq v_{\sigma(1)} \otimes \dots \otimes v_{\sigma(n)}.$
\end{itemize}
These two actions are well-defined and compatible\footnote{Compatibility in this context refers to the relationship between the two actions. For any $\sigma \in \mathbb{S}_n$ and $v \in V^{\otimes n}$, the identity $\sigma \cdot v = v \cdot \sigma^{-1}$ holds.}. In operad theory, and for consistency with standard references such as \cite{samchuckschnarch}, adopting the right action is the prevailing convention.
\end{remark}

As in the non-symmetric case, we begin by introducing the definition of a symmetric multicategory.

\begin{definition}[Symmetric multicategory]
\label{def: multicategoria simmetrica}
A \emph{symmetric multicategory} is a multicategory \(\mathscr C\) that, for any sequence of objects \(a_1,\dots,a_n,a \in \text{Ob}(\mathscr C)\) and any permutation \(\sigma\in\mathbb S_n\), is equipped with a family of maps:
\begin{equation}
\begin{array}{r c l}
(-)\cdot\sigma: \Hom\bigl(a_1,\dots,a_n; a\bigr) & \longrightarrow & \Hom\bigl(a_{\sigma(1)},\dots,a_{\sigma(n)}; a\bigr) \\[1ex]
\theta & \longmapsto & \theta\cdot\sigma,
\end{array}
\end{equation}
that satisfy:
\begin{enumerate}
\item \emph{Group action axioms:}
\begin{equation}
  (\theta\cdot\sigma)\cdot\tau = \theta\cdot(\sigma\circ\tau),
  \quad
  \theta\cdot\mathrm{id} = \theta,
\end{equation}
for every \(\theta\in\Hom(a_1,\dots,a_n;a)\) and \(\sigma,\tau\in\mathbb S_n\).

\item {Compatibility with composition (\emph{equivariance} property):}
\begin{equation}
\begin{split}
  \label{eq:equivarianzam}
  &(\theta\cdot\sigma)\;\circ\;
  \bigl(\theta_{\sigma(1)}\!\cdot\pi_{\sigma(1)},\dots,\theta_{\sigma(n)}\!\cdot\pi_{\sigma(n)}\bigr)
  = \\
  &= \bigl(\theta\circ(\theta_1,\dots,\theta_n)\bigr)
  \;\cdot\;\bigl(\sigma\circ(\pi_1,\dots,\pi_n)\bigr),
  \end{split}
\end{equation}
for every
\(\theta\in \Hom(a_1,\dots,a_n; a)\),
\(\theta_i\in \Hom(a_{i1},\dots,a_{i k_i}; a_i)\),
\(\sigma\in\mathbb S_n\) and \(\pi_i\in\mathbb S_{k_i}\).
\item {Compatibility with units:}
if \(1_a\in\Hom(a;a)\) is the unit, then
\begin{equation}
  1_a\cdot\id = 1_a,
\end{equation}
where \(\id\in\mathbb S_1\) is the unique permutation on one element.
\end{enumerate}
\end{definition}

\begin{remark}
For each \(n,\sigma,a_1,\dots,a_n,a\), the map
\begin{equation}
  \Hom(a_1,\dots,a_n; a)
  \;\longrightarrow\;
  \Hom(a_{\sigma(1)},\dots,a_{\sigma(n)}; a),
  \quad
  \theta \mapsto \theta\cdot\sigma
\end{equation}
is a bijection whose inverse is 
\(\theta' \mapsto \theta'\cdot\sigma^{-1}\).
\end{remark}

\begin{remark}
In Equation \eqref{eq:equivarianzam}, the term $\sigma\circ(\pi_1,\dots,\pi_n)\in\mathbb S_{k_1+\cdots+k_n}$ denotes the composition of permutations as defined in the symmetries operad.
\end{remark}

A symmetric operad is a symmetric multicategory with a single object. In analogy with the non-symmetric case, we give an explicit definition, specializing to the case of symmetric operads on \texttt{Vect}.
\begin{definition}[Symmetric operad]
\label{def: operad simmetrico}
A \emph{symmetric operad} $\mathcal{P}$ is a non-symmetric operad (see Definition \ref{def: ns operad}) such that:
\begin{enumerate}
    \item For each $n \in \mathbb{N}_0$, the vector space $\mathcal{P}(n)$ is endowed with a right linear action of the symmetric group $\mathbb{S}_n$, i.e., there exists a group anti-homomorphism $\rho_n \colon \mathbb{S}_n \;\to \; \Aut\bigl(\mathcal{P}(n)\bigr)$, also denoted $\rho_n(\sigma)(\theta) =: \theta \cdot \sigma $
    \item The composition
    \begin{equation}
      \gamma_{k_1,\dots,k_n}\colon
      \mathcal{P}(n)\otimes\mathcal{P}(k_1)\otimes\cdots\otimes\mathcal{P}(k_n)
      \;\longrightarrow\;
      \mathcal{P}\bigl(k_1+\cdots+k_n\bigr)
    \end{equation}
    is compatible with the action of the symmetric group (\emph{equivariance} law): given
    $\theta\in\mathcal{P}(n)$, $\theta_i\in\mathcal{P}(k_i)$, $\sigma\in\mathbb{S}_n$ and
    $\pi_i\in\mathbb{S}_{k_i}$, we have
    \begin{equation}
    \begin{split}
      \label{eq:equivarianza}
      & (\theta\cdot\sigma)\circ\bigl(\theta_{\sigma(1)}\cdot\pi_{\sigma(1)},\dots,
      \theta_{\sigma(n)}\cdot\pi_{\sigma(n)}\bigr)
      = \\
      &= \bigl(\theta\circ(\theta_1,\dots,\theta_n)\bigr)\cdot
      \bigl(\sigma\circ(\pi_1,\dots,\pi_n)\bigr),
    \end{split}
    \end{equation}
    where $\sigma\circ(\pi_1,\dots,\pi_n)\in\mathbb{S}_{k_1+\cdots+k_n}$ is the composition of permutations in the symmetries operad.
    \item The unit $1\in\mathcal{P}(1)$ is compatible with the action of the symmetric group:
    \begin{equation}
      1\cdot\mathrm{id} \;=\; 1,
    \end{equation}
    where $\mathrm{id}\in\mathbb{S}_1$ is the unique permutation on one element.
\end{enumerate}
\end{definition}

A sequence $(\mathcal{P}(n))_{n \in \mathbb{N}_0}$ of vector spaces $\mathcal{P}(n)$, each endowed with a right action of the symmetric group $\mathbb{S}_n$, as in the previous definition, is sometimes also called an \emph{$\mathbb{S}$-module} $\mathcal{P}$. 

\begin{remark}[Symmetric operads as multicategories with one object]
Analogously to the non-symmetric case, the definition of a symmetric operad is equivalent to the definition of a symmetric multicategory with a single object. If $\mathcal{M}$ is a symmetric multicategory with a single object $a$, then
\begin{equation}
\mathcal{P}(n) \coloneqq \Hom(\underbrace{a, \ldots, a}_n; a).
\end{equation}
The action of the symmetric group $\mathbb{S}_n$ on $\mathcal{P}(n)$ is the obvious one. The composition laws and compatibilities derive directly from the axioms of the symmetric multicategory.
\end{remark}

\begin{definition}[Morphism of symmetric operads]
\label{def: s morf}
Let $\mathcal{P}$ and $\mathcal{Q}$ be two symmetric operads on \texttt{Vect}. A \emph{morphism of symmetric operads} $\varphi: \mathcal{P} \to \mathcal{Q}$ is a morphism of non-symmetric operads (as in Definition \ref{def: morfismo di operad ns}) that, in addition, satisfies the following condition of \emph{compatibility with the action of the symmetric group:} for every $\theta \in \mathcal{P}(n)$ and $\sigma \in \mathbb{S}_n$,
    \begin{equation}
        \varphi_n(\theta \cdot \sigma) = \varphi_n(\theta) \cdot \sigma. \label{eq:comp_azione_simm}
    \end{equation}
A morphism of symmetric operads $\varphi$ is an \emph{isomorphism} if there exists an inverse morphism $\psi: \mathcal{Q} \to \mathcal{P}$ (also of symmetric operads) such that $\psi \circ \varphi = \mathrm{id}_{\mathcal{P}}$ and $\varphi \circ \psi = \mathrm{id}_{\mathcal{Q}}$. Equivalently, $\varphi$ is an isomorphism if and only if each of its components $\varphi_n$ is an isomorphism of vector spaces.
\end{definition}

Analogously to the non-symmetric case, one can define the category of all symmetric multicategories, which we  denote by ${\texttt{SymMultiCat}}$. The collection of all symmetric operads thus forms a \emph{full subcategory} of ${\texttt{SymMultiCat}}$, which we  denote by ${\texttt{SymOp}}$.

\begin{example}[Symmetric endomorphism operad]
\label{es: s end}
Let $V$ be a vector space.
The \emph{(symmetric) endomorphism operad} of $V$, denoted $\End_V^{\mathrm{sym}}$, is defined as follows.
\begin{itemize}
    \item For each $n \in \mathbb{N}_0$, we set:
    \begin{equation}
        \End_V^{\mathrm{sym}}(n) \coloneqq \Hom(V^{\otimes n}, V).
    \end{equation}
    \label{eq: azione su end}

    \item The \emph{action of the symmetric group} $\mathbb{S}_n$ on $\End_V^{\mathrm{sym}}(n)$ is given by: for $f \in \Hom(V^{\otimes n}, V)$ and $\sigma \in \mathbb{S}_n$,
    \begin{equation}
        (f \cdot \sigma)(v_1 \otimes \cdots \otimes v_n) = f\bigl(\sigma \cdot (v_1 \otimes \ldots \otimes v_n)\bigr) = f(v_{\sigma^{-1}(1)} \otimes \cdots \otimes v_{\sigma^{-1}(n)}).
    \end{equation}
    \item The \emph{composition} is defined as in the non-symmetric case.
    \item The \emph{unit} is the identity map $\mathrm{id}_V \in \End_V^{\mathrm{sym}}(1)$.
\end{itemize}
Lemma \ref{le: end} extends to this case. In other words, $\End_V^{\mathrm{sym}}$ is a well-defined symmetric operad, since the action of the symmetric group and the composition of multilinear functions are compatible.
\end{example}

\begin{definition}[Representation of a symmetric operad]
\label{def: P-algebra-sym}
Let $\mathcal{P}$ be a symmetric operad. A \emph{representation} of $\mathcal{P}$ (or $\mathcal{P}$-\emph{algebra}) consists of a vector space $V$ together with a morphism of symmetric operads
\begin{equation}
\Phi : \mathcal{P} \to \End_V^{\mathrm{sym}}.
\end{equation}
This implies that for each $n \in \mathbb{N}_0$, $\Phi_n : \mathcal{P}(n) \to \Hom(V^{\otimes n}, V)$ is a linear map that respects the action of the symmetric group, the composition, and the unit, as required by Definition \ref{def: s morf}.
\end{definition}

\begin{definition}[Morphism of $\mathcal{P}$-algebras (symmetric)]
\label{def: morfismo di P-algebre simm}
Let $\mathcal{P}$ be a symmetric operad and let $\left( V,\Phi^V: \mathcal{P} \to \End_V^{\mathrm{sym}} \right)$ and $\left( W, \Phi^W: \mathcal{P} \to \End_W^{\mathrm{sym}} \right)$ be two $\mathcal{P}$-algebras. A \emph{morphism of $\mathcal{P}$-algebras} from the first to the second is a linear map $f: V \to W$ such that the following diagram commutes for every operation $\mu_n \in \mathcal{P}(n)$ and for every $n \in \mathbb{N}_0$:
\begin{equation}
\begin{tikzcd}[column sep=large] 
V^{\otimes n} \arrow[d, "f^{\otimes n}"'] \arrow[r, "\Phi^V_n(\mu_n)"] & V \arrow[d, "f"] \\
W^{\otimes n} \arrow[r, "\Phi^W_n(\mu_n)"] & W
\end{tikzcd}
\end{equation}
i.e., in explicit form:
\begin{equation}
f \circ \Phi^V_n(\mu_n) = \Phi^W_n(\mu_n) \circ f^{\otimes n}.
\end{equation}
This condition is identical to the condition for morphisms of non-symmetric operads, but it applies in a context where the operations respect symmetries.
\end{definition}

\begin{definition}[Category of $\mathcal{P}$-algebras]
Let $\mathcal{P}$ be a symmetric operad. The category of $\mathcal{P}$-algebras, or algebras over $\mathcal{P}$, denoted $\texttt{Alg}(\mathcal{P})$, is defined as follows:
\begin{enumerate}
    \item The {objects} of $\texttt{Alg}(\mathcal{P})$ are the representations (or algebras) of the operad $\mathcal{P}$ on a vector space $V$, in the sense of Definition \ref{def: P-algebra-sym}.
    \item The {morphisms} of $\texttt{Alg}(\mathcal{P})$ are the morphisms of $\mathcal{P}$-algebras, in the sense of Definition \ref{def: morfismo di P-algebre simm}. 
    \item The {composition of morphisms} in $\texttt{Alg}(\mathcal{P})$ is the usual composition of linear maps. If $f: V \rightarrow W$ and $g: W \rightarrow Z$ are morphisms of $\mathcal{P}$-algebras, then $g \circ f: V \rightarrow Z$ is a morphism of $\mathcal{P}$-algebras, as can be easily verified.
    \item For each object $(V, \Phi^V: \mathcal{P} \rightarrow \End_V^{\text{sym}})$, the {identity morphism} is the identity $\text{id}_V: V \rightarrow V$.
\end{enumerate}
\end{definition}

\subsection{Visual Representation of Operadic Operations via Trees - Symmetric Case}

Recalling our discussion from Section \ref{subsec: rappr operad non simm}, we now focus on the symmetric case. The fundamental difference between non-symmetric and symmetric operads is that while in the non-symmetric case the order of inputs is rigid, in the symmetric case it is flexible and governed by the action of the symmetric group.

The abstract operations of a symmetric operad are represented graphically by non-planar operadic trees (see Definition \ref{def:albero_planare_e_non}). Recall that a tree is \emph{non-planar} when no order is fixed for its leaves. As a result, the tree represents an abstract operation $\theta$ without specifying the order of its inputs. To draw such a tree, one must \emph{choose} an ordering for the leaves, i.e., create a \emph{planar embedding}.

An abstract $n$-ary operation $\theta \in \mathcal{P}(n)$ is thus represented by a non-planar $n$-corolla, which is the object representing the operation itself, regardless of the order of its arguments. A specific application, such as $\theta(x_1, \dots, x_n)$, corresponds to a particular planar embedding of the corolla, obtained by labeling the leaves from 1 to $n$. The action of the symmetric group $\sigma \in \mathbb{S}_n$ on the operation, yielding $\theta \cdot \sigma$, corresponds to permuting these labels of the leaves of the planar embedding.

Fundamentally, all possible permutations of the inputs of an operation $\theta$ are not distinct operations, but rather different representations of the same underlying operation, linked together by the action of $\mathbb{S}_n$.

\begin{example}
Let us consider a ternary operation $\theta$, which we  think of as a multiplication. We can apply $\theta$ to 3 elements $a, b, c$, in 6 different ways: 
      \begin{equation}
      \theta_1(a,b,c) = abc,\qquad \theta_2(a,b,c) = cab,\qquad \theta_3(a,b,c) = bca, \quad \ldots \ .      
      \end{equation}
These 6 possibilities derive from a single underlying operation $\theta$. For instance, if we designate $\theta_1$ as the reference operation, the others are obtained by permuting its inputs, such as $\theta_2(a,b,c) = \theta_1(c,a,b)$ and $\theta_3(a,b,c) = \theta_1(b,c,a)$. This means that, having fixed one of the $\theta_i$, the remaining $\theta_i$ are obtained by permuting the inputs.
For example, having fixed $\theta_1$, we have
$\theta_1 = \theta \cdot \text{id}$, $\theta_2 = \theta \cdot \sigma$,
\smallskip
\\ $\theta_3 = \theta \cdot \sigma'$, \ldots,
where $\sigma = \left(\begin{smallmatrix} 1 & 2 & 3 \\ 2 & 3 & 1 \end{smallmatrix}\right)$ and $\sigma' = \left(\begin{smallmatrix} 1 & 2 & 3 \\ 3 & 1 & 2 \end{smallmatrix}\right)$. We can also identify the operations $\theta_1, \ldots, \theta_6$ with corollas:

\begin{figure}[H]
    \centering
    \begin{tikzpicture}[scale=0.9, every node/.style={font=\scriptsize}, thick]

    \node (ou1) at (0,-1.5) {};
    \node (theta1) at (0,-0.75) {$\theta_1$};
    \draw (ou1) -- (theta1);
    \node (a1) at (-0.7,0) {1}; \node (a2) at (0,0) {2}; \node (a3) at (0.7,0) {3};
    \draw (theta1) -- (a1); \draw (theta1) -- (a2); \draw (theta1) -- (a3);
    \node at (0.9, -1) {\large ,};

    \node (ou2) at (2.5, -1.5) {};
    \node (theta2) at (2.5,-0.75) {$\theta_2$};
    \draw (ou2) -- (theta2);
    \node (b1) at (1.8,0) {1}; \node (b2) at (3.2,0) {3}; \node (b3) at (2.5,0) {2};
    \draw (theta2) -- (b1); \draw (theta2) -- (b3); \draw (theta2) -- (b2);
    \node at (3.4, -1) {\large ,};

    \node (ou3) at (5,-1.5) {};
    \node (theta3) at (5,-0.75) {$\theta_3$};
    \draw (ou3) -- (theta3);
    \node (c1) at (4.3,0) {1}; \node (c2) at (5,0) {2}; \node (c3) at (5.7,0) {3};
    \draw (theta3) -- (c1); \draw (theta3) -- (c2); \draw (theta3) -- (c3);
    \node at (5.9, -1) {\large ,};

    \node at (6.5,-1) {\large \dots};
    \node at (7,-1) {\large ,};
    
    \node (ou6) at (8.5,-1.5) {};
    \node (theta6) at (8.5,-0.75) {$\theta_6$};
    \draw (ou6) -- (theta6);
    \node (f1) at (7.8,0) {1}; \node (f2) at (8.5,0) {2}; \node (f3) at (9.2,0) {3};
    \draw (theta6) -- (f1); \draw (theta6) -- (f2); \draw (theta6) -- (f3);
    \node at (9.4, -1) {\large .};

    \end{tikzpicture}
    \caption{Identification of operations with corollas.}
    \label{fig:operations_as_corollas}
\end{figure}
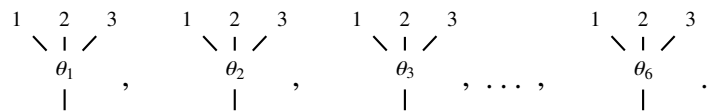

where

\begin{figure}[H]
    \centering
\resizebox{\textwidth}{!}{\begin{tikzpicture}[scale=0.9, every node/.style={font=\scriptsize}, thick]

    \node (ou1) at (0,-1.5) {};
    \node (theta1) at (0,-0.75) {$\theta_1$};
    \draw (ou1) -- (theta1);
    \node (a1) at (-0.7,0) {1}; \node (a2) at (0,0) {2}; \node (a3) at (0.7,0) {3};
    \draw (theta1) -- (a1); \draw (theta1) -- (a2); \draw (theta1) -- (a3);

    \node at (1.2, -0.75) {\large =}; 
    
    \node (ou2) at (2.4,-1.5) {}; 
    \node (theta2) at (2.4,-0.75) {$\theta$};
    \draw (ou2) -- (theta2);
    \node (b1) at (1.7,0) {1}; \node (b2) at (2.4,0) {2}; \node (b3) at (3.1,0) {3}; 
    \draw (theta2) -- (b1); \draw (theta2) -- (b2); \draw (theta2) -- (b3);

    \node at (3.8, -0.75) {\large ,}; 
    \node (ou3) at (5.0,-1.5) {}; 
    \node (theta3) at (5.0,-0.75) {$\theta_2$};
    \draw (ou3) -- (theta3);
    \node (c1) at (4.3,0) {1}; \node (c2) at (5.0,0) {2}; \node (c3) at (5.7,0) {3}; 
    \draw (theta3) -- (c1); \draw (theta3) -- (c2); \draw (theta3) -- (c3);

    \node at (6.2, -0.75) {\large =}; 
    \node (ou4) at (7.4,-1.5) {}; 
    \node (theta4) at (7.4,-0.75) {$\theta$};
    \draw (ou4) -- (theta4);
    \node (d1) at (6.7,0) {2}; \node (d2) at (7.4,0) {3}; \node (d3) at (8.1,0) {1}; 
    \draw (theta4) -- (d1); \draw (theta4) -- (d2); \draw (theta4) -- (d3);

    \node at (8.8, -0.75) {\large ,}; 

    \node (ou5) at (10.0,-1.5) {}; 
    \node (theta5) at (10.0,-0.75) {$\theta_3$};
    \draw (ou5) -- (theta5);
    \node (e1) at (9.3,0) {1}; \node (e2) at (10.0,0) {2}; \node (e3) at (10.7,0) {3}; 
    \draw (theta5) -- (e1); \draw (theta5) -- (e2); \draw (theta5) -- (e3);
    
    \node at (11.2, -0.75) {\large =};

    \node (ou6) at (12.4,-1.5) {}; 
    \node (theta6) at (12.4,-0.75) {$\theta$};
    \draw (ou6) -- (theta6);
    \node (f1) at (11.7,0) {3}; \node (f2) at (12.4,0) {1}; \node (f3) at (13.1,0) {2}; 
    \draw (theta6) -- (f1); \draw (theta6) -- (f2); \draw (theta6) -- (f3);

    \node at (13.8, -0.75) {\large ,}; 
    \node at (14.6, -0.75) {\large $\dots$}; 
\end{tikzpicture}}
    \caption{Non-planar corollas from Figure \ref{fig:operations_as_corollas} identified with $\theta$ and its permutations.}
    \label{fig:enter-label}
\end{figure}
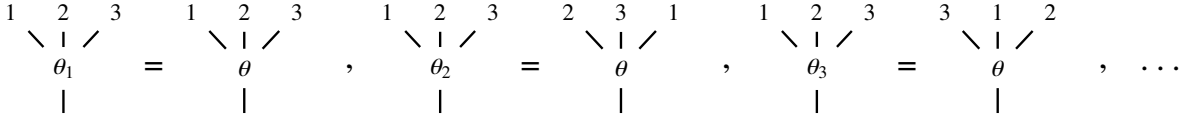
\end{example}

\begin{remark}
Building upon the previous discussion, we can now provide more details about the action of the symmetric group $\mathbb{S}_n$ on the endomorphism operad $\text{End}^{\mathrm{sym}}_V(n) = \Hom(V^{\otimes n}, V)$, which is specified by Equation \eqref{eq: azione su end}.
To understand this action in terms of trees, let us consider an operation $f$ with $n$ inputs. The operation $f \cdot \sigma$ can be represented graphically by re-labeling the leaves of the corolla corresponding to $f$. When we apply $(f \cdot \sigma)$ to a tensor $v_1 \otimes \ldots \otimes v_n$, the inputs are permuted according to $\sigma^{-1}$ before being supplied as arguments to $f$. This is illustrated for $n=3$ with $\sigma = \left(\begin{smallmatrix} 1 & 2 & 3 \\ 2 & 3 & 1 \end{smallmatrix}\right)$:
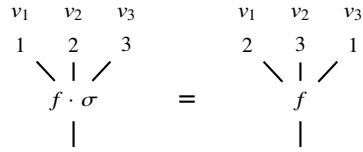
\begin{figure}[H]
    \centering
    \begin{tikzpicture}[every node/.style={font=\scriptsize}, thick]
        \node (root_left) at (0,-1.5) {};
        \node (phi_sigma) at (0,-0.75) {$f \cdot \sigma$};
        \draw (root_left) -- (phi_sigma);

        \node (v1_left) at (-0.7,0) {$1$};
        \node (v2_left) at (0,0) {$2$};
        \node (v3_left) at (0.7,0) {$3$};

        \node  at (-0.7,0.4) {$v_1$};
        \node  at (0,0.4) {$v_2$};
        \node  at (0.7,0.4) {$v_3$};

        \draw (phi_sigma) -- (v1_left);
        \draw (phi_sigma) -- (v2_left);
        \draw (phi_sigma) -- (v3_left);

        \node at (1.5, -0.75) {\large =};

        \node (root_left1) at (3,-1.5) {};
        \node (phi_sigma1) at (3,-0.75) {$f$};
        \draw (root_left1) -- (phi_sigma1);

        \node (v1_left1) at (2.3,0) {$2$};
        \node (v2_left1) at (3,0) {$3$};
        \node (v3_left1) at (3.7,0) {$1$};

        \node  at (2.3,0.4) {$v_1$};
        \node  at (3,0.4) {$v_2$};
        \node  at (3.7,0.4) {$v_3$};

        \draw (phi_sigma1) -- (v1_left1);
        \draw (phi_sigma1) -- (v2_left1);
        \draw (phi_sigma1) -- (v3_left1);

    \end{tikzpicture}
    \caption{Visualization of the action $(f \cdot \sigma)(v_1 \otimes \ldots \otimes v_n) = f(v_{\sigma^{-1}(1)} \otimes \ldots \otimes v_{\sigma^{-1}(n)})$.}
    \label{fig:endomorphism_operad_action}
\end{figure}
The left side represents the application of the permuted operation $(f \cdot \sigma)$ to the ordered inputs $(v_1, v_2, v_3)$. The right side shows that this is equivalent to applying the operation $f$ to the reordered inputs $(v_{\sigma^{-1}(1)}, v_{\sigma^{-1}(2)}, v_{\sigma^{-1}(3)})$. For the permutation $\sigma = \left(\begin{smallmatrix} 1 & 2 & 3 \\ 2 & 3 & 1 \end{smallmatrix}\right)$, its inverse is $\sigma^{-1} = \left(\begin{smallmatrix} 1 & 2 & 3 \\ 3 & 1 & 2 \end{smallmatrix}\right)$. The inputs are therefore permuted to $(v_3, v_1, v_2)$.
\end{remark}

\begin{remark}
Let us return to Definition \ref{def: operad simmetrico} of a symmetric operad, in particular the \emph{equivariance} property, which is given by Equation \eqref{eq:equivarianza}.
To clarify the meaning of this equivariance condition, it is useful to visualize the operations and their compositions through planar embeddings of non-planar trees, i.e., visual representations of the latter where the ordering of the leaves is made explicit by numerical labels.
Let us consider a ternary operation $\theta$, a 5-ary operation $\varphi_1$, a 2-ary operation $\varphi_2$, and a 4-ary operation $\varphi_3$. Let us also consider the permutations
$    \sigma = \left( \begin{smallmatrix} 1 & 2 & 3 \\ 2 & 3 & 1 \end{smallmatrix} \right) \in \mathbb{S}_3, \ \pi_1 = \left( \begin{smallmatrix} 1 & 2 & 3 & 4 & 5 \\ 5 & 1 & 2 & 3 & 4 \end{smallmatrix} \right) \in \mathbb{S}_5, \ \pi_2 = \left( \begin{smallmatrix} 1 & 2 \\ 2 & 1 \end{smallmatrix} \right) \in \mathbb{S}_2, \ \pi_3 = \left( \begin{smallmatrix} 1 & 2 & 3 & 4 \\ 1 & 4 & 2 & 3 \end{smallmatrix} \right)\in \mathbb{S}_4.$

The {left-hand side} of Equation \eqref{eq:equivarianza}, i.e., $(\theta \cdot \sigma) \circ (\varphi_{\sigma(1)} \cdot \pi_{\sigma(1)}, \varphi_{\sigma(2)} \cdot \pi_{\sigma(2)}, \varphi_{\sigma(3)} \cdot \pi_{\sigma(3)})$, corresponds to the following tree diagram.

\begin{figure}[H]
    \centering
    \resizebox{\textwidth}{!}{
        \begin{tikzpicture}[every node/.style={font=\scriptsize}, thick]
            \node (phi2pi2) at (-3, 1.5) {$\varphi_2 \cdot \pi_2$};
            \node (phi3pi3) at (0, 1.5) {$\varphi_3 \cdot \pi_3$};
            \node (phi1pi1) at (3, 1.5) {$\varphi_1 \cdot \pi_1$};

            \node at (-3.5, 2.7) {1}; 
            \node at (-2.5, 2.7) {2};
            \draw (phi2pi2) -- (-3.5, 2.5);
            \draw (phi2pi2) -- (-2.5, 2.5);
            
            \node at (-1, 2.7) {1};
            \node at (-0.4, 2.7) {2}; 
            \node at (0.2, 2.7) {3}; 
            \node at (0.8, 2.7) {4};
            \draw (phi3pi3) -- (-1, 2.5);
            \draw (phi3pi3) -- (-0.4, 2.5);
            \draw (phi3pi3) -- (0.2, 2.5);
            \draw (phi3pi3) -- (0.8, 2.5);

            \node at (2.2, 2.7) {1}; 
            \node at (2.6, 2.7) {2}; 
            \node at (3, 2.7) {3}; 
            \node at (3.4, 2.7) {4}; 
            \node at (3.8, 2.7) {5};
            \draw (phi1pi1) -- (2.2, 2.5);
            \draw (phi1pi1) -- (2.6, 2.5);
            \draw (phi1pi1) -- (3, 2.5);       
            \draw (phi1pi1) -- (3.4, 2.5);       
            \draw (phi1pi1) -- (3.8, 2.5);

            \node (int_node2) at (-3, 0.5) {1};
            \node (int_node3) at (0, 0.5) {2};
            \node (int_node1) at (3, 0.5) {3};

            \draw (phi2pi2) -- (int_node2);
            \draw (phi3pi3) -- (int_node3);
            \draw (phi1pi1) -- (int_node1);

            \node (theta_sigma) at (0, -0.5) {$\theta \cdot \sigma$};
            \draw (int_node2) -- (theta_sigma);
            \draw (int_node3) -- (theta_sigma);
            \draw (int_node1) -- (theta_sigma);

            \node (root) at (0, -1.5) {};
            \draw (theta_sigma) -- (root);

            \node at (4.2,0.5) {\large =};

            \begin{scope}[shift={(11.2,0)}]

                \node (phi2pi2b) at (-6, 1.5) {$\varphi_2$};
                \node (phi3pi3b) at (-3, 1.5) {$\varphi_3$};
                \node (phi1pi1b) at (0, 1.5) {$\varphi_1$};

                \node at (-6.5, 2.7) {2}; 
                \node at (-5.5, 2.7) {1};
                \draw (phi2pi2b) -- (-6.5, 2.5);
                \draw (phi2pi2b) -- (-5.5, 2.5);

                \node at (-4, 2.7) {1};
                \node at (-3.4, 2.7) {4}; 
                \node at (-2.8, 2.7) {2}; 
                \node at (-2.2, 2.7) {3};
                \draw (phi3pi3b) -- (-4, 2.5);
                \draw (phi3pi3b) -- (-3.4, 2.5);
                \draw (phi3pi3b) -- (-2.8, 2.5);
                \draw (phi3pi3b) -- (-2.2, 2.5);

                \node at (-0.8, 2.7) {5}; 
                \node at (-0.4, 2.7) {1}; 
                \node at (0, 2.7) {2}; 
                \node at (0.4, 2.7) {3}; 
                \node at (0.8, 2.7) {4};
                \draw (phi1pi1b) -- (-0.8, 2.5);
                \draw (phi1pi1b) -- (-0.4, 2.5);
                \draw (phi1pi1b) -- (0, 2.5);       
                \draw (phi1pi1b) -- (0.4, 2.5);       
                \draw (phi1pi1b) -- (0.8, 2.5);

                \node (int_node2b) at (-6, 0.5) {2};
                \node (int_node3b) at (-3, 0.5) {3};
                \node (int_node1b) at (0, 0.5) {1};

                \draw (phi2pi2b) -- (int_node2b);
                \draw (phi3pi3b) -- (int_node3b);
                \draw (phi1pi1b) -- (int_node1b);

                \node (theta_sigma_b) at (-3, -0.5) {$\theta$};
                \draw (int_node2b) -- (theta_sigma_b);
                \draw (int_node3b) -- (theta_sigma_b);
                \draw (int_node1b) -- (theta_sigma_b);

                \node (rootb) at (-3, -1.5) {};
                \draw (theta_sigma_b) -- (rootb);
            \end{scope}
        \end{tikzpicture}
    } 
    \caption{Representation of the left-hand side of Equation \eqref{eq:equivarianza}.}
    \label{fig:lhs_equivariance}
\end{figure}
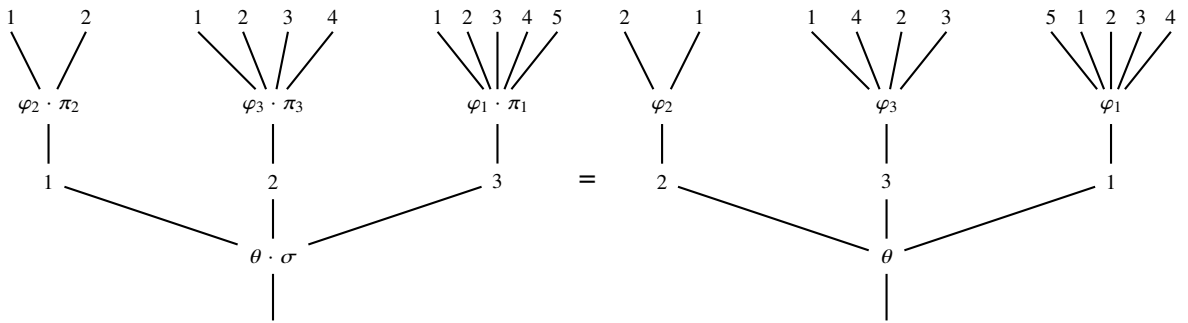

The left diagram represents the result of the composition in which the operations $\varphi_i$ (modified by $\pi_i$) are inserted into the reordered inputs of $\theta$ (modified by $\sigma$). 
The final tree can be identified with an 11-corolla with a specific ordering of the leaves resulting from the concatenation of the permuted inputs. 

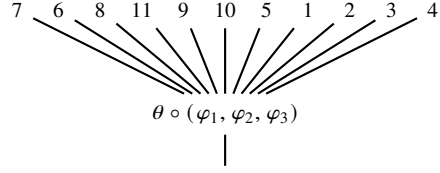
\begin{figure}[H]
    \centering
    \begin{tikzpicture}[scale=1.1, every node/.style={font=\scriptsize}, thick]

        \node (corolla_root) at (8, -1.5) {};
        \node (corolla_op) at (8, -0.75) {$\theta \circ (\varphi_1,\varphi_2,\varphi_3)$};
        \draw (corolla_root) -- (corolla_op);

        \node (l1) at (5.5, 0.5) {7}; 
        \node (l2) at (6, 0.5) {6}; 
        \node (l3) at (6.5, 0.5) {8};
        \node (l4) at (7, 0.5) {11};  
        \node (l5) at (7.5, 0.5) {9}; 
        \node (l6) at (8.0, 0.5) {10}; 
        \node (l7) at (8.5, 0.5) {5};  
        \node (l8) at (9, 0.5) {1};  
        \node (l9) at (9.5, 0.5) {2};  
        \node (l10) at (10, 0.5) {3}; 
        \node (l11) at (10.5, 0.5) {4}; 

        \draw (corolla_op) -- (l1);
        \draw (corolla_op) -- (l2);
        \draw (corolla_op) -- (l3);
        \draw (corolla_op) -- (l4);
        \draw (corolla_op) -- (l5);
        \draw (corolla_op) -- (l6);
        \draw (corolla_op) -- (l7);
        \draw (corolla_op) -- (l8);
        \draw (corolla_op) -- (l9);
        \draw (corolla_op) -- (l10);
        \draw (corolla_op) -- (l11);

    \end{tikzpicture}
    \caption{Representation of the 11-corolla resulting from the composition.}
    \label{fig:11_corolla}
\end{figure}

For the {right-hand side} of Equation 
\eqref{eq:equivarianza}, i.e., $\bigl(\theta \circ (\varphi_1, \varphi_2, \varphi_3)\bigr) \cdot \bigl(\sigma \circ (\pi_1, \pi_2, \pi_3)\bigr)$, we first calculate the permutation $\sigma \circ (\pi_1, \pi_2, \pi_3)$. This is the permutation given by the diagram: 

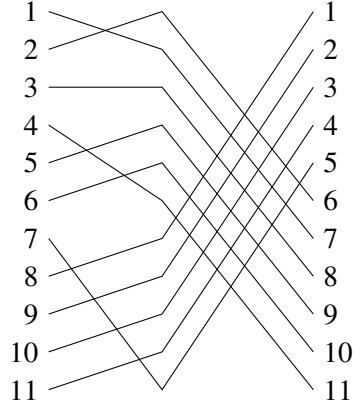
\begin{center}
\begin{tikzpicture}[scale=0.5, baseline={([yshift=-.5ex]current bounding box.center)}]

    \foreach \N in {1,...,11} {
        \node[left] at (0, -\N) {\N};
        \node[right] at (7, -\N) {\N};
    }

    \draw[thin] (0, -1) -- (3, -2) -- (7, -7);   
    \draw[thin] (0, -2) -- (3, -1) -- (7, -6);  
    \draw[thin] (0, -3) -- (3, -3) -- (7, -8);   
    \draw[thin] (0, -4) -- (3, -6) -- (7, -11);   
    \draw[thin] (0, -5) -- (3, -4) -- (7, -9);   
    \draw[thin] (0, -6) -- (3, -5) -- (7, -10);   
    \draw[thin] (0, -7) -- (3, -11) -- (7, -5);   
    \draw[thin] (0, -8) -- (3, -7) -- (7, -1);   
    \draw[thin] (0, -9) -- (3, -8) -- (7, -2);   
    \draw[thin] (0, -10) -- (3, -9) -- (7, -3);
    \draw[thin] (0, -11) -- (3, -10) -- (7, -4);
\end{tikzpicture}
\end{center}
i.e.
\begin{equation}
\sigma \circ (\pi_1, \pi_2, \pi_3) = \left(\begin{array}{ccccccccccc} 1 & 2 & 3 & 4 & 5 & 6 & 7 & 8 & 9 & 10 & 11 \\ 7 & 6 & 8 & 11 & 9 & 10 & 5 & 1 & 2 & 3 & 4 \end{array}\right).
\end{equation}
The application of this permutation to $\theta \circ (\varphi_1, \varphi_2, \varphi_3)$ coincides with Figure \ref{fig:11_corolla}, as shown below.

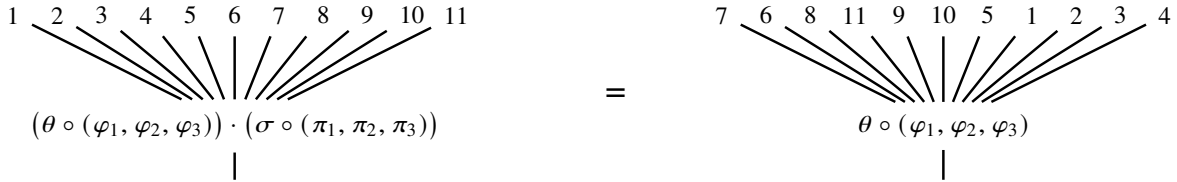
\begin{figure}[H]
    \centering

    \resizebox{\textwidth}{!}{
        \begin{tikzpicture}[every node/.style={font=\scriptsize}, thick]

            \node (corolla_root) at (8, -1.5) {};
            \node (corolla_op) at (8, -0.75) {$\bigl(\theta \circ (\varphi_1,\varphi_2,\varphi_3)\bigr) \cdot \bigl(\sigma \circ (\pi_1,\pi_2,\pi_3)\bigr)$};
            \draw (corolla_root) -- (corolla_op);

            \node (l1) at (5.5, 0.5) {1};
            \node (l2) at (6, 0.5) {2};
            \node (l3) at (6.5, 0.5) {3};
            \node (l4) at (7, 0.5) {4};
            \node (l5) at (7.5, 0.5) {5};
            \node (l6) at (8.0, 0.5) {6};
            \node (l7) at (8.5, 0.5) {7};
            \node (l8) at (9, 0.5) {8};
            \node (l9) at (9.5, 0.5) {9};
            \node (l10) at (10, 0.5) {10};
            \node (l11) at (10.5, 0.5) {11};

            \draw (corolla_op) -- (l1);
            \draw (corolla_op) -- (l2);
            \draw (corolla_op) -- (l3);
            \draw (corolla_op) -- (l4);
            \draw (corolla_op) -- (l5);
            \draw (corolla_op) -- (l6);
            \draw (corolla_op) -- (l7);
            \draw (corolla_op) -- (l8);
            \draw (corolla_op) -- (l9);
            \draw (corolla_op) -- (l10);
            \draw (corolla_op) -- (l11);

            \node at (12.3,-0.4) {\large =};

            \node (corolla_root2) at (16, -1.5) {};
            \node (corolla_op2) at (16, -0.75) {$\theta \circ (\varphi_1,\varphi_2,\varphi_3)$};
            \draw (corolla_root2) -- (corolla_op2);
            
            \node (r1) at (13.5, 0.5) {7};
            \node (r2) at (14, 0.5) {6};
            \node (r3) at (14.5, 0.5) {8};
            \node (r4) at (15, 0.5) {11};
            \node (r5) at (15.5, 0.5) {9};
            \node (r6) at (16.0, 0.5) {10};
            \node (r7) at (16.5, 0.5) {5};
            \node (r8) at (17, 0.5) {1};
            \node (r9) at (17.5, 0.5) {2};
            \node (r10) at (18, 0.5) {3};
            \node (r11) at (18.5, 0.5) {4};
            
            \draw (corolla_op2) -- (r1);
            \draw (corolla_op2) -- (r2);
            \draw (corolla_op2) -- (r3);
            \draw (corolla_op2) -- (r4);
            \draw (corolla_op2) -- (r5);
            \draw (corolla_op2) -- (r6);
            \draw (corolla_op2) -- (r7);
            \draw (corolla_op2) -- (r8);
            \draw (corolla_op2) -- (r9);
            \draw (corolla_op2) -- (r10);
            \draw (corolla_op2) -- (r11);
        \end{tikzpicture}
    } 
    \label{fig:11_corolla_uguale} 
\end{figure}
\end{remark}

\subsubsection{The Operad $\mathsf{Com}$}

Within the framework of symmetric operads, the commutative operad, commonly denoted by $\mathsf{Com}$, plays a very important role. In contrast to the operad $\mathsf{As}$, which models associative algebras, $\mathsf{Com}$ is specifically designed to encapsulate, in operadic form, the structure of (associative) commutative algebras: a $\mathsf{Com}$-algebra is simply a commutative algebra.

The interest in this operad stems not only from it being a paradigmatic example of a symmetric operad, but also from its role in providing an operadic description of an important category of algebraic structures in which every \(n\)-ary operation $\mu$ is completely symmetric, i.e., \emph{invariant under any permutation of the inputs}:
\begin{equation}
\mu(x_1, x_2, \dots, x_n) = \mu(x_{\sigma(1)}, x_{\sigma(2)}, \dots, x_{\sigma(n)})
\end{equation}
for every $\sigma \in \mathbb{S}_n$. This corresponds to saying that the symmetric group \(\mathbb{S}_n\) acts \emph{trivially} on \(\mathcal{P}(n)\), the set of \(n\)-ary operations of the operad. All possible permutations of the arguments correspond to the same operation.

We now provide an explicit definition of $\mathsf{Com}$ as a symmetric operad, detailing its components, composition rules, and unit element. Finally, we  prove that $\mathsf{Com}$-algebras are in one-to-one correspondence with commutative algebras.

\begin{definition}
\label{def:operad-com}
We define $\mathsf{Com}$ as follows:
\begin{enumerate}
    \item $\mathsf{Com}(0) \coloneqq 0$.
    \item For each $n \in \mathbb{N}$, $\mathsf{Com}(n) \coloneqq \mathbb{K}\{\mu_n\} \cong \mathbb{K}$.
    \item The action of the symmetric group $\mathbb{S}_n$ on $\mathsf{Com}(n)$ is the {trivial} action. This means that for every $\mu_n \in \mathsf{Com}(n)$ and every $\sigma \in \mathbb{S}_n$, we have $\mu_n \cdot \sigma \coloneqq \mu_n$. 
    \item The composition $\gamma_{k_1, \ldots, k_n}: \mathsf{Com}(n) \otimes \mathsf{Com}(k_1) \otimes \ldots \otimes \mathsf{Com}(k_n) \longrightarrow \mathsf{Com}(k_1 + \cdots + k_n)$ is the canonical isomorphism given by the product in $\mathbb{K}$:
    \begin{equation}
    \underbrace{\mathbb{K} \otimes \cdots \otimes \mathbb{K}}_{n+1} \longrightarrow \mathbb{K}, \qquad a_1 \otimes \cdots\otimes a_n \mapsto a_1 \cdots a_n.
    \end{equation}
    \end{enumerate}
\end{definition}

\begin{lemma}
The operad $\mathsf{Com}$ of Definition \ref{def:operad-com} is well-defined.
\end{lemma}

\begin{proof}
We verify the following points based on the definition:
\begin{itemize}
    \item For each $n \in \mathbb{N}_0$, $\mathsf{Com}(n)$ is a vector space by construction.
    \item There exists a unit element $1 \in \mathsf{Com}(1)$ (the unit of the field $\mathbb{K}$ for $\mu_1$), such that for every $\mu \in \mathsf{Com}(n) \cong \mathbb{K}$ the unit laws hold:
    \begin{equation}
    1 \circ \mu = 1 \cdot \mu = \mu \qquad \text{and} \qquad \mu \circ (1 \otimes \dots \otimes 1) = \mu \cdot 1 \cdots 1 = \mu.
    \end{equation}
    \item The composition maps are associative, since the product in the field $\mathbb{K}$ is.
    \item The action of the symmetric group is trivial ($\mu_n \cdot \sigma = \mu_n$) and is compatible with composition. This follows from the fact that each component $\mathsf{Com}(n)$ is one-dimensional and the trivial action of the symmetric group does not alter its elements, which guarantees compatibility with composition and equivariance.
\end{itemize}
\end{proof}

\begin{remark}
    In the binary case (\(n = 2\)), in a {symmetric operad}, two versions of the same binary product are distinguished, differing only in the order of the inputs, represented by the following planar embeddings of a non-planar 2-corolla.
\begin{figure}[H]
    \centering
    \begin{tikzpicture}[every node/.style={font=\scriptsize}, thick]

        \node (out1) at (0,-1.5) {};
        \node (op1) at (0,-0.75) {\(\mu\)};
        \draw (out1) -- (op1);
        \node (l1_1) at (-0.5,0) {1};
        \node (l1_2) at (0.5,0) {2};
        \draw (op1) -- (l1_1);
        \draw (op1) -- (l1_2);

        \node at (1.5, -0.75) {\Large\(\neq\)};

        \node (out2) at (3,-1.5) {};
        \node (op2) at (3,-0.75) {\(\mu\)};
        \draw (out2) -- (op2);
        \node (l2_1) at (2.5,0) {2}; 
        \node (l2_2) at (3.5,0) {1}; 
        \draw (op2) -- (l2_1);
        \draw (op2) -- (l2_2);

    \end{tikzpicture}
    \caption{Two distinct planar embeddings of a non-planar 2-corolla representing a binary product $\mu$.}
\end{figure}
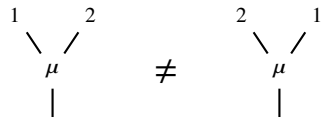
In \(\mathsf{Com}\), however, these two configurations are {identified}, since the multiplication is commutative:
\begin{equation}
    \mu(a, b) = \mu(b, a).
\end{equation}
This is reflected in the fact that the same trees now represent {the same operation}.

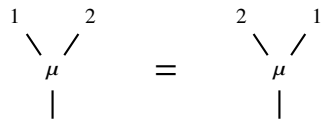
\begin{figure}[H]
    \centering
    \begin{tikzpicture}[every node/.style={font=\scriptsize}, thick]

        \node (out1) at (0,-1.5) {};
        \node (op1) at (0,-0.75) {\(\mu\)};
        \draw (out1) -- (op1);
        \node (l1_1) at (-0.5,0) {1};
        \node (l1_2) at (0.5,0) {2};
        \draw (op1) -- (l1_1);
        \draw (op1) -- (l1_2);

        \node at (1.5, -0.75) {\Large\(=\)};

        \node (out2) at (3,-1.5) {};
        \node (op2) at (3,-0.75) {\(\mu\)};
        \draw (out2) -- (op2);
        \node (l2_1) at (2.5,0) {2}; 
        \node (l2_2) at (3.5,0) {1}; 
        \draw (op2) -- (l2_1);
        \draw (op2) -- (l2_2);

    \end{tikzpicture}
    \caption{In \(\mathsf{Com}\), permutations of the inputs define the same operation.}
\end{figure}
\end{remark}

We now establish the main result concerning the operad $\mathsf{Com}$: the category of $\mathsf{Com}$-algebras is isomorphic to the category of commutative associative $\mathbb{K}$-algebras.

Our proof strategy is to construct a pair of functors, one in each direction, and then show that they are mutually inverse. The following two propositions are dedicated to the rigorous construction of these functors.

\begin{proposition}
\label{prop:functor_F_Com}
The correspondence $F : \emph{\texttt{Alg}}(\mathsf{Com}) \to \emph{\texttt{CommAlg}}_\mathbb{K}$, defined on objects by 
\begin{equation}
    F(V,\Phi) \coloneqq \bigl(V, m_V \coloneqq \Phi_2(\mu_2)\bigr)
\end{equation}
and on morphisms by
\begin{equation}
    F(f) \coloneqq f,
\end{equation}
is a well-defined functor.
\end{proposition}

\begin{proof}
We verify that $F$ is well-defined on objects and morphisms, and that it satisfy the functorial axioms.
\begin{enumerate}
    \item \emph{Well-definedness on objects.} We show that for any $\mathsf{Com}$-algebra $(V, \Phi)$, the pair $(V, m_V)$ is a commutative associative algebra. By construction, $m_V$ is a bilinear map. Associativity follows from the same argument as for $\mathsf{As}$-algebras (Proposition \ref{prop:functor_F_As}). For commutativity, we note that the action of the transposition $\tau \in \mathbb{S}_2$ on $\mathsf{Com}(2)$ is trivial, so $\mu_2 \cdot \tau = \mu_2$. Since $\Phi$ is a morphism of symmetric operads, it preserves this action:
    \begin{equation}
    \Phi_2(\mu_2) = \Phi_2(\mu_2 \cdot \tau) = \Phi_2(\mu_2) \cdot \tau.
    \end{equation}
    From the definition of the action on $\End_V^{\mathrm{sym}}(2)$ (Example \ref{es: s end}), this implies that for every $v_1, v_2 \in V$:
    \begin{equation}
       m_V(v_1 \otimes v_2) = (m_V \cdot \tau)(v_1 \otimes v_2) = m_V(v_2 \otimes v_1). 
    \end{equation}
    Thus, $m_V$ is commutative.
    \item \emph{Well-definedness on morphisms.} The proof is identical to that of Proposition \ref{prop:functor_F_As}, as the compatibility condition \eqref{eq: condizione morf di P-algebre} is the same.
    \item \emph{Functorial properties.} The axioms for identity and composition are obvious, analogously to what was seen in Proposition \ref{prop:functor_F_As}.
\end{enumerate}

Therefore, $F$ is a well-defined functor.
\end{proof}

\begin{proposition}
\label{prop:functor_G_Com}
The correspondence $G : \emph{\texttt{CommAlg}}_\mathbb{K} \longrightarrow \emph{\texttt{Alg}}(\mathsf{Com})$, defined on objects by 
\begin{equation}
  G(A, m_A) \coloneqq (A, \Psi)  
\end{equation}
and on morphisms by
\begin{equation}
    G(g) \coloneqq g,
\end{equation}
where $\Psi$ is the canonical $\mathsf{Com}$-algebra structure on $A$ induced by the multiplication $m_A$ (as in Proposition \ref{prop:functor_G_As}), is a well-defined functor.
\end{proposition}

\begin{proof}
We need to verify that $G$ is well-defined on objects and morphisms, and that it satisfies the functorial axioms.
\begin{enumerate}
    \item \emph{Well-definedness on objects.} For any commutative algebra $(A, m_A)$, we show that the canonically induced map $\Psi: \mathsf{Com} \to \End_A^{\mathrm{sym}}$ is a well-defined morphism of symmetric operads. The underlying morphism of non-symmetric operads is defined as in Proposition \ref{prop:functor_G_As}, and its compatibility is guaranteed by the associativity of $m_A$. We only need to verify the compatibility with the symmetric group action, i.e., that $\Psi_n(\mu_n \cdot \sigma) = \Psi_n(\mu_n) \cdot \sigma$ for all $\sigma \in \mathbb{S}_n$.
    Since the action on $\mathsf{Com}(n)$ is trivial, the left-hand side is $\Psi_n(\mu_n)$. For the right-hand side, the operation $\Psi_n(\mu_n)$ is an iterated composition of the multiplication $m_A$. Since $m_A$ is both associative and commutative, any such iterated composition is a symmetric operation. This means it is invariant under any permutation of its inputs, so $\Psi_n(\mu_n) \cdot \sigma = \Psi_n(\mu_n)$. Therefore, $G$ is well-defined on objects.
    \item \emph{Well-definedness on morphisms.} Let $g : (A, m_A) \to (B, m_B)$ be a morphism of commutative algebras. We show it is also a morphism of $\mathsf{Com}$-algebras. This requires verifying the identity $g \circ \Psi^A_n(\mu_n) = \Psi^B_n(\mu_n) \circ g^{\otimes n}$. As this condition does not depend on the symmetric structure, the proof is identical to that for the non-symmetric case, as established in Proposition \ref{prop:functor_G_As}.
    \item \emph{Functorial properties.} The functorial axioms for identity and composition hold, since the action of $G$ on morphisms is the identity map $\bigl( G(g) \coloneqq g \bigr)$.
\end{enumerate}

Therefore, $G$ is a well-defined functor.
\end{proof}

\begin{theorem}
\label{th:iso_com}
The category of \(\mathsf{Com}\)-algebras is isomorphic to the category of commutative associative algebras.
\end{theorem}

\begin{proof}
The functors $F$ (Proposition \ref{prop:functor_F_Com}) and $G$ (Proposition \ref{prop:functor_G_Com}) provide the isomorphism. The verification that $F \circ G = \mathrm{Id}_{\texttt{CommAlg}_{\mathbb{K}}}$ and $G \circ F = \mathrm{Id}_{\texttt{Alg}(\mathsf{Com})}$ is identical to the argument in the proof of Theorem \ref{th:as_isomorphism}.
\end{proof}

\subsubsection{The Operad $\mathsf{uCom}$}
\label{sec:operad-ucom}

The operad $\mathsf{uCom}$ is a variant of the operad $\mathsf{Com}$. As with $\mathsf{Com}$, the interest in $\mathsf{uCom}$ lies in the fact that the latter provides an operadic description of an important class of algebraic structures, namely unital commutative algebras.

We can explicitly construct $\mathsf{uCom}$ as a symmetric operad with the same components as $\mathsf{Com}$, but which, in addition, includes the operation that represents the unit.

\begin{definition}
\label{def:ucom}
We define $\mathsf{uCom}$ as follows:
\begin{enumerate}
    \item We set $\mathsf{uCom}(0) \coloneqq \mathbb{K}\{\eta\} \cong \mathbb{K}$.
    \item For each $n \ge 1$, we set $\mathsf{uCom}(n) \coloneqq \mathbb{K}\{\mu_n\} \cong \mathbb{K}$.
    \item The action of the symmetric group $\mathbb{S}_n$ on $\mathsf{uCom}(n)$ is the trivial action.
    \item The compositions are given by the usual product of $n+1$ elements in $\mathbb{K}$, analogously to the case of $\mathsf{Com}$ in Definition \ref{def:operad-com}.
\end{enumerate}
\end{definition}

\begin{remark}
Compared to the case of $\mathsf{Com}$, the only difference in Definition \ref{def:ucom} is that $\mathsf{uCom}(0) \cong \mathbb{K}$, a choice that corresponds to the presence of the unit. In this sense, $\mathsf{uCom}$ can be seen as a ``unitalization'' of $\mathsf{Com}$, obtained by formally adding a nullary operation compatible with the binary multiplication.
\end{remark}

\begin{lemma}
The operad $\mathsf{uCom}$ of Definition \ref{def:ucom} is well-defined.
\end{lemma}

\begin{proof}
The proof is analogous to the one for $\mathsf{uAs}$ (Lemma \ref{le: uas ben definito}). We observe that by construction $\mathsf{uCom}(n)$ is a vector space over $\mathbb{K}$ for every $n \in \mathbb{N}_0$, that the unit element $1 \in \mathsf{uCom}(1)$ satisfies the unit laws, and that the composition maps are associative. Compatibility with the trivial action of the symmetric group is obvious.
\end{proof}

It follows from this definition that an algebra over $\mathsf{uCom}$ is a vector space $V$ equipped with a bilinear map $m_V: V \otimes V \rightarrow V$ that is associative and commutative, and a unit element $1_V \in V$, such that $m_V(1_V, x) = x = m_V(x, 1_V)$ for every $x \in V$. 

\begin{theorem}
The category of $\mathsf{uCom}$-algebras is isomorphic to the category of unital commutative associative algebras. The isomorphisms are the same as in Theorem \ref{th:iso_com}.
\end{theorem}

\begin{proof}
A rigorous proof can be done by replicating that of Theorem \ref{th:uas} with the necessary adjustments.
\end{proof}

\begin{remark}
\label{rem:symmetric_group_action_encoding}
We observe that the action of the symmetric group encodes the symmetries of the operations of the considered algebraic structure. More precisely, commutativity corresponds to the \emph{trivial action}, since any permutation of the factors yields the same result. If no symmetry is present, we get the \emph{regular action}, since, in general, any permutation of the factors produces a different result. Following this idea, anticommutativity corresponds to the \emph{sign action} (see Example \ref{ex:representations}), while the complete absence of symmetry is captured by the framework of non-symmetric operads.
\end{remark}

\subsection{Relationship Between Symmetric and Non-Symmetric Operads}

Having defined both symmetric and non-symmetric operads, it is natural to formalize the relationship between them. This connection is best described by pairs of adjoint functors between their respective categories, $\texttt{SymOp}$ and $\texttt{NonSymOp}$. In this section, we detail two fundamental adjunctions that link these categories, each providing a different perspective on their interplay.

\subsubsection{The Free-Forgetful Adjunction}
The most standard way to relate the two categories involves a pair of functors: one ``forgets'' the symmetric structure, while the other provides a canonical way to ``symmetrize'' a non-symmetric operad.

\begin{remark}
It is important to note that, although in this work the operad $\mathsf{As}$ has been defined as a non-symmetric operad, it is also possible to define a symmetric operad that encodes associative algebras. Such an operad, denoted by $\mathsf{Ass}$, is constructed as follows:
\begin{itemize}
    \item $\mathsf{Ass}(0) := 0$.
    \item For each $n \ge 1$, $\mathsf{Ass}(n) := \mathbb{K}[\mathbb{S}_n]$.
    \item The action of the symmetric group $\mathbb{S}_n$ on $\mathsf{Ass}(n)$ is the right regular action (see Example \ref{ex:representations}).
    \item The composition $\gamma_{k_1, \ldots, k_n}: \mathsf{Ass}(n) \otimes \mathsf{Ass}(k_1) \otimes \ldots \otimes \mathsf{Ass}(k_n) \longrightarrow \mathsf{Ass}(k_1 + \ldots + k_n)$ on the basis elements is the usual composition in the symmetries operad.
    \item The unit element is $1 \in \mathsf{Ass}(1)$ (the identity in $\mathbb{S}_1$). 
\end{itemize}
This means that each permutation of the inputs is treated as a distinct operation, which reflects the fact that a generic associative algebra is not commutative.
\end{remark}

Indeed, we can say more:

\begin{remark}[The symmetrization functor]
\label{rem:symmetrization}
We note that there is a canonical procedure, known as \emph{symmetrization}, which allows constructing a symmetric operad from a non-symmetric one. This construction defines a functor, which we  denote by $\mathsf{Sym}$, from the category of non-symmetric operads to the category of symmetric operads.

Given a non-symmetric operad $\mathcal{P}$, the corresponding symmetric operad $\mathsf{Sym}(\mathcal{P})$ is defined as follows:

\begin{enumerate}
    \item \emph{Spaces of operations.} For each $n \in \mathbb{N}_0$, the space of $n$-ary operations of $\mathsf{Sym}(\mathcal{P})$ is given by the tensor product of the original space with the group algebra of the symmetric group:
    \begin{equation}
    \mathsf{Sym}(\mathcal{P})(n) \coloneqq \mathcal{P}(n) \otimes_{\mathbb{K}} \mathbb{K}[\mathbb{S}_n].
    \end{equation}
    An element in this space is a linear combination of tensors of the form $p \otimes \sigma$, where $p \in \mathcal{P}(n)$ and $\sigma \in \mathbb{S}_n$.

    \item \emph{Action of the symmetric group.} The right action of $\tau \in \mathbb{S}_n$ on $\mathsf{Sym}(\mathcal{P})(n)$ is defined by acting only on the second factor via multiplication in the group algebra (right regular action):
    \begin{equation}
    (p \otimes \sigma) \cdot \tau \coloneqq p \otimes (\sigma \tau).
    \end{equation}

    \item \emph{Composition.} The composition law in $\mathsf{Sym}(\mathcal{P})$ is a combination of the original composition $\gamma_{\mathcal{P}}$ and the composition of permutations. Given $\theta = p \otimes \sigma \in \mathsf{Sym}(\mathcal{P})(n)$ and $\theta_i = p_i \otimes \tau_i \in \mathsf{Sym}(\mathcal{P})(k_i)$ for $i=1,\ldots,n$, their composition is defined by:
    \begin{equation}
    \gamma_{\mathsf{Sym}(\mathcal{P})}(\theta; \theta_1, \ldots, \theta_n) \coloneqq \gamma_{\mathcal{P}}(p; p_{\sigma(1)}, \ldots, p_{\sigma(n)}) \otimes \bigl(\sigma \circ (\tau_1, \ldots, \tau_n)\bigr).
    \end{equation}
    Note how the permutation $\sigma$ associated with the main operation $p$ acts by permuting the operations $p_i$ that are being composed.
\end{enumerate}
\end{remark}

\begin{example}
    The most paradigmatic example of symmetrization is the application of the $\mathsf{Sym}$ functor to the operad $\mathsf{As}$, which yields the operad $\mathsf{Ass}$.
Recall that $\mathsf{As}(n) \cong \mathbb{K}$ for $n \ge 1$, generated by $\mu_n$. Applying the $\mathsf{Sym}$ construction yields:
\begin{equation}
\mathsf{Sym}(\mathsf{As})(n) = \mathsf{As}(n) \otimes \mathbb{K}[\mathbb{S}_n] \cong \mathbb{K} \otimes \mathbb{K}[\mathbb{S}_n] \cong \mathbb{K}[\mathbb{S}_n].
\end{equation}
This is precisely the vector space that defines $\mathsf{Ass}(n)$. The action of $\mathbb{S}_n$ is the right regular action, which coincides with the action on $\mathsf{Ass}$. Finally, the composition of elements $(\mu_n \otimes \sigma)$ and $(\mu_{k_i} \otimes \tau_i)$ in $\mathsf{Sym}(\mathsf{As})$ is given by:
\begin{align}
&\gamma_{\mathsf{As}}(\mu_n; \mu_{k_{\sigma(1)}}, \ldots, \mu_{k_{\sigma(n)}}) \otimes \bigl(\sigma \circ (\tau_1, \ldots, \tau_n)\bigr) = \nonumber \\
&=\mu_{k_1+\dots+k_n} \otimes \bigl(\sigma \circ (\tau_1, \ldots, \tau_n)\bigr).
\end{align}
Identifying the element $\mu_k \otimes \pi$ with the element $\pi \in \mathbb{K}[\mathbb{S}_k]$, this composition rule corresponds exactly to the block composition of permutations that defines the structure of the operad $\mathsf{Ass}$.

This construction, therefore, formalizes the idea of endowing a non-symmetric operad with a coherent symmetric structure.
\end{example}

\begin{remark}[The forgetful functor]
\label{rem:forgetful_functor}
As a non-symmetric operad can be canonically symmetrized to yield a symmetric one, it is also possible to reverse the process: given a symmetric operad, one can obtain a non-symmetric operad by simply ``forgetting'' the structure given by the action of the symmetric groups.

This operation defines a \emph{forgetful functor} $U$ from the category of symmetric operads to that of non-symmetric operads.
\begin{equation}
U: {\texttt{SymOp}} \longrightarrow {\texttt{NonSymOp}}.
\end{equation}
Given a symmetric operad $\mathcal{P}_{\text{sym}}$, the corresponding non-symmetric operad, $U(\mathcal{P}_{\text{sym}})$, is defined as follows:
\begin{itemize}
    \item \emph{Spaces of operations:} The vector spaces are the same as in the starting operad. For each $n \in \mathbb{N}_0$:
    \begin{equation}
    U(\mathcal{P}_{\text{sym}})(n) \coloneqq \mathcal{P}_{\text{sym}}(n).
    \end{equation}
    \item \emph{Composition and Unit:} The composition maps and the unit element are the same as in the starting operad.
\end{itemize}
Concretely, this functor acts by retaining all the data that define a non-symmetric operad (the spaces $\mathcal{P}(n)$, the composition $\gamma$, and the unit $1$) while ignoring the action of $\mathbb{S}_n$. Since a symmetric operad must already satisfy the associativity and unit axioms (Definition \ref{def: operad simmetrico}), the resulting structure is automatically a well-defined non-symmetric operad.
\end{remark}

\begin{example}
    If we apply the forgetful functor to the symmetric operad $\mathsf{Ass}$, we obtain a \emph{non-symmetric} operad whose spaces are $\mathsf{Ass}(n) = \mathbb{K}[\mathbb{S}_n]$ and whose composition is the block composition of permutations. It is important to note that this non-symmetric operad is much larger and structurally different from the operad $\mathsf{As}$ (where $\mathsf{As}(n) \cong \mathbb{K}$).
\end{example}

The previous example clarifies that the symmetrization functor $\mathsf{Sym}$ and the forgetful functor $U$ are not inverses of each other. Their formal relationship is that of \emph{adjoint functors}. Specifically, the symmetrization functor $\mathsf{Sym}$ is the \emph{left adjoint} of the forgetful functor $U$; this is a standard result in the theory of algebraic operads, see e.g. \cite[Section 5.8]{lodayvallette}. This relationship is expressed by the natural isomorphism in the sets of morphisms:
\begin{equation}
\Hom_{\text{\texttt{SymOp}}}\bigl(\mathsf{Sym}(\mathcal{P}), \mathcal{Q}\bigr) \cong \Hom_{\text{\texttt{NonSymOp}}}\bigl(\mathcal{P}, U(\mathcal{Q})\bigr)
\end{equation}
for any non-symmetric operad $\mathcal{P}$ and any symmetric operad $\mathcal{Q}$.

\subsubsection{The Quotient-Inclusion Adjunction}
The previous discussion raises a natural question: since $U(\mathsf{Ass}) \neq \mathsf{As}$, is there a functor that correctly maps $\mathsf{Ass}$ to $\mathsf{As}$? The answer is yes, and it leads to a second, distinct adjunction.

\begin{remark}[The quotient functor]
\label{rem:quotient_functor}
The construction that maps $\mathsf{Ass}$ to $\mathsf{As}$ is given by a \emph{quotient functor} $Q: \texttt{SymOp} \to \texttt{NonSymOp}$.
Given a symmetric operad $\mathcal{P}$, this functor defines a non-symmetric operad $Q(\mathcal{P})$ where the space of operations of arity $n$ is the quotient of $\mathcal{P}(n)$ by the action of the symmetric group. Specifically, it is the quotient of $\mathcal{P}(n)$ by the subspace spanned by all differences between an element and its permutations:
\begin{equation}
Q(\mathcal{P})(n) \coloneqq \mathcal{P}(n) / \mathrm{Span}\{p - p \cdot \sigma \mid p \in \mathcal{P}(n), \sigma \in \mathbb{S}_n\}.
\end{equation}
This construction ``forces'' symmetry by identifying all operations that differ only by a permutation of their inputs. The composition maps for $Q(\mathcal{P})$ are well-defined as they are induced by the composition of $\mathcal{P}$.
\end{remark}

\begin{example}
    Applying the quotient functor to $\mathsf{Ass}$ yields exactly the operad $\mathsf{As}$. Indeed, since the action of $\mathbb{S}_n$ on $\mathbb{K}[\mathbb{S}_n]$ permutes the basis elements, the quotient space becomes one-dimensional for $n \ge 1$: $Q(\mathsf{Ass})(n) \cong \mathbb{K} \cong \mathsf{As}(n)$.
\end{example}

\begin{remark}[The inclusion functor]
The right adjoint to the quotient functor $Q$ is the \emph{inclusion functor} $I: \texttt{NonSymOp} \to \texttt{SymOp}$. This functor takes a non-symmetric operad $\mathcal{P}$ and views it as a symmetric one by endowing it with the \emph{trivial} $\mathbb{S}_n$-action, where $p \cdot \sigma = p$ for all $p \in \mathcal{P}(n)$ and $\sigma \in \mathbb{S}_n$.

This functor provides an embedding of $\texttt{NonSymOp}$ as a full subcategory of $\texttt{SymOp}$.
\end{remark}

The relationship between $Q$ and $I$ is that of an adjoint pair, where the quotient functor $Q$ is the \emph{left adjoint} of the inclusion functor $I$:
\begin{equation}
\Hom_{\texttt{NonSymOp}}\bigl(Q(\mathcal{P}), \mathcal{Q}\bigr) \cong \Hom_{\texttt{SymOp}}\bigl(\mathcal{P}, I(\mathcal{Q})\bigr).
\end{equation}

\subsubsection{Summary and Interplay of Functors}
In conclusion, there are at least two fundamental ways to relate symmetric and non-symmetric operads, each described by a pair of adjoint functors with different conceptual meanings.

\begin{center}
\begin{tabular}{|l|l|}
\hline
\textbf{Left Adjoint} & \textbf{Right Adjoint}  \\ \hline
Symmetrization ($\mathsf{Sym}$) & Forgetful ($U$) \\
Quotient ($Q$) & Inclusion ($I$) \\
\hline
\end{tabular}
\end{center}

Beyond the two adjunctions discussed, there is a noteworthy direct relationship between the symmetrization functor $\mathsf{Sym}$ and the quotient functor $Q$. While they do not form an adjoint pair, their composition in one direction yields the identity functor, revealing a form of ``annihilation'' of the symmetrization process.

Specifically, the composition $Q \circ \mathsf{Sym}: \texttt{NonSymOp} \to \texttt{NonSymOp}$ is naturally isomorphic to the identity functor $\mathrm{Id}_{\texttt{NonSymOp}}$.
For any non-symmetric operad $\mathcal{P}$, we have $(Q \circ \mathsf{Sym})(\mathcal{P}) \cong \mathcal{P}$.

This shows that the quotient functor effectively "undoes" the free symmetrization. The composition in the opposite direction, $\mathsf{Sym} \circ Q$, is not the identity, as seen in the following examples.

\begin{example}[From $\mathsf{As}$ to $\mathsf{Ass}$ and back]
We can easily observe that:
\begin{enumerate}
    \item $\mathsf{Sym}(\mathsf{As}) = \mathsf{Ass}.$
    \item $Q(\mathsf{Ass}) = \mathsf{As}.$
\end{enumerate}
The sequence $\mathsf{As} \xrightarrow{\mathsf{Sym}} \mathsf{Ass} \xrightarrow{Q} \mathsf{As}$ demonstrates how the quotient functor inverts the symmetrization in this case.
\end{example}

\begin{example}[From $\mathsf{Com}$ to $\mathsf{As}$]
The composition in the opposite direction, $\mathsf{Sym} \circ Q$, is not the identity. Consider the operad $\mathsf{Com}$ (where $\mathsf{Com}(n) = \mathbb{K}$ with trivial action).
\begin{enumerate}
    \item Applying the quotient functor yields
    \begin{equation}
    Q(\mathsf{Com})(n) = \mathsf{Com}(n) / W_n = \mathbb{K} / \{0\} \cong \mathbb{K} \cong \mathsf{As}(n). \end{equation}
    So, $Q(\mathsf{Com}) = \mathsf{As}$.
    \item Applying the symmetrization functor to $\mathsf{As}$ gives
    \begin{equation}
    \mathsf{Sym}(\mathsf{As}) = \mathsf{Ass}. \end{equation}
\end{enumerate}
The sequence $\mathsf{Com} \xrightarrow{Q} \mathsf{As} \xrightarrow{\mathsf{Sym}} \mathsf{Ass}$ shows that this composition is not the identity, as we started with $\mathsf{Com}$ and ended with $\mathsf{Ass}$. This highlights that the relationship is not that of an inverse equivalence.
\end{example}

\subsection{Symmetric Operads on a Symmetric Monoidal Category}
\label{sec:operads-monoidal-category}

For completeness, more generally, and without claiming exhaustive rigor on all formal details, this section provides an overview of operads defined on a generic symmetric monoidal category. For a complete and rigorous treatment, see \cite[Chapter VII]{maclane98}.

In Sections~\ref{sec:non_symmetric_operads} and \ref{sec: s operad}, we introduced the definitions of non-symmetric and symmetric operads, focusing on the category \texttt{Vect} of vector spaces. However, the formalism of operads is significantly more general and extends to the broader context of a symmetric monoidal category. To formulate this generalization, we can adapt the definition previously given for linear operads by replacing the specific structures of \texttt{Vect} with their categorical counterparts. 

For completeness, we state this general definition explicitly.

\begin{definition}[Symmetric operad in a symmetric monoidal category]
\label{def:operad_in_C}
Let $(\mathscr{C}, \otimes, u)$ be a symmetric monoidal category. A \emph{symmetric operad} in $\mathscr{C}$ consists of the following data:
\begin{enumerate}
    \item A collection of objects, $\{\mathcal{P}(n) \in \mathrm{Ob}(\mathscr{C})\}_{n \in \mathbb{N}_0}$. 
    
    \item For each $n \ge 0$, a right action of the symmetric group $\mathbb{S}_n$ on the object $\mathcal{P}(n)$. This is given by a family of morphisms $\cdot\sigma: \mathcal{P}(n) \to \mathcal{P}(n)$ for each $\sigma \in \mathbb{S}_n$, satisfying the usual axioms of a group action, formulated as commuting diagrams in $\mathscr{C}$.
    
    \item A family of \textit{composition morphisms} in $\mathscr{C}$:
    \begin{equation}
    \gamma_{k_1, \dots, k_n}: \mathcal{P}(n) \otimes \mathcal{P}(k_1) \otimes \cdots \otimes \mathcal{P}(k_n) \to \mathcal{P}(k_1 + \cdots + k_n).
    \end{equation}
    
    \item A \textit{unit morphism} $1: u \to \mathcal{P}(1)$ in $\mathscr{C}$. 
\end{enumerate}
This data is required to satisfy axioms of associativity, unitality, and equivariance. These axioms take the form of the commutativity of diagrams in $\mathscr{C}$, analogous to the pentagon and triangle diagrams that govern the coherence of the monoidal category itself.
\end{definition}

While this definition is fully general, to construct the endomorphism operad and generalize the definition of algebra over an operad, one needs to work in a richer setting. Specifically, we use \emph{symmetric monoidal categories closed over sets.} Such categories have the following features:

\begin{itemize}
    \item \emph{Underlying sets:} The objects admit an underlying set, an idea formalized by the existence of a \emph{forgetful functor} $U: \mathscr{C} \to \texttt{Set}$. For any object $O$, its underlying set is denoted by $U(O)$. For example, the underlying set of a vector space $(V,\mathbb{K},+,\cdot)$ is the set $V$ of its vectors.
    
    \item \emph{Underlying functions:} Every morphism $f: X \to Y$ in $\mathscr{C}$ has an underlying function between the corresponding sets, denoted by $U(f): U(X) \to U(Y)$ (e.g., linear maps are functions between vectors).
    
    \item \emph{Internal Hom-objects:} The \emph{Hom-sets} (sets of morphisms between two objects) are themselves objects in the same category (e.g., the set of linear maps between two $\mathbb{K}$-vector spaces is itself a $\mathbb{K}$-vector space).
    
    \item \emph{Tensor map:} There exists a map
    \begin{equation}
    \otimes: U(X) \times U(Y) \rightarrow U(X \otimes Y)
    \end{equation}
    from the Cartesian product of the underlying sets of two objects $X, Y$ to the underlying set of their tensor product, satisfying the universal property of the tensor product \cite{Vitagliano2025}.
    
    \item \emph{Unit:} The underlying set of the monoidal unit, $U(u)$, contains a distinguished element $1 \in U(u)$ such that for every object $X$ and every element $x \in U(X)$, there exists a unique morphism $f: u \rightarrow X$ with $U(f)(1) = x$. Furthermore, for the left and right unit constraints, $\lambda$ and $\rho$, their underlying functions must satisfy the following identities:
\begin{equation}
U(\lambda)(1 \otimes x) = x = U(\rho)(x \otimes 1).
\end{equation}
\end{itemize}

In categories possessing this additional structure, the endomorphism operads ($\End_O$) can be defined in an analogous way, and the isomorphism theorems seen for $\mathsf{As}$, $\mathsf{uAs}$, $\mathsf{Com}$, and $\mathsf{uCom}$ generalize, establishing correspondences with the respective algebras in these broader contexts. For example, the endomorphism operad on an object $O$ of a symmetric monoidal category closed over sets is given by:
\begin{equation}
\End_O(n) = \Hom(O^{\otimes n}, O),
\end{equation}
with composition defined analogously to the case of linear operads.

The following results, summarized in a single theorem for clarity, characterize the correspondence between algebras over specific operads and the corresponding algebraic structures in a generic symmetric monoidal category closed over sets $\mathscr{C}$. The notations $\mathsf{As}^{\mathscr{C}}$, $\mathsf{uAs}^{\mathscr{C}}$, $\mathsf{Com}^{\mathscr{C}}$, and $\mathsf{uCom}^{\mathscr{C}}$ denote the versions of these operads defined specifically on the category $\mathscr{C}$.

\begin{theorem}
\label{thm:operad_algebra_correspondence_general}
Let $\mathscr{C}$ be a symmetric monoidal category closed over sets. Then:
\begin{enumerate}
    \item If $\mathscr{C}$ has an initial object, the category of $\mathsf{As}$-algebras on $\mathscr{C}$, denoted $\emph{\texttt{Alg}}(\mathsf{As}^{\mathscr{C}})$, is isomorphic to the category of semigroup objects in $\mathscr{C}$, denoted \emph{\texttt{Semi}}$_{\mathscr{C}}$. In symbols:
    \begin{equation}
        \emph{\texttt{Alg}}(\mathsf{As}^{\mathscr{C}}) \cong \emph{\texttt{Semi}}_{\mathscr{C}}.
    \end{equation}
    
    \item The category of $\mathsf{uAs}$-algebras on $\mathscr{C}$, denoted $\emph{\texttt{Alg}}(\mathsf{uAs}^{\mathscr{C}})$, is isomorphic to the category of monoid objects in $\mathscr{C}$, denoted \emph{\texttt{Mon}}$_{\mathscr{C}}$. In symbols:
    \begin{equation}
        \emph{\texttt{Alg}}(\mathsf{uAs}^{\mathscr{C}}) \cong \emph{\texttt{Mon}}_{\mathscr{C}}.
    \end{equation}
    
    \item If $\mathscr{C}$ has an initial object, the category of $\mathsf{Com}$-algebras on $\mathscr{C}$, denoted $\emph{\texttt{Alg}}(\mathsf{Com}^{\mathscr{C}})$, is isomorphic to the category of commutative semigroup objects in $\mathscr{C}$, denoted $\emph{\texttt{CSemi}}_{\mathscr{C}}$. In symbols:
    \begin{equation}
        \emph{\texttt{Alg}}(\mathsf{Com}^{\mathscr{C}}) \cong \emph{\texttt{CSemi}}_{\mathscr{C}}.
    \end{equation}
    
    \item The category of $\mathsf{uCom}$-algebras on $\mathscr{C}$, denoted $\emph{\texttt{Alg}}(\mathsf{uCom}^{\mathscr{C}})$, is isomorphic to the category of commutative monoid objects in $\mathscr{C}$, denoted $\emph{\texttt{CMon}}_{\mathscr{C}}$. In symbols:
    \begin{equation}
        \emph{\texttt{Alg}}(\mathsf{uCom}^{\mathscr{C}}) \cong \emph{\texttt{CMon}}_{\mathscr{C}}.
    \end{equation}
\end{enumerate}
These isomorphisms of categories are established by functors constructed analogously to those for \emph{\texttt{Vect}}.
\end{theorem}

\begin{proof}
The proof follows by a direct verification. For each case, one constructs functors in both directions and shows they establish an isomorphism of categories. For instance, given a $\mathsf{uAs}^{\mathscr{C}}$-algebra, the operation $\gamma: A \otimes A \to A$ and the unit $u: I \to A$ satisfy associativity and unit axioms for a monoid object. Conversely, a monoid object immediately provides the structure of a $\mathsf{uAs}^{\mathscr{C}}$-algebra. The other cases are similar.
\end{proof}

\begin{example}[Correspondences for the category \texttt{Set}]
\label{ex:set_correspondences}
An instructive example of the application of Theorem~\ref{thm:operad_algebra_correspondence_general} is obtained by considering the category $\texttt{Set}$. The category $\texttt{Set}$ is a symmetric monoidal category closed over sets: it has an initial object (the empty set $\varnothing$) and a terminal object (any singleton). Its \emph{Hom-sets} are sets of functions, and the tensor product coincides with the Cartesian product.

In the context of $\texttt{Set}$, the $\mathscr{C}$-algebraic objects of Theorem~\ref{thm:operad_algebra_correspondence_general} specialize as follows:
\begin{itemize}
    \item Semigroup objects in \texttt{Semi}$_{\texttt{Set}}$ are semigroups.
    \item Monoid objects in \texttt{Mon}$_{\texttt{Set}}$ are monoids.
    \item Commutative semigroup objects in \texttt{CSemi}$_{\texttt{Set}}$ are commutative semigroups.
    \item Commutative monoid objects in \texttt{CMon}$_{\texttt{Set}}$ are commutative monoids.
\end{itemize}

Therefore,
\begin{itemize}
    \item The category of $\mathsf{As}^{\texttt{Set}}$-algebras is isomorphic to the category of semigroups.
    \item The category of $\mathsf{uAs}^{\texttt{Set}}$-algebras is isomorphic to the category of monoids.
    \item The category of $\mathsf{Com}^{\texttt{Set}}$-algebras is isomorphic to the category of commutative semigroups.
    \item The category of $\mathsf{uCom}^{\texttt{Set}}$-algebras is isomorphic to the category of commutative monoids. 
\end{itemize}

Thus, operads provide a unified language for describing and classifying different algebraic structures, highlighting their common properties through a categorical formalism.
\end{example}

\begin{remark}[Operads as enriched multicategories]
\label{rem:enriched_multicategories}
The link between operads and multicategories, already explored in the case of sets (see Remark~\ref{rem:cat_nonsym_op}), extends to the more general context of a symmetric monoidal category $\mathscr{C}$.

As an operad on \texttt{Set} is a multicategory with a single object, an operad on $\mathscr{C}$ can be defined in a completely equivalent way as a \emph{multicategory enriched over $\mathscr{C}$} with a single object. The foundational reference for theory of enriched categories is Kelly \cite{Kelly1982}. By ``enriched'' over $\mathscr{C}$, we mean that:
\begin{itemize}
    \item The collections of morphisms of the multicategory, $\Hom(x_1, \dots, x_n; y)$, are \emph{objects} of the category $\mathscr{C}$.
    \item The composition maps are \emph{morphisms} in $\mathscr{C}$.
\end{itemize}
In our case, since the operad has only one type of object, the enriched multicategory has a single object. Its collections of $n$-ary morphisms correspond to the objects $\mathcal{P}(n)$ of the operad, and the composition maps of the operad are morphisms in the category $\mathscr{C}$, as required by the definition (e.g., the map $\gamma$ in Definition~\ref{def: ns operad} is a morphism in \texttt{Vect}, i.e., a linear map).

This point of view solidifies the idea that operads are nothing more than the single-object version of multicategories, within a given categorical context $(\mathscr{C}, \otimes, u)$.
\end{remark}

\section{Alternative Definitions of an Operad}
\label{chap:def_alternative}
In this section, after introducing the classical definition of an operad as a multicategory with a single object, we explore some equivalent formulations. These alternative definitions can be more suitable depending on the context and highlight different aspects of the structure of an operad.

\subsection{Partial Definition of an Operad}

Besides the classical definition of an operad based on a single ``total'' composition map $\gamma$, there is an alternative and equivalent axiomatic approach. This approach, known as the \emph{partial definition}, is often more suitable for combinatorial and homological algebra. It was developed extensively by Markl (see e.g. \cite{markl08, markl_shnider_stasheff}) and is based on more elementary operations.

The central idea is to use a series of \emph{partial compositions}, denoted by $\circ_i$, instead of a single total one. Each $\circ_i$ is a binary operation that describes a fundamental action: inserting an operad element $\nu$ into the $i$-th input of another element $\mu$. This perspective decomposes complex compositions into their atomic building blocks.

While this approach results in a larger set of axioms (for example, two distinct associativity laws), it provides a powerful combinatorial handle on the operad structure. This fine-grained control is crucial for developing deep structural results, such as the theory of Koszul duality for operads pioneered by Ginzburg and Kapranov \cite{GinzburgKapranov1994}.

We  only consider the case of symmetric operads, as the non-symmetric case can be obtained by simply forgetting the action of the symmetric group.

\begin{definition}[Partial definition of an operad]
\label{def:operad_partial}
A \emph{symmetric operad} $\mathcal{P}$ is a sequence of vector spaces $(\mathcal{P}(n))_{n \in \mathbb{N}_0}$ equipped with a right $\mathbb{S}_n$-module structure (Definition \ref{def:group_representation}), a unit element $1 \in \mathcal{P}(1)$, and a family of \emph{partial composition maps}
\begin{equation}
\circ_i : \mathcal{P}(m) \otimes \mathcal{P}(n) \rightarrow \mathcal{P}(m+n-1),
\end{equation}
for $1 \le i \le m$ and $n \ge 0$. These maps are required to satisfy the following axioms:

\begin{enumerate}
    \item \emph{Associativity laws:} For any $\lambda \in \mathcal{P}(\ell)$, $\mu \in \mathcal{P}(m)$, $\nu \in \mathcal{P}(n)$, the following relations hold:
    \begin{align}
        (\lambda \circ_i \mu) \circ_{i+j-1} \nu &= \lambda \circ_i (\mu \circ_j \nu), && \text{for } 1 \le i \le \ell, 1 \le j \le m. \label{eq:assoc_partial_nested} \\
        (\lambda \circ_i \mu) \circ_{k-1+m} \nu &= (\lambda \circ_k \nu) \circ_i \mu, && \text{for } 1 \le i < k \le \ell. \label{eq:assoc_partial_disjoint}
    \end{align}
    
    \item \emph{Unit laws:} For any $\mu \in \mathcal{P}(n)$ and $1 \le i \le n$:
    \begin{equation}
    \mu \circ_i 1 = \mu \quad \text{and} \quad 1 \circ_1 \mu = \mu.
    \end{equation}
    
    \item \emph{Equivariance:} For any $\mu \in \mathcal{P}(m)$, $\nu \in \mathcal{P}(n)$, $\sigma \in \mathbb{S}_m$, $\tau \in \mathbb{S}_n$:
    \begin{equation}
    (\mu \cdot \sigma) \circ_i (\nu \cdot \tau) = (\mu \circ_{\sigma(i)} \nu) \cdot (\sigma \circ_i \tau),
    \end{equation}
    where $\sigma \circ_i \tau \in \mathbb{S}_{m+n-1}$ is the permutation obtained by inserting the block of inputs of $\tau$ into the $i$-th position of the input blocks of $\sigma$.
\end{enumerate}
\end{definition}

Graphically, the partial composition $\mu \circ_i \nu$ corresponds to grafting the root of the tree representing $\nu$ onto the $i$-th leaf of the tree representing $\mu$.

\begin{figure}[H]
    \centering
    \begin{tikzpicture}[scale=1.4]

    \node (mu) at (0,0) {$\mu$};
    \node (nu) at (0,1.5) {$\nu$};

    \draw (mu) -- (nu);

    \node at (0.15,0.75) {$i$};

    \draw (mu) -- (0,-0.5);

    \draw (mu) -- (-1, 0.5);
    \draw (mu) -- (-0.5, 0.5);
    \draw (mu) -- (0.5, 0.5);
    \draw (mu) -- (1, 0.5);

    \draw (nu) -- (-1, 2);
    \draw (nu) -- (-0.5, 2);
    \draw (nu) -- (0, 2);
    \draw (nu) -- (0.5, 2);
    \draw (nu) -- (1, 2);
    \end{tikzpicture}
    \caption{Graphical representation of the composition $\mu \circ_i \nu$}
    \label{fig:partialcomp}
\end{figure}
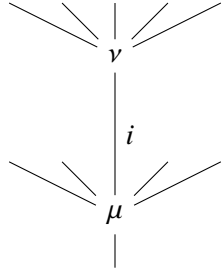

\begin{remark}
There are two distinct associativity laws because there exist two topologically different ways to compose three operations $\lambda, \mu, \nu$ through successive graftings.
\begin{enumerate}
    \item \emph{Nested composition:} The first case describes a "vertical" composition, where the operation $\nu$ is grafted into $\mu$, and the result $(\mu \circ_j \nu)$ is then grafted into $\lambda$. The relation \eqref{eq:assoc_partial_nested} states that the result is the same, regardless of how the operations are grouped.
    
    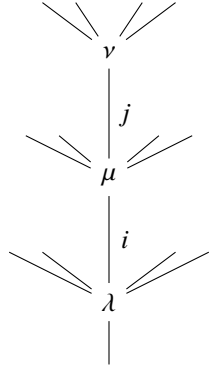
\begin{figure}[H]
        \centering
        \begin{tikzpicture}[scale=1.1, every node/.style={font=\small}]
            \node (l) at (0,0) {$\lambda$};
            \node (m) at (0,1.5) {$\mu$};
            \node (n) at (0,3) {$\nu$};
            
            \draw (l) -- (m);
            \draw (m) -- (n);
            \draw (l) -- (0,-0.75); 
            
            \node at (0.2, 0.75) {$i$};
            \node at (0.2, 2.25) {$j$};
            
            \draw (l) -- (-1.2, 0.6); \draw (l) -- (-0.8, 0.6);
            \draw (l) -- (0.8, 0.6); \draw (l) -- (1.2, 0.6);
            
            \draw (m) -- (-1, 2); \draw (m) -- (-0.6, 2.0);
            \draw (m) -- (0.6, 2.0); \draw (m) -- (1, 2);
            
            \draw (n) -- (-0.8, 3.6); \draw (n) -- (-0.4, 3.6);
            \draw (n) -- (0.4, 3.6); \draw (n) -- (0.8, 3.6);
        \end{tikzpicture}
        \caption{Visualization of the associativity law for nested graftings.}
    \end{figure}

    \item \emph{Disjoint composition:} The second case describes a "horizontal" composition, where the operations $\mu$ and $\nu$ are grafted onto two \emph{distinct} leaves ($i$ and $k$) of the operation $\lambda$. The axiom \eqref{eq:assoc_partial_disjoint} states that the order of these two independent graftings does not affect the result.
    
    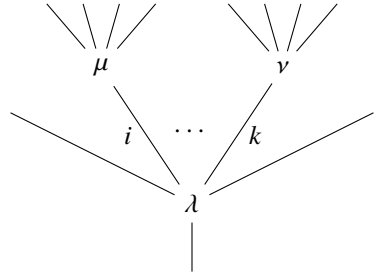
\begin{figure}[H]
        \centering
        \begin{tikzpicture}[scale=1.2, every node/.style={font=\small}]
            \node (l) at (0,0) {$\lambda$};
            \node (m) at (-1, 1.5) {$\mu$};
            \node (n) at (1, 1.5) {$\nu$};
            
            \draw (l) -- (m);
            \draw (l) -- (n);
            \draw (l) -- (0, -0.75); 
            \draw (l) -- (-2, 1);
            \draw (l) -- (2, 1);

            \node at (0,0.8) {\ldots};

            \node at (-0.7, 0.75) {$i$};
            \node at (0.7, 0.75) {$k$};

            \draw (m) -- (-1.6, 2.2); \draw (m) -- (-1.2, 2.2); \draw (m) -- (-0.8, 2.2); \draw (m) -- (-0.4, 2.2);
            \draw (n) -- (0.4, 2.2); \draw (n) -- (0.8, 2.2); \draw (n) -- (1.2, 2.2); \draw (n) -- (1.6, 2.2);
        \end{tikzpicture}
        \caption{Visualization of the associativity law for disjoint graftings.}
    \end{figure}
\end{enumerate}
\end{remark}

\begin{example}
To better understand the equivariance condition for partial compositions, let us consider a concrete example. Let $\mu \in \mathcal{P}(4)$, $\nu \in \mathcal{P}(3)$, and let the permutations be $\sigma = \left(\begin{smallmatrix} 1 & 2 & 3 & 4 \\ 3 & 4 & 2 & 1 \end{smallmatrix}\right) \in \mathbb{S}_4$ and $\tau = \left(\begin{smallmatrix} 1 & 2 & 3 \\ 2 & 3 & 1 \end{smallmatrix}\right) \in \mathbb{S}_3$. 
We want to verify equivariance for the composition for $i=2$:
\begin{equation}
(\mu \cdot \sigma) \circ_2 (\nu \cdot \tau) = (\mu \circ_{\sigma(2)} \nu) \cdot (\sigma \circ_2 \tau). 
\end{equation}
The following chain of diagrams illustrates this identity.
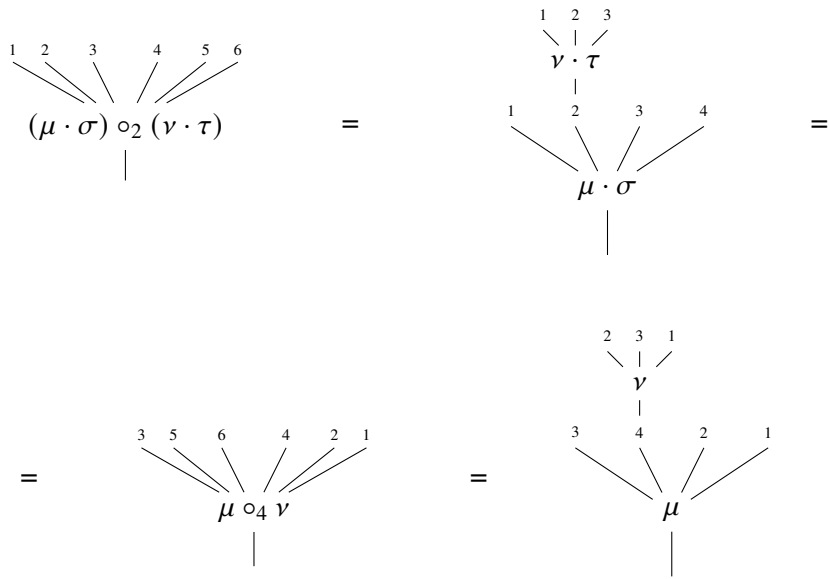
\vspace{0.3cm}

\begin{tikzpicture}[scale=0.85]
        
        \begin{scope}[shift={(-7,0)}]
            \node (op) at (0,0) {$(\mu\cdot\sigma)\circ_2(\nu\cdot\tau)$};
            \node(r) at (0,-1){};
            \draw (r)-- (op);
            \foreach \x/\l in {-1.75/1, -1.25/2, -0.5/3, 0.5/4, 1.25/5, 1.75/6} {
                \draw (op)--(\x,1);
                \node at (\x,1.2){\tiny \l};
            }
        \end{scope}

        \node at (-3.5,0) {\large =};
        
        \begin{scope}[shift={(0.5,-1)}]
            \node(m) at (0,0){$\mu\cdot\sigma$};
            \node(n) at (-0.5,2){$\nu\cdot\tau$};
            \draw (m)--(0,-1);
            \draw (m)--(-1.5,1); \node[above] at (-1.5,1){\tiny 1};
            \draw (m)--(-0.5,1); \node[above] at (-0.5,1){\tiny 2};
            \draw (m)--(0.5,1); \node[above] at (0.5,1){\tiny 3};
            \draw (m)--(1.5,1); \node[above] at (1.5,1){\tiny 4};
            \draw (-0.5,1.5)--(n);
            \draw (n)--(-1,2.5); \node[above] at (-1,2.5){\tiny 1};
            \draw (n)--(-0.5,2.5); \node[above] at (-0.5,2.5){\tiny 2};
            \draw (n)--(0,2.5); \node[above] at (0,2.5){\tiny 3};
        \end{scope}

        \node (align_point) at (3.8,0) {\large =};

        \node at (-8.5,-5.5) {\large =};

        \begin{scope}[shift={(1.5, -6cm)}]
             \node(m) at (0,0){$\mu$};
             \node(n) at (-0.5,2){$\nu$};
             \draw (m)--(0,-1);
             \draw (m)--(-1.5,1); \node[above] at (-1.5,1){\tiny 3};
             \draw (m)--(-0.5,1); \node[above] at (-0.5,1){\tiny 4};
             \draw (m)--(0.5,1); \node[above] at (0.5,1){\tiny 2};
             \draw (m)--(1.5,1); \node[above] at (1.5,1){\tiny 1};
             \draw (-0.5,1.5)--(n);
             \draw (n)--(-1,2.5); \node[above] at (-1,2.5){\tiny 2};
             \draw (n)--(-0.5,2.5); \node[above] at (-0.5,2.5){\tiny 3};
             \draw (n)--(0,2.5); \node[above] at (0,2.5){\tiny 1};
        \end{scope}

        \node at (-1.5, -5.5cm) {\large =};

        \begin{scope}[shift={(-5, -6cm)}]
            \node (op) at (0,0) {$\mu\circ_4\nu$};
            \node(r) at (0,-1){};
            \draw (r)-- (op);
            \foreach \x/\l in {-1.75/3, -1.25/5, -0.5/6, 0.5/4, 1.25/2, 1.75/1} {
                \draw (op)--(\x,1);
                \node at (\x,1.2){\tiny \l};
            }
        \end{scope}

\end{tikzpicture}

\vspace{0.3cm}

The composite permutation $\sigma \circ_2 \tau \in \mathbb{S}_6$ is given by the diagram

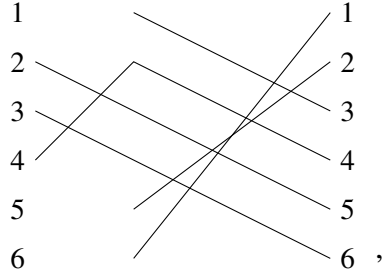
\begin{center}
\begin{tikzpicture}[scale=0.65, baseline={([yshift=-.5ex]current bounding box.center)}]

    \foreach \N in {1,...,6} {

        \node[left] at (0, -\N) {\N};
        \node[right] at (6, -\N) {\N};
    }

    \draw (2, -1) -- (6, -3);  
    \draw (0, -2) -- (6, -5);  
    \draw (0, -3) -- (6, -6);  
    \draw (0,-4) -- (2,-2) -- (6, -4); 
    \draw (2, -5) -- (6,-2);
    \draw (2,-6) -- (6,-1);
    \node at (7, -6) {\large ,};
\end{tikzpicture}
\end{center}
and $(\mu \circ_{\sigma(2)} \nu) \cdot (\sigma \circ_2 \tau)$ corresponds to 
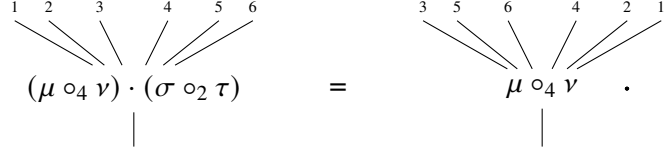
\begin{figure}[H]
    \centering
    \begin{tikzpicture}[scale=0.9]
        \begin{scope}[shift={(-3,0)}]
            \node (op) at (0,0) {$(\mu \circ_4\nu) \cdot (\sigma \circ_2 \tau)$};
            \node(r) at (0,-1){};
            \draw (r)-- (op);
            \foreach \x/\l in {-1.75/1, -1.25/2, -0.5/3, 0.5/4, 1.25/5, 1.75/6} {
                \draw (op)--(\x,1);
                \node at (\x,1.2){\tiny \l};
            }
        \end{scope}

        \node at (0,0) {\large =};

            \begin{scope}[shift={(3,0)}]
            \node (op) at (0,0) {$\mu\circ_4\nu$};
            \node(r) at (0,-1){};
            \draw (r)-- (op);
            \foreach \x/\l in {-1.75/3, -1.25/5, -0.5/6, 0.5/4, 1.25/2, 1.75/1} {
                \draw (op)--(\x,1);
                \node at (\x,1.2){\tiny \l};
                \node at (1.25, 0) {\large .};
            }
        \end{scope}

    \end{tikzpicture}
\end{figure}

\end{example}

\begin{proposition}[Equivalence of operad definitions]
\label{prop:equivalence_of_operad_definitions}
The classical definition of an operad (based on total composition $\gamma$) and the partial definition (based on partial compositions $\circ_i$) are equivalent.
\end{proposition}

\begin{proof}
The equivalence is shown by constructing each type of composition in terms of the other.

\begin{enumerate}
    \item \emph{From total to partial composition.} Given the total composition map $\gamma$, the partial composition is defined as a special case of $\gamma$ where only one non-trivial operation is grafted onto the $i$-th leaf, while all other leaves are grafted with the unit $1$:
    \begin{equation}
    \mu \circ_i \nu \coloneqq \gamma(\mu; 1, \dots, 1, \underbrace{\nu}_{\text{pos. } i}, 1, \dots, 1).
    \end{equation}

    \item \emph{From partial to total composition.} Given the family of partial compositions $\circ_i$, the total composition $\gamma(\mu; \nu_1, \dots, \nu_k)$ is uniquely defined by induction on the number $k$ of grafted operations.
    \begin{itemize}
    \item \emph{Base case ($k=1$):} The composition with a single operation is simply $\gamma(\mu; \nu_1) \coloneqq \mu \circ_1 \nu_1$.
    \item \emph{Inductive step ($k>1$):} Assuming that $\gamma(\mu; \nu_1, \dots, \nu_{k-1})$ has been defined, one defines
    \begin{equation}
    \gamma(\mu; \nu_1, \dots, \nu_k) \coloneqq \gamma\bigl(\gamma(\mu; \nu_1, \dots, \nu_{k-1}); \nu_k\bigr), \end{equation}
    where the composition on the right-hand side of this definition is an iterated application of partial compositions. Since the associativity laws \eqref{eq:assoc_partial_nested} and \eqref{eq:assoc_partial_disjoint} ensure that the order of these compositions does not matter, the result is well-defined.
\end{itemize}
\end{enumerate}
\end{proof}

\subsection{Functorial Definition of an Operad}
\label{sec:functorial_definition}

For completeness, we introduce a third definition of an operad, which is purely categorical. This approach, known as the \emph{functorial definition} \cite{lodayvallette}, defines an operad as a particular type of \emph{monad}.

To formulate this definition, one considers the \emph{category of endofunctors} on \texttt{Vect}, denoted by $\texttt{End}(\texttt{Vect})$. The objects of this category are the functors $F\colon \texttt{Vect}\to\texttt{Vect}$, and the morphisms are the natural transformations between these functors. This category is monoidal: the monoidal product is given by the \emph{composition of functors} and the unit object is the \emph{identity functor} $\mathrm{Id}$.

\begin{definition}[Functorial definition of an operad]
A (symmetric) \emph{operad} is a \emph{monoid} in the monoidal category $(\texttt{End}(\texttt{Vect}), \circ, \mathrm{Id})$. Such a structure is more commonly known as a \emph{monad}.

Explicitly, an operad $\mathcal{P}$ is an endofunctor $\mathcal{P} \in \mathrm{Ob}(\texttt{End}(\texttt{Vect}))$ equipped with two natural transformations:
\begin{enumerate}
    \item a \emph{multiplication} $\gamma\colon \mathcal{P}\circ \mathcal{P}\;\Longrightarrow\;\mathcal{P}$,
    \item a \emph{unit} $\eta\colon \mathrm{Id}\;\Longrightarrow\;\mathcal{P}$,
\end{enumerate}
that satisfy the coherence diagrams for the associativity \eqref{eq:monoid_associativity_diagram} and unitality \eqref{eq:monoid_unit_diagrams} of a monoid. The components of these natural transformations, for each vector space $V$, are linear maps:
\begin{equation}
 \gamma_V\colon \mathcal{P}\bigl(\mathcal{P}(V)\bigr)\to \mathcal{P}(V)
 \quad \text{and} \quad
 \eta_V\colon V\to \mathcal{P}(V).
\end{equation}
\end{definition}
It is important to emphasize that this functorial approach is entirely equivalent to the classical definition. The equivalence between the two formulations is established by a construction known as the \emph{Schur functor}, which canonically associates an endofunctor to every $\mathbb{S}$-module.

A treatment of this functor is beyond the scope of this review. For more details, see \cite{lodayvallette, hilgerponcin}.

\section{Operads via Generators and Relations}
\label{chap:gen_rel}

 In this section, we present a more constructive and algebraic approach to defining operads, known as the \emph{definition of an operad via generators and relations}. This method allows for the description of an operad by specifying a minimal set of generating operations and the relations. To develop this formalism, three fundamental algebraic concepts are necessary, which we adapt to the operadic context: \emph{ideals}, \emph{quotients}, and \emph{free} structures. We begin by defining the notion of an operadic ideal and a quotient operad. Subsequently, we introduce the crucial concept of a \emph{free operad}, which can be thought of as the operad constructed from a set of operations without imposing any additional relations. By combining these tools, we show how an operad can be defined as the quotient of a free operad by an ideal generated by the relations. Finally, we  apply this technique to construct the example of the $\mathsf{Lie}$ operad, which encodes the structure of Lie algebras.

\subsection{Operadic Ideals and Quotient Operads}
\label{sec:ideals_and_quotients}
\begin{definition}[Operadic ideal]
\label{def:operadic_ideal}
Let $\mathcal{P}$ be a symmetric operad. An \emph{operadic ideal} $\mathcal{I}$ of $\mathcal{P}$ is a collection of vector spaces $\mathcal{I} = \{\mathcal{I}(n) \subseteq \mathcal{P}(n)\}_{n \in \mathbb{N}_0}$ that satisfies the following conditions:
\begin{enumerate}
    \item For each $n \ge 0$, $\mathcal{I}(n)$ is a sub-$\mathbb{S}_n$-module of $\mathcal{P}(n)$ (Definition \ref{def:group_representation}). This means that $\mathcal{I}(n)$ is a vector subspace of $\mathcal{P}(n)$ that is stable under the action of the symmetric group, i.e.:
    \begin{equation}
        \forall x \in \mathcal{I}(n), \ \forall \sigma \in \mathbb{S}_n, \quad x \cdot \sigma \in \mathcal{I}(n).
    \end{equation}
    
    \item It is closed with respect to composition with any element of the operad (absorption property). Formally, for any $\mu \in \mathcal{P}(k)$ and $\nu_j \in \mathcal{P}(n_j)$ for $j=1,\dots,k$, the composition $\gamma(\mu; \nu_1, \dots, \nu_k)$ belongs to $\mathcal{I}(n_1+\dots+n_k)$ if at least one of the component operations ($\mu$ or one of the $\nu_j$) belongs to $\mathcal{I}$.
\end{enumerate}
\end{definition}

\begin{proposition}
\label{prop:intersection_of_ideals}
The intersection of an arbitrary family of operadic ideals is an operadic ideal.
\end{proposition}

\begin{proof}
Let $\mathcal{P}$ be a symmetric operad and let $\{\mathcal{I}_j\}_{j \in J}$ be a family of operadic ideals. We define their intersection as the collection $\mathcal{I} \coloneqq \bigcap_{j \in J} \mathcal{I}_j$, where for each $n \ge 0$, the space $\mathcal{I}(n)$ is given by $\mathcal{I}(n) \coloneqq \bigcap_{j \in J} \mathcal{I}_j(n)$.

We need to verify that $\mathcal{I}$ satisfies the two conditions of Definition~\ref{def:operadic_ideal}.
\begin{itemize}
\item \emph{Stability as a sub-$\mathbb{S}_n$-module:} for each $n \ge 0$, $\mathcal{I}(n)$ is a vector subspace of $\mathcal{P}(n)$, since the intersection of any family of vector subspaces is still a vector subspace. Furthermore, if an element $x$ belongs to $\mathcal{I}(n)$, then $x$ belongs to every $\mathcal{I}_j(n)$. Since each $\mathcal{I}_j(n)$ is stable under the action of $\mathbb{S}_n$, it follows that $x \cdot \sigma \in \mathcal{I}_j(n)$ for every $j \in J$. Therefore, $x \cdot \sigma$ belongs to their intersection, $\mathcal{I}(n)$. Thus, $\mathcal{I}(n)$ is a sub-$\mathbb{S}_n$-module of $\mathcal{P}(n)$.
    
\item \emph{Absorption property:} consider a composition $\gamma(\mu; \nu_1, \dots, \nu_k)$ and suppose that at least one of the operands belongs to $\mathcal{I}$. Assume that $\mu \in \mathcal{I}(k)$; if $\nu_j \in \mathcal{I}(k)$ the reasoning is analogous.
By definition of intersection, if $\mu \in \mathcal{I}(k)$, then $\mu \in \mathcal{I}_j(k)$ for every $j \in J$.
Since each $\mathcal{I}_j$ is an operadic ideal, its absorption property implies that $\gamma(\mu; \nu_1, \dots, \nu_k) \in \mathcal{I}_j(n_1+\dots+n_k)$ for every $j \in J$. Since the composition belongs to every ideal $\mathcal{I}_j$ in the family, it also belong to their intersection. Thus, $\gamma(\mu; \nu_1, \dots, \nu_k) \in \mathcal{I}$.
\end{itemize}

\end{proof}

\begin{definition}[Construction of the quotient structure]
\label{def:quotient_construction}
Let $\mathcal{P}$ be a symmetric operad and let $\mathcal{I}$ be an operadic ideal in $\mathcal{P}$. We define the \emph{quotient structure} $\mathcal{P}/\mathcal{I}$ as the collection of quotient vector spaces
\begin{equation}
(\mathcal{P}/\mathcal{I})(n) \coloneqq \frac{\mathcal{P}(n)}{\mathcal{I}(n)}, \quad \text{for each } n \ge 0,
\end{equation}
equipped with the following operations defined via representatives of the equivalence classes:
\begin{enumerate}
    \item \emph{Unit:} the unit is the equivalence class of the unit $1 \in \mathcal{P}(1)$, i.e., $[1] \coloneqq 1 + \mathcal{I}(1)$.
    
    \item \emph{Composition:} for any $\mu \in \mathcal{P}(k)$ and any $\nu_j \in \mathcal{P}(n_j)$ with $j=1, \ldots,k$, the composition of the respective equivalence classes is defined as:
    \begin{equation}
    \label{eq:quotient_composition}
    \gamma([\mu]; [\nu_1], \dots, [\nu_k]) \coloneqq [\gamma(\mu; \nu_1, \dots, \nu_k)].
    \end{equation}
    
    \item \emph{Symmetric group action:} for any $\mu \in \mathcal{P}(n)$ and any permutation $\sigma \in \mathbb{S}_n$, the action on the equivalence class $[\mu]$ is defined as:
    \begin{equation}
    [\mu] \cdot \sigma \coloneqq [\mu \cdot \sigma].
    \end{equation}
\end{enumerate}
\end{definition}

\begin{proposition}
The quotient structure $\mathcal{P}/\mathcal{I}$, with the operations defined above, is a well-defined symmetric operad. This operad is called the \emph{quotient operad}.
\end{proposition}

\begin{proof}
The essential task is to show that the composition and the symmetric group action are well-defined, i.e., that they do not depend on the choice of representatives. This is guaranteed by the fact that $\mathcal{I}$ is an ideal. For example, for the composition (Equation \eqref{eq:quotient_composition}), if different representatives $\mu' \in [\mu]$ and $\nu_j' \in [\nu_j]$ are chosen, the difference $\gamma(\mu; \nu_1,\ldots,\nu_k) - \gamma(\mu'; \nu_1',\ldots,\nu_k')$, can be rewritten as a sum of compositions in which at least one of the component operations belongs to the ideal $\mathcal{I}$. Due to the absorption property of the ideal, the difference also belongs to $\mathcal{I}$, and thus the two compositions belong to the same equivalence class.

Once this property is established, the associativity, unit, and equivariance axioms for $\mathcal{P}/\mathcal{I}$ follow directly from the corresponding axioms for the operad $\mathcal{P}$.
\end{proof}

\begin{proposition}[Universal property of the quotient]
\label{prop:universal_property_quotient}
Let $\mathcal{P}$ be an operad, $\mathcal{I}$ an ideal, and $\pi: \mathcal{P} \to \mathcal{P}/\mathcal{I}$ the canonical projection onto the quotient. For any other operad $\mathcal{Q}$ and any operad morphism $f: \mathcal{P} \to \mathcal{Q}$ such that the ideal $\mathcal{I}$ is contained in the kernel of $f$, there exists a \emph{unique} operad morphism $\tilde{f}: \mathcal{P}/\mathcal{I} \to \mathcal{Q}$ that makes the following diagram commute:
\begin{equation}
\begin{tikzcd}
\mathcal{P} \arrow[r, "\pi"] \arrow[dr, "f"'] & \mathcal{P}/\mathcal{I} \arrow[d, dashed, "\exists! \tilde{f}"] \\
& \mathcal{Q}
\end{tikzcd}.
\end{equation}
\end{proposition}

\begin{proof}
The proof proceeds in two parts: the construction of the morphism $\tilde{f}$ and the verification of its uniqueness.
\begin{itemize}
\item \emph{Existence of $\tilde{f}$.}
We define the family of maps $\{\tilde{f}_n\}$ for each $n \ge 0$ as follows: for each equivalence class $[p] \in (\mathcal{P}/\mathcal{I})(n)$, where $p \in \mathcal{P}(n)$, we set:
\begin{equation}
\tilde{f}_n([p]) \coloneqq f_n(p).
\end{equation}
We must first show that this definition is well-posed, i.e., that it does not depend on the choice of the representative $p$ for the class $[p]$. Let us choose another representative $p' \in [p]$. By definition of an equivalence class, $p' - p \in \mathcal{I}(n)$. Since by hypothesis $\mathcal{I} \subseteq \ker(f)$, we have $f_n(p' - p) = 0$. By the linearity of $f_n$, this implies $f_n(p') - f_n(p) = 0$, i.e., $f_n(p') = f_n(p)$. The map $\tilde{f}_n$ is therefore well-defined.
We verify that $\tilde{f} = \{\tilde{f}_n\}$ is a morphism of operads.
\begin{itemize}
    \item \emph{Compatibility with composition:}
    \begin{equation}
    \begin{split}
        \tilde{f}\bigl(\gamma([\mu]; [\nu_1], \dots, [\nu_k])\bigr)
        &\overset{(1)}{=} \tilde{f}\bigl([\gamma(\mu; \nu_1, \dots, \nu_k)]\bigr) \\
        &\overset{(2)}{=} f\bigl(\gamma(\mu; \nu_1, \dots, \nu_k)\bigr) \\
        &\overset{(3)}{=} \gamma\bigl(f(\mu); f(\nu_1), \dots, f(\nu_k)\bigr) \\
        &\overset{(4)}{=} \gamma\bigl(\tilde{f}([\mu]); \tilde{f}([\nu_1]), \dots, \tilde{f}([\nu_k])\bigr).
        \end{split}
    \end{equation}
    where:
    \begin{itemize}
        \item[(1)]\quad follows from the definition of composition $\gamma$ in the quotient operad $\mathcal{P}/\mathcal{I}$;
        \item[(2)]\quad follows from the definition of $\tilde{f}$;
        \item[(3)]\quad follows from the fact that $f$ is an operad morphism,
        \item[(4)]\quad follows from the definition of $\tilde{f}$.
    \end{itemize}
    \smallskip
    \item \emph{Compatibility with the unit and the $\mathbb{S}_n$-action} is proven analogously, as these properties also follow directly from the fact that $f$ is an operad morphism.
\end{itemize}

\item \emph{Uniqueness of $\tilde{f}$.}
Suppose there exists another operad morphism $g: \mathcal{P}/\mathcal{I} \to \mathcal{Q}$ such that $g \circ \pi = f$. For any class $[p] \in \mathcal{P}/\mathcal{I}$, it holds that $g([p]) = g\bigl(\pi(p)\bigr) = (g \circ \pi)(p) = f(p)$. But by definition $\tilde{f}([p]) = f(p)$, so $g([p]) = \tilde{f}([p])$ for every $[p]$. Therefore, $g = \tilde{f}$.
\end{itemize}
\end{proof}

To lighten the notation, from now on sometimes we  omit the square brackets for equivalence classes in the context of a quotient operad $\mathcal{P}/\mathcal{I}$. We  thus simply write $\mu$ to denote the class $[\mu]$. The context  always make it clear whether we are referring to an element of the operad $\mathcal{P}$ or its class in the quotient.

\begin{definition}[Ideal generated by a collection]
\label{def:generated_ideal}
Let $\mathcal{P}$ be a symmetric operad and let $R = \{R(n) \subseteq \mathcal{P}(n)\}_{n \in \mathbb{N}_0}$ be any sub-collection of $\mathcal{P}$.
The \emph{ideal generated} by $R$, denoted by $(R)$, is defined as the intersection of all operadic ideals of $\mathcal{P}$ that contain $R$.
\end{definition}

\begin{remark}
The definition of a generated ideal is well-defined. Indeed, the family of ideals containing $R$ is never empty, since the operad $\mathcal{P}$ itself can be considered an ideal that contains every sub-collection of itself, including $R$.

Furthermore, by Proposition~\ref{prop:intersection_of_ideals}, the intersection of this family of ideals is still an operadic ideal. Consequently, $(R)$ is the \emph{smallest} operadic ideal (with respect to inclusion) of $\mathcal{P}$ that contains $R$.
\end{remark}

\subsection{The Free Symmetric Operad}
\label{subsec:free_operad}

This section is dedicated to the construction of the free symmetric operad, a fundamental concept in the algebraic approach to the theory. First, we  construct a concrete symmetric operad whose elements are decorated trees. Second, we introduce the abstract definition of free operad via its universal property. Finally, we prove that our concrete construction satisfies this property, thereby establishing its existence and uniqueness.

\subsubsection{The Operad of Decorated Trees}
\label{subsec:construction_free_operad}

We begin by constructing an operad whose elements are trees where each internal vertex is labeled, or ``decorated'', by a generator.

\begin{definition}[P-decorated tree]
Let $P = \{P(n)\}_{n \in \mathbb{N}_0}$ be a collection of generators. A \emph{P-decorated tree} is a planar operadic tree where each internal vertex with $k$ children (inputs) is assigned a specific generator $p \in P(k)$. This assignment is called a \emph{decoration}.
\end{definition}

As specified in \cite{samchuckschnarch}, the construction starts from a collection of sets of generators $P = \{P(n)\}_{n \in \mathbb{N}_0}$. Formally, one could require that $P(1)$ contains a distinguished element for the unit. However, since every operad (free or not) must have a unit, by convention, the unit is omitted from the set of generators $P$. It is assumed that the unit element $1$ is a structural part of the operad we are constructing, represented by the \emph{trivial tree} (Definition~\ref{def:albero_banale}), and not a generator in $P$.

\begin{definition}[Operad of decorated trees]
\label{def:construction_trees_set}
Let $P = \{P(n)\}_{n \in \mathbb{N}_0}$ be a collection of generators. The \emph{operad of $P$-decorated trees}, denoted $\mathrm{Tree}(P)$, is a symmetric operad in the category $\texttt{Set}$ defined as follows:
\begin{enumerate}
    \item For each $n \ge 0$, its set of $n$-ary operations, $\mathrm{Tree}(P)(n)$, consists of all $P$-decorated planar operadic trees with $n$ leaves.
    \item The composition $\gamma$ is the grafting of trees.
    \item The action of $\sigma \in \mathbb{S}_n$ on a tree is the permutation of its leaf labels.
    \item The unit is the trivial tree.
\end{enumerate}
\end{definition}

\begin{remark}[Linearization]
The operad $\mathrm{Tree}(P)$ is defined on sets. To obtain an operad in the category of vector spaces, we apply a process of \emph{linearization}. Let $\mathbb{K}$ be a field. We define the linear operad $\mathcal{T}(P)$ as the collection of vector spaces
\begin{equation}
\mathcal{T}(P)(n) \coloneqq \mathbb{K}[\mathrm{Tree}(P)(n)].
\end{equation}
The composition operations and the symmetric group action are extended by linearity to $\mathcal{T}(P)$. Henceforth, $\mathcal{T}(P)$ refers to this linear operad.
\end{remark}

\subsubsection{The Free Operad and its Universal Property}

Now that we have constructed a concrete object, the operad of decorated trees, we introduce the abstract definition of free operad and prove that our construction fulfills this definition.

\begin{definition}[Inclusion functions]
\label{def:inclusion_functions}
For each $n \ge 0$, the \emph{inclusion function} $i_n: P(n) \to \mathcal{T}(P)(n)$ maps a generator $p \in P(n)$ to the element of $\mathcal{T}(P)(n)$ corresponding to the \emph{n-corolla} whose unique internal vertex is decorated with $p$.
\begin{equation}
i_n(p) \coloneqq
 \begin{tikzpicture}[baseline={([yshift=-.5ex]current bounding box.center)}, thick, scale=0.7, font=\tiny]
    \draw (0,-0.7) -- (0,0); 
    \draw (-0.8, 0.8) -- (0,0); 
    \draw (-0.3, 0.8) -- (0,0); 
    \node at (0.2, 0.5) {$\dots$};
    \draw (0.8, 0.8) -- (0,0); 
    \node[fill=white, inner sep=1pt] at (0,0) {$p$};
    \node at (-0.8, 1) {1}; \node at (-0.3, 1) {2}; \node at (0.8, 1) {n};
 \end{tikzpicture}
 .
\end{equation}
\end{definition}

\begin{definition}[Free symmetric operad]
\label{def:free_operad_universal}
Let $P = \{P(n)\}_{n \in \mathbb{N}_0}$ be a collection of generators. A \emph{free symmetric operad} on $P$ is a pair $\bigl(\mathcal{F}(P), i\bigr)$, where $\mathcal{F}(P)$ is a symmetric operad and $i = \{i_n: P(n) \to \mathcal{F}(P)(n)\}_{n \ge 0}$ is a family of maps, that satisfies the following \emph{universal property}:

For any other symmetric operad $\mathcal{Q}$ and any family of maps $f = \{f_n: P(n) \to \mathcal{Q}(n)\}$, there exists a \emph{unique} operad morphism $\tilde{f}: \mathcal{F}(P) \to \mathcal{Q}$  that extends the family $f$, i.e., such that the following diagram commutes for each $n$:
\begin{equation}
\begin{tikzcd}
P(n) \arrow[r, "i_n"] \arrow[dr, "f_n"'] & \mathcal{F}(P)(n) \arrow[d, dashed, "\exists! \tilde{f}_n"] \\
& \mathcal{Q}(n)
\end{tikzcd}
.
\end{equation}
\end{definition}

The following lemma is the key combinatorial result needed to prove that the operad of trees is free.

\begin{lemma}
\label{lemma:tree_decomposition}
Every element of $\mathrm{Tree}(P)$ can be expressed as a finite sequence of compositions of elements of $P$ (viewed as elementary corollas via the inclusion maps), followed by an action of the symmetric group.
\end{lemma}
\begin{proof}
The proof proceeds by induction on the height of the tree $t \in \mathrm{Tree}(P)$.
If $t$ has height 0, it consists only of a root decorated by a generator $p \in P(0)$, so the statement is trivially true. If $t$ has height 1, it is a corolla decorated by an element $p \in P(n)$ with a certain ordering of the leaves, so it is an element of $P$ followed by a permutation.

Now, for the inductive step, assume the statement holds for all trees of height less than or equal to $k$. Let $t$ be a tree of height $k+1$. We can decompose $t$ into its root corolla, $t_{root}$, and a family of subtrees $\{t_1, \dots, t_m\}$ grafted onto its leaves. Each of these subtrees has a height of at most $k$. By the inductive hypothesis, each $t_j$ can be written as a composition of generators. Applying the inductive hypothesis to each subtree thus yields a decomposition of $t$ as a composition of generators, proving the inductive step.
\end{proof}

\begin{remark}[Convention on the ordering of trees]
\label{rem:tree_convention}
The construction of the free operad is based on \emph{planar trees}, which by definition have an intrinsic order on their leaves. By convention, this order is represented in the plane by a left-to-right arrangement. The elementary corolla associated with a generator $p \in P(k)$ is therefore the tree with $k$ leaves labeled from $1$ to $k$. This labeled tree represents the base operation, corresponding to the action of the identity permutation of the symmetric group. The action of any other permutation $\sigma \in \mathbb{S}_k$ produces a new tree, identical in shape but with the leaf labels reordered according to $\sigma$.
\end{remark}

The following theorem makes the connection between our concrete construction and the abstract definition explicit, establishing the existence and uniqueness of the free operad.

\begin{theorem}[Existence and uniqueness of the free operad]
\label{thm:existence_free_operad}
For any collection of generators $P = \{P(n)\}_{n \in \mathbb{N}_0}$, the free operad on $P$ exists and is unique up to a \emph{unique} isomorphism.
\end{theorem}

\begin{proof}
The proof is divided into two parts: the uniqueness of the free operad (up to isomorphism), and the existence, which is shown by verifying that our construction $\mathcal{T}(P)$ satisfies the universal property.
\begin{itemize}
\item \emph{Uniqueness (up to isomorphism).}
The proof of uniqueness is a standard argument for objects defined by a universal property. Suppose there exist two free operads on $P$, say $\bigr(\mathcal{F}(P), i\bigl)$ and $\bigr(\mathcal{F}'(P), i'\bigl)$, both satisfying the universal property.

The universal property of $\mathcal{F}(P)$ guarantees the existence of a \emph{unique} operad morphism $f: \mathcal{F}(P) \to \mathcal{F}'(P)$ such that $f \circ i = i'$ and the existence of a \emph{unique} operad morphism $g: \mathcal{F}'(P) \to \mathcal{F}(P)$ such that $g \circ i' = i$.

Consider the compositions $g \circ f$ and $f \circ g$.
The morphism $g \circ f: \mathcal{F}(P) \to \mathcal{F}(P)$ extends the family $i$:
\begin{equation}
(g \circ f) \circ i = g \circ (f \circ i) = g \circ i' = i. 
\end{equation}
The identity morphism $\mathrm{id}_{\mathcal{F}(P)}: \mathcal{F}(P) \to \mathcal{F}(P)$ also extends $i$. Since the universal property guarantees that such a morphism is \emph{unique}, we conclude that $g \circ f = \mathrm{id}_{\mathcal{F}(P)}$.

A completely analogous argument shows that $f \circ g = \mathrm{id}_{\mathcal{F}'(P)}$.
Having constructed two mutually inverse morphisms, we conclude that $f$ is an isomorphism and thus the two free operads are isomorphic.

\item \emph{Existence.} We now verify that the operad $\mathcal{T}(P)$, equipped with the inclusion functions $i_n$ from Definition \ref{def:inclusion_functions}, satisfies the universal property. Let $f = \{f_n\}_{n \ge 0}$ be a family of maps from $P$ to an operad $\mathcal{Q}$. We must construct a unique operad morphism $\tilde{f}: \mathcal{T}(P) \to \mathcal{Q}$ such that $\tilde{f} \circ i_n = f_n$ for each $n$ and prove its uniqueness.
\begin{itemize}
\item   \emph{Construction of $\tilde{f}$:}
        We define $\tilde{f}$ on the basis of $\mathcal{T}(P)$, i.e., on the trees. For a generator tree (a corolla $i_k(p)$), the condition $\tilde{f} \circ i_k = f_k$ forces the definition:
        \begin{equation} \tilde{f}\bigl(i_k(p)\bigr) \coloneqq f_k(p) \in \mathcal{Q}(k). 
        \end{equation}
        For any other tree $T$, Lemma \ref{lemma:tree_decomposition} guarantees that it can be uniquely decomposed as the grafting of a family of subtrees $\{T_1, \dots, T_k\}$ onto its root corolla $T_0$. Since $T = \gamma(T_0; T_1, \dots, T_k)$, and since $\tilde{f}$ must be an operad morphism that preserves composition, we are forced to define $\tilde{f}$ recursively:
        \begin{equation} \tilde{f}\big(\gamma(T_0; T_1, \dots, T_k)\big) \coloneqq \gamma_{\mathcal{Q}}\big(\tilde{f}(T_0); \tilde{f}(T_1), \dots, \tilde{f}(T_k)\big). 
        \end{equation}
        This definition is extended by linearity to the entire space $\mathcal{T}(P)$ and, by construction, defines an operad morphism.
        
\item \emph{Uniqueness of $\tilde{f}$:}
Let $g: \mathcal{T}(P) \to \mathcal{Q}$ be another operad morphism that extends $f$, i.e., such that $g \circ i = f$. We  show that $g$ must necessarily coincide with $\tilde{f}$.

Since $\tilde{f}$ and $g$ are linear morphisms, it is sufficient to show that they coincide on the basis elements of $\mathcal{T}(P)$, i.e., on all decorated trees.
\begin{itemize}
    \item \emph{Action on generators:}
    By hypothesis, both $g$ and $\tilde{f}$ extend $f$. Consequently, their images coincide on the corollas, which are the images of the generators via the inclusion maps $i_k$:
    \begin{equation} g\bigl(i_k(p)\bigr) = f_k(p) = \tilde{f}\bigl(i_k(p)\bigr), \quad \forall p \in P(k). 
    \end{equation}
    \item \emph{Action on composite trees:}
    Every tree $T$ with more than one internal vertex is obtained by a finite sequence of compositions starting from corollas (Lemma \ref{lemma:tree_decomposition}). Since both $g$ and $\tilde{f}$ are operad morphisms, they preserve composition.
    
    The action of a morphism on a composite tree is uniquely determined by its action on the component subtrees. Since $g$ and $\tilde{f}$ coincide on the corollas and preserve composition, they coincide on any tree.
\end{itemize}

Since $g$ and $\tilde{f}$ coincide on all basis elements of $\mathcal{T}(P)$, they are the same morphism. Therefore, $\tilde{f}$ is unique.
\end{itemize}

Since we have constructed a unique morphism $\tilde{f}$ for any given family of maps $f$, the operad $\mathcal{T}(P)$ satisfies the universal property. This completes the proof of existence.

\end{itemize}
\end{proof}

\subsection{Presentations of an Operad}
This section applies the algebraic tools of ideals and quotients, developed in Section~\ref{sec:ideals_and_quotients}, to the free operad constructed above. By combining the concept of a free object (which has no relations beyond the axioms) with that of a quotient (which imposes relations), we formalize the powerful idea of defining an operad via \emph{generators and relations}.

\begin{definition}[Operad via generators and relations]
\label{def:operad_generated}
Let $P = \{P(n)\}_{n \in \mathbb{N}_0}$ be a collection of generators and let $\mathcal{T}(P)$ be the corresponding free operad. Let $R = \{R(n) \subseteq \mathcal{T}(P)(n)\}_{n \in \mathbb{N}_0}$ be a collection of subsets, called \emph{relations}. The pair $(P,R)$ is called a \emph{presentation data}.
The \emph{operad generated by $P$ subject to $R$} is the quotient operad:
\begin{equation}
\langle P \mid R \rangle \coloneqq \frac{\mathcal{T}(P)}{(R)},
\end{equation}
where $(R)$ is the operadic ideal generated by $R$.
\end{definition}

\begin{remark}[Form of the relations]
In practice, relations are not usually specified as vectors, but more intuitively as equalities, for example $v=w$. An equality of this type is interpreted as the relation given by the vector $v-w$. In the quotient $\mathcal{T}(P)/(R)$, the equivalence class $[v-w]$ becomes zero, thus forcing the equality $[v]=[w]$.
\end{remark}

\begin{definition}[Presentation and quadratic operads]
\label{def:presentation_and_quadratic}
Let $P = \{P(n)\}_{n \in \mathbb{N}_0}$ be a collection of generators.
\begin{enumerate}
    \item A \emph{presentation} of a symmetric operad $\mathcal{P}$ is a presentation data $(P,R)$ together with an operad isomorphism $\mathcal{P} \cong \langle P \mid R\rangle$. An operad equipped with a presentation is said to be \emph{presented}.
    
    \item The free operad $\mathcal{T}(P)$ admits a natural grading, called \emph{weight}, given by the number of internal vertices (i.e., the number of generators from $P$) that make up each tree. We denote by $\mathcal{T}(P)^{(k)}$ the vector subspace of $\mathcal{T}(P)$ spanned by trees of weight $k$.
    
    \item A presentation data $(P,R)$ constitutes a set of \emph{quadratic data} if all relations are of weight 2, i.e., if $R \subseteq \mathcal{T}(P)^{(2)}$.
    
    \item An operad is called \emph{quadratic} if it can be presented by quadratic data, i.e., if it is of the form
    \begin{equation}
    \langle P \mid R \rangle \quad \text{with} \quad R \subseteq \mathcal{T}(P)^{(2)}.
    \end{equation}
\end{enumerate}
\end{definition}

In other words, an operad is quadratic if all its defining relations involve compositions of exactly two generators. 

\begin{remark}[Generators as Sets vs. $\mathbb{S}$-sets]
It is important to note that the approach presented here, which starts from a collection of sets of generators $P$, is a simplified version adopted for pedagogical reasons.

A more formal treatment would begin with an \emph{$\mathbb{S}$-set} $E = \{E(n)\}_{n \in \mathbb{N}_0}$. The fundamental difference is that each component $E(n)$ is already equipped with a right action of the symmetric group $\mathbb{S}_n$. This action defines the \emph{intrinsic symmetries} of the generators, prior to any composition.

For the purposes of this review, starting from a simple collection of generators and imposing all symmetries through relations is entirely adequate.
\end{remark}

We now state a fundamental result which establishes that the constructive approach via generators and relations is fully general.

\begin{theorem}
\label{thm:every_operad_has_presentation}
Every symmetric operad $\mathcal{P}$ admits a presentation.
\end{theorem}

\begin{proof}
The proof is constructive. We provide a \emph{canonical presentation} for a given operad $\mathcal{P}$.
\begin{itemize}
\item \emph{Choice of generators:} We choose as the set of generators $P$ the operad $\mathcal{P}$ itself, viewed as a collection of sets. That is, we set $P(n) \coloneqq \mathcal{P}(n)$ for each $n \ge 0$.

\item \emph{Choice of relations:} The relations $R$ must ensure that the composition in the quotient operad coincides with that of the operad $\mathcal{P}$.

The idea is to impose, for every possible composition in $\mathcal{P}$, a relation in the free operad. Consider any valid composition in $\mathcal{P}$:
\begin{equation}
p' = \gamma_{\mathcal{P}}(p_0; p_1, \dots, p_k). 
\end{equation}
We want this rule to hold in $\langle P \mid R \rangle$ as well. To do this, we translate both sides of this equality into elements of $\mathcal{T}(P)$.

\begin{itemize}
    \item The left-hand side, $p'$, is an element of $\mathcal{P}$ and therefore a generator. Its correspondent in $\mathcal{T}(P)$ is the corolla $i_m(p')$, where $m$ is the arity of $p'$.

    \item The right-hand side is a composition. The composition $\gamma_{\mathcal{T}(P)}$ acts on elements of $\mathcal{T}(P)$, not directly on the generators in $P$. We must therefore first map the generators $p_0, \dots, p_k$ into $\mathcal{T}(P)$ via their respective inclusion functions, obtaining the corollas $i_{n_0}(p_0), \dots, i_{n_k}(p_k)$, which are the elements upon which the composition in $\mathcal{T}(P)$ can act.
\end{itemize}

Therefore, the formal relation expressing the desired equality is the difference between these two trees in $\mathcal{T}(P)$. We define $R$ as the set of all elements of the form:
\begin{equation}
    \gamma_{\mathcal{T}(P)}\big(i_{n_0}(p_0); i_{n_1}(p_1), \dots, i_{n_k}(p_k)\big) - i_m(p')
\end{equation}
where $n_j$ is the arity of the generator $p_j$.

\item \emph{Construction of the isomorphism:} With this presentation data $(P,R)$, we need to show that $\mathcal{P} \cong \langle P \mid R \rangle$.
    Consider the family of identity maps on the generators, $f_n: P(n) \to \mathcal{P}(n)$, given by $f_n(p) = p$. By the universal property of the free operad, this extends to a unique operad morphism $\tilde{f}: \mathcal{T}(P) \to \mathcal{P}$. This morphism is surjective, since the image of the generators is all of $\mathcal{P}$.

    By the way we defined the relations $R$, the kernel of $\tilde{f}$ is exactly the ideal $(R)$. Since $\tilde{f}$ vanishes on the ideal $(R)$, the universal property of the quotient (Proposition \ref{prop:universal_property_quotient}) guarantees that $\tilde{f}$ induces a unique operad morphism $\varphi: \langle P \mid R \rangle \to \mathcal{P}$.
    
    Since $\tilde{f}$ is surjective and we have quotiented by exactly its kernel, the induced morphism $\varphi$ is both injective and surjective, and therefore is an isomorphism. We thus obtain:
    \begin{equation}
    \langle P \mid R \rangle = \frac{\mathcal{T}(P)}{(R)} \cong \mathcal{P}.
    \end{equation}
    This establishes the desired isomorphism and completes the proof.
    \end{itemize}
\end{proof}

Let us apply the notions learned in this section to define $\mathsf{As}$ and $\mathsf{Com}$ via generators and relations.

\begin{example}[Presentation of $\mathsf{As}$ as a quadratic operad]
The operad $\mathsf{As}$ encodes the structure of associative algebras.
\begin{itemize}
        \item \emph{Generators:} we start with a collection $P$ with a single binary generator $\mu \in P(2)$. Its representation as an elementary tree is the 2-corolla:
        \begin{equation}
        \mu = 
        \begin{tikzpicture}[baseline={([yshift=-.5ex]current bounding box.center)}, thick, scale=0.6]
            \draw (0,0) -- (0,0.5);
            \draw (-0.3, 1) -- (0,0.5) -- (0.3, 1);
            \node at (-0.3,1.2) {\tiny 1}; \node at (0.3,1.2) {\tiny 2};
        \end{tikzpicture}
        .
        \end{equation}

        \item \emph{Relations:} the only relation is associativity. This is expressed as an equality between two ternary operations. $R_{\text{assoc}} = \mu \circ (\mu, 1) - \mu \circ (1, \mu)$, where $1 \in \mathcal{T}(P)(1)$ is the unit element of the free operad, represented by the trivial tree. Graphically, the relation is:
        \begin{equation}
        \begin{tikzpicture}[baseline={([yshift=-.5ex]current bounding box.center)}, thick, scale=0.6]
            \draw (0,0) -- (0,0.5); \draw (0,0.5) -- (-0.5,1);
            \draw (-0.5,1) -- (-0.8,1.5); \draw (-0.5,1) -- (-0.2,1.5);
            \draw (0,0.5) -- (0.5,1); \draw (0.5,1) -- (0.5,1.5);
            \node at (-0.8, 1.7) {\tiny 1}; \node at (-0.2, 1.7) {\tiny 2}; \node at (0.5, 1.7) {\tiny 3};
        \end{tikzpicture}
        \quad = \quad
        \begin{tikzpicture}[baseline={([yshift=-.5ex]current bounding box.center)}, thick, scale=0.6]
            \draw (0,0) -- (0,0.5); \draw (0,0.5) -- (0.5,1);
            \draw (0.5,1) -- (0.2,1.5); \draw (0.5,1) -- (0.8,1.5);
            \draw (0,0.5) -- (-0.5,1) -- (-0.5,1.5);
            \node at (-0.5, 1.7) {\tiny 1}; \node at (0.2, 1.7) {\tiny 2}; \node at (0.8, 1.7) {\tiny 3};
        \end{tikzpicture}
        .
        \end{equation}
        \end{itemize}
\end{example}
\begin{remark}
    The relation $R_{\text{assoc}}$ has weight 2. Therefore, the pair $(P, \{R_{\text{assoc}}\})$ constitutes a set of \emph{quadratic data}.
\end{remark}
This formalization of a single associative operation allows us to state the following proposition.
    \begin{proposition}
        $\mathsf{As}$ is presented by $(P,R_\text{assoc})$.
    \end{proposition}

\begin{example}[Presentation of $\mathsf{Com}$ as a quadratic operad]
\label{ex: Com_presentation}
The operad $\mathsf{Com}$ encodes commutative associative algebras. 
\begin{itemize}
\item \emph{Generators:} We start again with a binary generator $\mu \in P(2)$.
    
\item \emph{Relations:} Two types of relations are imposed:
    \begin{itemize}
        \item \emph{Associativity:} $R_{\text{assoc}} = \mu \circ (\mu, 1) - \mu \circ (1, \mu)$.
        \item \emph{Commutativity:} $R_{\text{comm}} = \mu - \mu \cdot \tau = 0$, where $\tau \in \mathbb{S}_2$ is the non-trivial permutation. Graphically:
        \begin{equation}
         \begin{tikzpicture}[baseline={([yshift=-.5ex]current bounding box.center)}, thick, scale=0.6, font=\tiny]
        \draw (0,0) -- (0,0.5);
        \draw (-0.4, 1.2) -- (0,0.5) -- (0.4, 1.2);
        \node at (-0.4,1.4) {1}; \node at (0.4,1.4) {2};
        \end{tikzpicture}
        \quad = \quad
        \begin{tikzpicture}[baseline={([yshift=-.5ex]current bounding box.center)}, thick, scale=0.6, font=\tiny]
        \draw (0,0) -- (0,0.5);
        \draw (-0.4, 1.2) -- (0,0.5) -- (0.4, 1.2);
        \node at (-0.4,1.4) {2}; \node at (0.4,1.4) {1};
        \end{tikzpicture}
        .
        \end{equation}
        \end{itemize}
\end{itemize}
\end{example}
By defining the set of relations as $R = \{R_{\text{assoc}}, R_{\text{comm}}\}$, we impose the constraints required for a commutative and associative algebra, a structure that is formalized in the following proposition.
    \begin{proposition}
        $\mathsf{Com}$ is presented by $(P,R)$, where $R = \{R_{\text{assoc}},R_{\text{comm}}\}$.
    \end{proposition}

\begin{remark}
The commutativity relation $R_{\text{comm}}$ has weight 1 (it involves only one generator per term), so the presentation of $\mathsf{Com}$ of Example \ref{ex: Com_presentation} is not, strictly speaking, quadratic.

However, it is possible to obtain an equivalent quadratic presentation. Instead of starting from a simple set of generators, one starts from an $\mathbb{S}$-set in which the generator $\mu$ is symmetric. In this case, commutativity is "intrinsic" to the generator. The only relation still required is associativity, which is of weight 2, making the presentation quadratic. Both approaches yield the same operad.
\end{remark}

The following lemma shows that to define a morphism from a presented operad, we only need to specify its action on the generators and verify that this action is compatible with the relations. This simplifies the process considerably.

\begin{lemma}
\label{lem:uniqueness_of_morphism_presented}
Let $\mathcal{P}$ be an operad given by a presentation $\langle P \mid R \rangle$ and let $\mathcal{Q}$ be another operad. A morphism of operads $\varphi: \mathcal{P} \to \mathcal{Q}$ is uniquely determined by its action on $P$.
\end{lemma}

\begin{proof}
Let $\varphi, \psi: \mathcal{P} \to \mathcal{Q}$ be two operad morphisms. Suppose that $\varphi(i_k(p)) = \psi(i_k(p))$ for every generator $p \in P(k)$ and for every $k \ge 0$. We are required to prove that $\varphi = \psi$.

Since $\varphi$ and $\psi$ are linear maps, to prove their equality it is sufficient to show that they coincide on a \emph{set of generators} of the vector space $\mathcal{P}(n)$, for each $n \ge 0$.
By definition, $\mathcal{P}(n) = \mathcal{T}(P)(n)/(R)(n)$. Since the set of trees $\mathrm{Tree}(P)(n)$ is a basis for $\mathcal{T}(P)(n)$, the set of their equivalence classes, $\{[T] \mid T \in \mathrm{Tree}(P)(n)\}$, is a set of generators for the quotient space $\mathcal{P}(n)$.
Therefore, it suffices to show that $\varphi([T]) = \psi([T])$ for every tree $T$.

Consider an arbitrary tree $T \in \mathcal{T}(P)$. By Lemma \ref{lemma:tree_decomposition}, $T$ is constructed through a finite number of compositions starting from corollas (the images of the generators $i_k(p)$).
Let us analyze the action of $\varphi$ on $[T]$. Since $\varphi$ is a morphism, it must preserve composition. This means that the image of $\varphi$ on a tree built by composition is determined by the image of $\varphi$ on its component subtrees.

For example, if $T = \gamma(T_0; T_1, \dots, T_k)$, then:
\begin{equation}
\varphi([T]) = \varphi([\gamma(T_0; \dots, T_k)]) = \gamma_{\mathcal{Q}}\bigl(\varphi([T_0]); \dots, \varphi([T_k])\bigr).
\end{equation}
Similarly,
\begin{equation}
\psi([T]) = \psi([\gamma(T_0; \dots, T_k)]) = \gamma_{\mathcal{Q}}\bigl(\psi([T_0]); \dots, \psi([T_k])\bigr).
\end{equation}
Applying this reasoning recursively, the action of $\varphi$ and $\psi$ on $T$ can be reduced to their action on the corollas that make up $T$.
By our initial hypothesis, $\varphi$ and $\psi$ coincide on all corollas. Thus, $\varphi([T]) = \psi([T])$.
By linearity, we conclude that $\varphi = \psi$.
\end{proof}

The previous lemma guarantees that a morphism from a presented operad is uniquely determined by its action on the generators. The following lemma provides the existence condition for such a morphism, playing a role analogous to the Homomorphism Theorem for presented operads. It establishes a practical criterion for verifying when a map on generators gives rise to a well-defined morphism on the quotient.

\begin{lemma}
\label{lem:existence_of_morphism_presented}
Let $\mathcal{P}$ be an operad given by a presentation $\langle P \mid R \rangle$ and let $\mathcal{Q}$ be another operad. A family of maps $f_n: P(n) \to \mathcal{Q}(n)$ extends to a \emph{unique} operad morphism $\varphi: \mathcal{P} \to \mathcal{Q}$ if and only if the induced morphism $\tilde{f}:\mathcal{T}(P) \to \mathcal{Q}$ vanishes on $R$.
\end{lemma}

\begin{proof}
The argument is a direct application of the universal properties of the free operad and the quotient operad, applied in succession.
Given the family of maps $f_n$, the universal property of the free operad $\mathcal{T}(P)$ (Theorem \ref{thm:existence_free_operad}) guarantees the existence of a unique morphism $\tilde{f}: \mathcal{T}(P) \to \mathcal{Q}$.

The universal property of the quotient $\mathcal{P} = \mathcal{T}(P)/(R)$ (Proposition \ref{prop:universal_property_quotient}) states that $\tilde{f}$ induces a morphism $\varphi: \mathcal{P} \to \mathcal{Q}$ if and only if $\tilde{f}$ vanishes on the ideal $(R)$, which is equivalent to requiring that it vanishes on the generators of the ideal, i.e., on $R$.

The uniqueness of the morphism $\varphi$ is a direct consequence of Lemma \ref{lem:uniqueness_of_morphism_presented}.

\end{proof}

\subsection{The Operad $\mathsf{Lie}$}
\label{subsec:lie_operad}

We now apply the formalism of generators and relations to construct one of the most important operads in algebra: the $\mathsf{Lie}$ operad, which encodes the structure of Lie algebras.

\begin{definition}[Presentation of the $\mathsf{Lie}$ operad]
The $\mathsf{Lie}$ operad is defined by the following data:
\begin{enumerate}
    \item \emph{Generators:} the collection of generators $P$ contains a single binary operation, denoted by $\ell$:
    \begin{equation}
    P(2) = \{\ell\}, \ \ P(n)= \varnothing \ \ \text{if} \ n \neq 2.
    \end{equation}
    
    \item \emph{Relations:} two relations, $R_{anti}$ and $R_{Jac}$, are imposed.
    \begin{itemize}
        \item The \emph{anticommutativity} relation ($R_{anti}$):
        \begin{equation}
        \ell + \ell \cdot \tau = 0,
        \end{equation}
        where $\tau \in \mathbb{S}_2$ is the non-trivial permutation.
        
        \item The \emph{Jacobi identity} ($R_{Jac}$). Setting 
        \begin{equation}
            \sigma = \begin{pmatrix} 1 & 2 & 3 \\ 2 & 3 & 1 \end{pmatrix} \in \mathbb{S}_3,
        \end{equation}
        the relation is:
        \begin{equation}
                \ell \circ (\ell, 1) + (\ell \circ (\ell, 1)) \cdot \sigma + (\ell \circ (\ell, 1)) \cdot \sigma^2 = 0.
        \end{equation}
    \end{itemize}
    We define the family of relations as $R \coloneqq \{R_{anti}, R_{Jac}\}$.
\end{enumerate}
The operad is then defined as the quotient operad:
\begin{equation}
\mathsf{Lie} \coloneqq \langle P \mid R \rangle.
\end{equation}
\end{definition}

\begin{remark}[Visualization of the $\mathsf{Lie}$ relations]
The two relations that define the $\mathsf{Lie}$ operad can be represented graphically as follows:

\begin{itemize}
    \item The \emph{anticommutativity} relation $l + l \cdot \tau = 0$ is equivalent to:
    \begin{equation}
    \begin{tikzpicture}[baseline={([yshift=-.5ex]current bounding box.center)}, thick, scale=0.7, font=\tiny]
        \draw (0,0) -- (0,0.7);
        \draw (-0.4, 1.4) -- (0,0.7) -- (0.4, 1.4);
        \node at (-0.4,1.6) {1}; \node at (0.4,1.6) {2};
            \node[fill=white, inner sep=1pt] at (0,0.7) {$l$};
    \end{tikzpicture}
    \quad = \quad - \quad
    \begin{tikzpicture}[baseline={([yshift=-.5ex]current bounding box.center)}, thick, scale=0.7, font=\tiny]
        \draw (0,0) -- (0,0.7);
        \draw (-0.4, 1.4) -- (0,0.7) -- (0.4, 1.4);
        \node at (-0.4,1.6) {2}; \node at (0.4,1.6) {1};
            \node[fill=white, inner sep=1pt] at (0,0.7) {$l$};
    \end{tikzpicture}
    .
    \end{equation}

    \item The \emph{Jacobi identity} is visualized as:
    \begin{equation}
    \begin{tikzpicture}[baseline={([yshift=-.5ex]current bounding box.center)}, thick, scale=0.7, font=\tiny]
        \draw (0,0) -- (0,0.5) -- (-0.5,1.5);
        \draw (0,0.5) -- (0.5,1.5) --(0.5,2.2); \node at (0.5,2.4) {3};
        \draw (-0.5,1.5) -- (-0.8,2.2); \node at (-0.8,2.4) {1};
        \draw (-0.5,1.5) -- (-0.2,2.2); \node at (-0.2,2.4) {2};
            \node[fill=white, inner sep=1pt] at (0,0.5) {$l$}; \node[fill=white, inner sep=1pt] at (-0.5,1.5) {$l$};
    \end{tikzpicture}
    \quad + \quad
    \begin{tikzpicture}[baseline={([yshift=-.5ex]current bounding box.center)}, thick, scale=0.7, font=\tiny]
        \draw (0,0) -- (0,0.5) -- (-0.5,1.5);
        \draw (0,0.5) -- (0.5,1.5) -- (0.5,2.2); \node at (0.5,2.4) {1};
        \draw (-0.5,1.5) -- (-0.8,2.2); \node at (-0.8,2.4) {2};
        \draw (-0.5,1.5) -- (-0.2,2.2); \node at (-0.2,2.4) {3};
            \node[fill=white, inner sep=1pt] at (0,0.5) {$l$}; \node[fill=white, inner sep=1pt] at (-0.5,1.5) {$l$};
    \end{tikzpicture}
    \quad + \quad
    \begin{tikzpicture}[baseline={([yshift=-.5ex]current bounding box.center)}, thick, scale=0.7, font=\tiny]
        \draw (0,0) -- (0,0.5) -- (-0.5,1.5);
        \draw (0,0.5) -- (0.5,1.5) -- (0.5,2.2); \node at (0.5,2.4) {2};
        \draw (-0.5,1.5) -- (-0.8,2.2); \node at (-0.8,2.4) {3};
        \draw (-0.5,1.5) -- (-0.2,2.2); \node at (-0.2,2.4) {1};
                \node[fill=white, inner sep=1pt] at (0,0.5) {$l$}; \node[fill=white, inner sep=1pt] at (-0.5,1.5) {$l$};
    \end{tikzpicture}
    \quad = \quad 0.
    \end{equation}
\end{itemize}
\end{remark}

As with the operads $\mathsf{As}$ and $\mathsf{Com}$, the presentation of $\mathsf{Lie}$ via generators and relations is designed to encode the structure of Lie algebras. 

The proof of the main isomorphism theorem relies on constructing a pair of functors between the two categories. The following two propositions are dedicated to establishing the existence and properties of these functors.

\begin{proposition}
\label{prop:functor_F_Lie}
The correspondence $F: \emph{\texttt{Alg}}(\mathsf{Lie}) \to \emph{\texttt{LieAlg}}_{\mathbb{K}}$, defined on objects by 
\begin{equation}
    F(V,\Phi) \coloneqq \bigl(V, [-,-]_V \coloneqq \Phi_2(\ell)\bigr)
\end{equation} 
and on morphisms by 
\begin{equation}
    F(f)\coloneqq f ,
\end{equation}
is a well-defined functor.
\end{proposition}

\begin{proof}
We verify that $F$ is well-defined on objects and morphisms, and satisfies the functorial axioms.
\begin{enumerate}
    \item \emph{Well-definedness on objects.} We first show that for any $\mathsf{Lie}$-algebra $(V, \Phi)$, the pair $(V, [-,-]_V)$ is a Lie algebra. Bilinearity is obvious since $\Phi_2(\ell) \in \End_V(2)$. We verify the two axioms of Lie algebras.
    \begin{itemize}
        \item \emph{Anticommutativity:} let $\tau \in \mathbb{S}_2$ be the non-trivial transposition. For any $x,y \in V$:
        \begin{equation}
        \begin{split}
            [x,y]_V + [y,x]_V &= \Phi_2(\ell)(x \otimes y) + \Phi_2(\ell)(y \otimes x) \\
            &\overset{(1)}{=} \Phi_2(\ell)(x \otimes y) + \bigl(\Phi_2(\ell) \cdot \tau\bigr)(x \otimes y) \\
            &\overset{(2)}{=} \Phi_2(\ell + \ell \cdot \tau)(x \otimes y) \\
            &= \Phi_2(R_{anti})(x \otimes y) = \Phi_2(0) = 0,
        \end{split}
        \end{equation}
        where
        \begin{itemize}
            \item[(1)]\quad follows from the action of $\tau$ on $\End_V$,
            \item[(2)]\quad follows from the linearity and compatibility with the symmetric group action of $\Phi_2$.
        \end{itemize}
        \item \emph{Jacobi identity:} let $t = \ell \circ (\ell,1)$. For any $x,y,z \in V$:
        \begin{equation}
        \begin{split}
            &[[x,y]_V, z]_V + [[y,z]_V, x]_V + [[z,x]_V, y]_V \\
            &= \Phi_3(t)(x \otimes y \otimes z) + \Phi_3(t \cdot \sigma)(x \otimes y \otimes z) + \Phi_3(t \cdot \sigma^2)(x \otimes y \otimes z) \\
            &\overset{(3)}{=} \Phi_3\big( t + t \cdot \sigma + t \cdot \sigma^2 \big)(x \otimes y \otimes z) \\
            &= \Phi_3(R_{Jac})(x \otimes y \otimes z) = \Phi_3(0) = 0,
        \end{split}
        \end{equation}
        where (3) follows from the linearity and compatibility with the symmetric group action of $\Phi_3$.
    \end{itemize}
    Thus $(V,[-,-]_V)$ is a Lie algebra.

    \item \emph{Well-definedness on morphisms.} If $f: (V, \Phi^V) \to (W, \Phi^W)$ is a morphism of $\mathsf{Lie}$-algebras, then $f$ is a morphism of Lie algebras. The proof is sufficient to use arguments analogous to those in the proof of Proposition \ref{prop:functor_F_As}.

    \item \emph{Functorial properties.} The functoriality is immediate, as in the previous cases for the operads $\mathsf{As}$ and $\mathsf{Com}$, since the action on morphisms is the identity map.
\end{enumerate}
Therefore, $F$ is a well-defined functor.
\end{proof}

\begin{proposition}[Functor from Lie Algebras to $\mathsf{Lie}$-algebras]
\label{prop:functor_G_Lie}
The correspondence $G: \emph{\texttt{LieAlg}}_{\mathbb{K}} \to \emph{\texttt{Alg}}(\mathsf{Lie})$, defined on objects by 
\begin{equation}
    G(\mathfrak{g}, [-,-]) \coloneqq (\mathfrak{g}, \Phi)
\end{equation}
and on morphisms by
\begin{equation}
    G(g)\coloneqq g,
\end{equation}
is a well-defined functor.
\end{proposition}

\begin{proof}
The proof consists of three steps: verifying that the functor is well-defined on objects, on morphisms, and that it satisfies the functorial properties.
\begin{enumerate}
    \item \emph{Well-definedness on objects.} We show that every Lie algebra $(\mathfrak{g}, [-,-])$ can be uniquely endowed with a $\mathsf{Lie}$-algebra structure $(\mathfrak{g}, \Phi)$. This requires proving the existence and uniqueness of an operad morphism $\Phi: \mathsf{Lie} \to \End_{\mathfrak{g}}$ that maps the generator $\ell$ to the Lie bracket $[-,-]$.
    
    We use Lemma \ref{lem:existence_of_morphism_presented}. The operad $\mathsf{Lie}$ is given by the presentation $\langle P \mid R \rangle$, where $P=\{\ell\}$ and $R=\{R_{anti}, R_{Jac}\}$ are the anticommutativity and Jacobi relations. We define a map on the generators $f: P(2) \to \End_{\mathfrak{g}}(2)$ by setting $f(\ell) \coloneqq [-,-]$. By Lemma \ref{lem:existence_of_morphism_presented}, this map extends to a unique operad morphism $\Phi$ if and only if the image, $[-,-]$, satisfies the relations $R$ in the operad $\End_{\mathfrak{g}}$. This is true, since:
    \begin{itemize}
        \item The bracket is anticommutative, thus satisfying the relation $R_{anti}$.
        \item The bracket satisfies the Jacobi identity, thus satisfying the relation $R_{Jac}$.
    \end{itemize}
    Therefore, the morphism $\Phi$ exists and is unique, and $G$ is well-defined on objects.

    \item \emph{Well-definedness on morphisms.} Let $g: (\mathfrak{g}, [-,-]_{\mathfrak{g}}) \to (\mathfrak{h}, [-,-]_{\mathfrak{h}})$ be a morphism of Lie algebras. We are required to verify that $g$ is a morphism of $\mathsf{Lie}$-algebras from $G(\mathfrak{g})$ to $G(\mathfrak{h})$. Let $\Phi^{\mathfrak{g}}$ and $\Phi^{\mathfrak{h}}$ be the $\mathsf{Lie}$-algebra structures on $\mathfrak{g}$ and $\mathfrak{h}$. The condition is $g \circ \Phi^{\mathfrak{g}}_2(\ell) = \Phi^{\mathfrak{h}}_2(\ell) \circ g^{\otimes 2}$. By construction, this condition becomes the equation  $g([x,y]_{\mathfrak{g}}) = [g(x), g(y)]_{\mathfrak{h}}$, which holds by hypothesis, as it is the definition of a Lie algebra morphism.
    
    \item \emph{Functorial properties.} Functoriality follows immediately because the action of $G$ on morphisms is the identity map. It preserves identity morphisms $\bigl(G(\mathrm{id}_\mathfrak{g}) = \mathrm{id}_\mathfrak{g} = \mathrm{id}_{G(\mathfrak{g})}\bigr)$ and composition $\bigl(G(g \circ f) = g \circ f = G(g) \circ G(f)\bigr)$.
\end{enumerate}
Since all conditions are met, $G$ is a well-defined functor.
\end{proof}

Having constructed the two functors, we can state and prove the isomorphism theorem.

\begin{theorem}
\label{th:lie_isomorphism}
The category of $\mathsf{Lie}$-algebras is isomorphic to the category of Lie algebras. This isomorphism is established by the functors defined in Propositions \ref{prop:functor_F_Lie} and \ref{prop:functor_G_Lie}.
\end{theorem}

\begin{proof}
Propositions \ref{prop:functor_F_Lie} and \ref{prop:functor_G_Lie} establish that the functors $F$ and $G$ are well-defined. A direct verification shows that they are mutually inverse, since their compositions yield the identity functors on the respective categories. This provides the isomorphism of categories.
\end{proof}

\subsection{A Unified Perspective: Operads as Algebraic Theories}

The formalism of generators and relations, developed in this section, provides a universal ``algorithm'' for translating any \emph{type of algebra}\footnote{This category includes associative, commutative, and Lie algebras, which are the central examples of this review. Conversely, structures defined by operations with multiple outputs (like co-algebras) or whose operations are not generally multilinear (like groups) cannot be described by this formalism. While more general tools like PROPs are needed to handle multi-output operations, they are still unsuitable for describing groups, whose inversion map is not linear.} (defined by multilinear operations with an arbitrary number of inputs and a single output) into an operad.

The central idea is that every algebraic theory can be encoded through the following systematic procedure:
\begin{enumerate}
    \item \emph{Extract the generators:} one identifies the fundamental operations of the theory. For associative algebras, it is a binary product; for Lie algebras, it is a binary bracket. These fundamental operations form the collection of generators $P$.
    
    \item \emph{Extract the relations:} one translates the axioms of the theory (such as associativity, the Jacobi identity, anticommutativity) into relations $R$.
    
    \item \emph{Construct the operad:} the operad that encodes the algebraic theory is $\langle P \mid R \rangle$.
\end{enumerate}
This framework culminates in the fundamental result that the category of algebras defined in the classical way and the category of algebras over the operad $\langle P \mid R \rangle$ (i.e., its representations) are isomorphic \cite{hilgerponcin}.

This means that an operad is much more than a simple structure: it is an \emph{algebraic theory}, a single object that contains all the information about the operations, symmetries, and relations of a given species of algebras.

\section{Later Developments and Applications}
\label{chap:developments}

In the previous sections, we introduced the notion of an operad and showed how it encodes several fundamental algebraic structures. However, operad theory is much broader and extends far beyond the examples covered. The flexibility of the operadic formalism has allowed for the development of numerous generalizations, designed to describe increasingly complex algebraic and geometric contexts. In this concluding section, we  present a brief overview of some of these later developments, without pretending to be exhaustive, but with the aim of illustrating the theory's broader scope and power.

We  begin with \emph{coloured operads}, a direct extension of operad theory that provides a framework for describing algebraic structures acting on multiple distinct spaces. Subsequently, we  briefly introduce \emph{PROPs} and \emph{properads}, which generalize operads to include operations with multiple inputs and multiple outputs, and \emph{cyclic operads}, which are suitable in contexts endowed with a duality between inputs and outputs.

\subsection{Coloured Operads}
\label{sec:coloured_operads}

The equivalent definitions of an operad that we have considered so far share a common feature: they assume that all operations act on a single base object or ``type''. \emph{Coloured operads} generalize this concept to describe operations that act on a collection of different objects.

In a coloured operad, one starts with a set of ``colours'' $C$, which serves as an index set for the different types of objects. An operation in a coloured operad is thus endowed with a colour for each of its inputs and a colour for its single output.

\begin{definition}[Coloured operad]
\label{def:coloured_operad}
Let $C$ be a set, whose elements are called \emph{colours}. A \emph{(non-symmetric) coloured operad} $\mathcal{P}$ on $C$ consists of the following data:
\begin{enumerate}
    \item A collection of vector spaces $\{\mathcal{P}(c_1, \dots, c_n; c_0)\}_{n \in \mathbb{N}_0}$ for every choice of colours $c_0, c_1, \dots, c_n \in C$. The elements of $\mathcal{P}(c_1, \dots, c_n; c_0)$ are the \emph{abstract operations} with inputs of type $(c_1, \dots, c_n)$ and output of type $c_0$.
    
    \item For each $n, k_1, \dots, k_n \in \mathbb{N}_0$ and for every compatible choice of colours, a composition map, denoted $\gamma$:
\begin{equation}
\begin{split}
\mathcal{P}(c_1, \dots, c_n; c_0) & \otimes \mathcal{P}(d_{1,1}, \dots, d_{1,k_1}; c_1) \otimes \dots \\
& \otimes \mathcal{P}(d_{n,1}, \dots, d_{n,k_n}; c_n) \to \mathcal{P}(d_{1,1}, \dots, d_{n,k_n}; c_0).
\end{split}
\label{eq:your_label_here_split}
\end{equation} 
    
    \item For each colour $c \in C$, a unit operation $1_c \in \mathcal{P}(c;c)$.
\end{enumerate}
This data must satisfy associative and unit laws, taking into account the compatibility of colours.

A coloured operad is called \emph{symmetric} if, in addition, each space $\mathcal{P}(c_1, \dots, c_n; c_0)$ is equipped with a right action of the symmetric group $\mathbb{S}_n$ that permutes the inputs and their respective colours, and if the composition is equivariant with respect to this action.
\end{definition}

\begin{remark}
This definition is a direct generalization of the definition of an operad (see Definitions \ref{def: ns operad} and \ref{def: operad simmetrico}). Indeed, if the set of colours $C$ contains only one element, we recover the definition of a non-symmetric operad.
\end{remark}

\begin{remark}
    The definition \ref{def:coloured_operad} of a coloured operad is given in \texttt{Vect}, but it can be generalized to any symmetric monoidal category $(\mathscr{C},\otimes,u)$.
\end{remark}

\begin{remark}[Equivalence with multicategories]
The definition of coloured operad coincides with the definition of a \emph{multicategory} (symmetric or non-symmetric, depending on the case), in which the set of objects is the set of colours $C$ (see Definitions \ref{def:multicategory} and \ref{def: multicategoria simmetrica}). In other words, an operad on $C$ is a multicategory whose objects are the elements of $C$.
\end{remark}

The notion of a \emph{representation} of an operad (or an \emph{algebra over an operad}) extends to the case of a {coloured operad}.

\begin{definition}
    Let $\mathcal{P}$ be a coloured operad on $C$. A \emph{representation} of $\mathcal{P}$ consists of:
    \begin{enumerate}
        \item A collection of vector spaces $\{A_c\}_{c \in C}$, one for each colour $c \in C$.
        \item For each operation $p \in \mathcal{P}(c_1, \dots, c_n; c)$, where $c_1, \dots, c_n$ are the input colours and $c$ is the output colour, a linear map
        \begin{equation}
        \Phi_p : A_{c_1} \otimes \dots \otimes A_{c_n} \to A_c.
        \end{equation}
        These maps must satisfy the following three axioms:
        \begin{itemize}
          \item \emph{Associativity.}
            Given $p\in P(c_1,\dots,c_n;c)$ and $q_i\in P(c_{i,1},\dots,c_{i,k_i};c_i)$
            \begin{equation}
              \Phi_{\gamma(p;q_1,\dots,q_n)}
              \;=\;
              \Phi_p
              \,\circ\,
              \bigl(\Phi_{q_1}\otimes\cdots\otimes\Phi_{q_n}\bigr).
            \end{equation}
          
          \item \emph{Unit.}
            For each $c \in C$ and each \(1_c\in P(c;c)\)
            \begin{equation}
              \Phi_{1_c} \;=\; \mathrm{id}_{A_c}.
            \end{equation}
          
          \item \emph{Equivariance.}
            For each permutation \(\sigma\in \mathbb{S}_n\) and each
            \(p\in P(c_1,\dots,c_n;c)\),
            \begin{equation}
              \Phi_{p\cdot\sigma}
              \;=\;
              \Phi_p \,\circ\, \sigma,
            \end{equation}
            where \(\sigma\) acts on 
            \(A_{c_1}\otimes\cdots\otimes A_{c_n}\) by permuting the factors.
        \end{itemize}
    \end{enumerate}
\end{definition}

\subsubsection{Graphical Representation and Composition}
Graphically, the operations of a coloured operad are represented by planar trees whose half-edges are coloured to indicate the type of object they act upon. Composition is allowed only if the colours of the outputs of the ``inner'' operations match the colours of the inputs of the ``outer'' operation.

\begin{example}
Suppose we have an operad with four colours: orange ($O$), red ($R$), blue ($B$), and green ($G$). An operation $f \in \mathcal{P}(O, R, B, R; G)$ can be represented as follows.

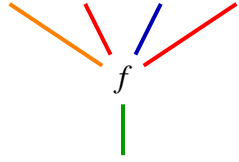
\begin{figure}[H]
    \centering
    \begin{tikzpicture}[thick]

        \colorlet{orange}{orange}
        \colorlet{red}{red}
        \colorlet{blue}{blue!70!black}
        \colorlet{green}{green!60!black}

        \node[font=\large] (f) at (0,0) {$f$};

        \draw[green, line width=1.5pt] (f) -- (0,-1);

        \draw[orange, line width=1.5pt] (f) -- (-1.5,1);
        \draw[red, line width=1.5pt] (f) -- (-0.5,1);
        \draw[blue, line width=1.5pt] (f) -- (0.5,1);
        \draw[red, line width=1.5pt] (f) -- (1.5,1);
    \end{tikzpicture}
    \caption{Representation of an operation $f: O \times R \times B \times R \to G$.}
\end{figure}

Consider two other operations, $g: O \times R \to O$ and $h: B \times G \to R$. The composition $g \circ (g,h)$ is well-defined, since the output colours of $g$ (orange) and $h$ (red) match the input colours of $g$.
\begin{figure}[H]
    \centering
    \begin{tikzpicture}[thick]

        \colorlet{orange}{orange}
        \colorlet{red}{red}
        \colorlet{blue}{blue!70!black}
        \colorlet{green}{green!60!black}

        \node[font=\large] (g_root) at (0,0) {$g$};
        \node[font=\large] (g_leaf) at (-1,1.5) {$g$};
        \node[font=\large] (h_leaf) at (1,1.5) {$h$};
        
        \draw[orange, line width=1.5pt] (g_root) -- (g_leaf);
        \draw[red, line width=1.5pt] (g_root) -- (h_leaf);
        \draw[orange, line width=1.5pt] (g_root) -- (0,-1);
        
        \draw[orange, line width=1.5pt] (g_leaf) -- (-1.5,2.5);
        \draw[red, line width=1.5pt] (g_leaf) -- (-0.5,2.5);
        \draw[blue, line width=1.5pt] (h_leaf) -- (0.5,2.5);
        \draw[green, line width=1.5pt] (h_leaf) -- (1.5,2.5);
    \end{tikzpicture}
    \caption{Valid composition $g \circ (g,h)$.}
\end{figure}
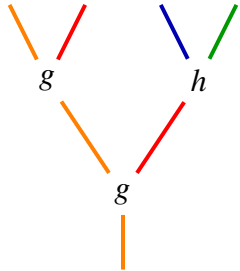
On the other hand, the composition $h \circ (g,g)$ is not allowed. The operation $h$ requires inputs of type blue and green, but it would receive two inputs of type orange from $g$. This colour mismatch makes the composition not well-defined.
\end{example}

\subsubsection{Examples}

\begin{example}[Coloured endomorphism operad]
The canonical example is the endomorphism operad on a family of vector spaces $\{V_c\}_{c \in C}$. The operations of this operad are the multilinear maps of the form:
\begin{equation}
f: V_{c_1} \otimes V_{c_2} \otimes \dots \otimes V_{c_n} \to V_{c_0}. 
\end{equation}
The composition laws are given by the usual composition of multilinear applications.
\end{example}

\begin{example}[The operad $\mathsf{As_{\bullet,\circ}}$ for an algebra morphism]
An application of coloured operads is to encode morphisms between algebras. The structure given by two associative algebras $(A, \mu_A)$, $(B, \mu_B)$ and an algebra morphism $f: A \to B$ can be encoded by a two-coloured operad, denoted $\mathsf{As_{\bullet,\circ}}$.

This encoding is achieved by using a set of two colours, $C \coloneqq \{\bullet, \circ\}$, which represent $A$ and $B$ respectively.

The operad $\mathsf{As_{\bullet,\circ}}$ is defined by the following generators, whose types are specified by their input and output colours:
\begin{itemize}
    \item A binary operation $\mu_\bullet$ of type $(\bullet, \bullet; \bullet)$, which corresponds to the product $\mu_A$ in $A$.
    \item A binary operation $\mu_\circ$ of type $(\circ, \circ; \circ)$, which corresponds to the product $\mu_B$ in $B$.
    \item A unary operation $f$ of type $(\bullet; \circ)$, which corresponds to the morphism $f: A \to B$.
\end{itemize}
Graphically, the three generators are:
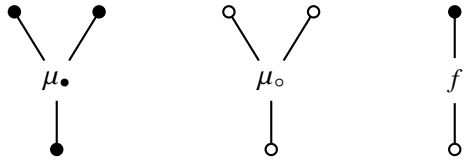
\[
\begin{tikzpicture}[baseline={([yshift=-.5ex]current bounding box.center)}, thick, scale=0.7]
    \node (v) at (0,0.5) {$\mu_\bullet$};
    \draw (v) -- (0,-0.8);
    \draw (v) -- (-0.8, 1.8);
    \draw (v) -- (0.8, 1.8);
    \node[draw, circle, fill=black, inner sep=1.5pt] at (0,-0.8) {};
    \node[draw, circle, fill=black, inner sep=1.5pt] at (-0.8, 1.8) {};
    \node[draw, circle, fill=black, inner sep=1.5pt] at (0.8, 1.8) {};
\end{tikzpicture}
\qquad\qquad
\begin{tikzpicture}[baseline={([yshift=-.5ex]current bounding box.center)}, thick, scale=0.7]
    \node (w) at (0,0.5) {$\mu_\circ$};
    \draw (w) -- (0,-0.8);
    \draw (w) -- (-0.8, 1.8);
    \draw (w) -- (0.8, 1.8);
    \node[draw, circle, fill=white, inner sep=1.5pt] at (0,-0.8) {};
    \node[draw, circle, fill=white, inner sep=1.5pt] at (-0.8, 1.8) {};
    \node[draw, circle, fill=white, inner sep=1.5pt] at (0.8, 1.8) {};
\end{tikzpicture}
\qquad\qquad
\begin{tikzpicture}[baseline={([yshift=-.5ex]current bounding box.center)}, thick, scale=0.7, font=\small]
    \node(f) at (0,0.5) {$f$};
    \draw (f) -- (0,1.8);
    \draw (f) -- (0,-0.8);
    \node[draw, circle, fill=white, inner sep=1.5pt] at (0, -0.8) {};
    \node[draw, circle, fill=black, inner sep=1.5pt] at (0, 1.8) {};
\end{tikzpicture}.\
\]
The relations defining this operad are the associativity of $\mu_\bullet$ and $\mu_\circ$, and the condition that $f$ is an algebra morphism, i.e., $f \circ \mu_\bullet = \mu_\circ \circ (f,f)$.

A representation of this coloured operad consists of the choice of two vector spaces $V_\bullet$ and $V_\circ$ and a set of linear maps that interpret the generators. The relations of the operad force these linear maps to define two associative algebra structures and a morphism from the first algebra to the second.
\end{example}

The presentation of the operad $\mathsf{As_{\bullet,\circ}}$ is intentionally informal. For a formal construction, it would have been necessary to extend the notions of the previous section to the coloured context, defining the free coloured operad, the ideal generated by relations, the quotient coloured operad, etc.

\subsection{Notes on Properads and PROPs}
\label{sec:props_properads}
The theory of operads, as we have developed it, is based on operations with an arbitrary number of inputs but \emph{only one output}. Operads are ideal for encoding associative, commutative, and Lie algebras. However, this single-output framework is insufficient to describe some fundamental algebraic structures. 

Perhaps the most important such structure is the \emph{bialgebra} \cite{markl08}.

\begin{example}[Bialgebra]
A \emph{bialgebra} is a vector space $V$ equipped with four linear maps: a \emph{product} $\mu: V \otimes V \to V$, a \emph{unit} $\eta: \mathbb{K} \to V$, a \emph{coproduct} $\Delta: V \to V \otimes V$, and a \emph{counit} $\varepsilon: V \to \mathbb{K}$.

This structure is required to be both a unital associative algebra and a counital coassociative coalgebra, whose structures are compatible. In other words, the following axioms must be satisfied:
\begin{enumerate}
    \item \emph{Algebra Structure:} $(V,\mu,\eta)$ forms a unital associative algebra. That is, for all $u,v,w \in V$:
    \begin{itemize}
        \item \emph{Associativity:} $\mu(\mu(u,v),w) = \mu(u,\mu(v,w))$.
        \item \emph{Unit:} For $1_V = \eta(1_\mathbb{K})$, we have $\mu(1_V, v) = v = \mu(v, 1_V)$.
    \end{itemize}

    \item \emph{Coalgebra Structure:} $(V,\Delta,\varepsilon)$ forms a counital coassociative coalgebra. That is, for all $v \in V$:
    \begin{itemize}
        \item \emph{Coassociativity:} $(\mathrm{id} \otimes \Delta) \circ \Delta(v) = (\Delta \otimes \mathrm{id}) \circ \Delta(v)$.
        \item \emph{Counit:} After identifying $V$ with $V \otimes \mathbb{K}$ and $\mathbb{K} \otimes V$, we have $(\varepsilon \otimes \mathrm{id}) \circ \Delta(v) = v = (\mathrm{id} \otimes \varepsilon) \circ \Delta(v)$.
    \end{itemize}

    \item \emph{Compatibility Condition:} The algebra and coalgebra structures must be compatible. As established in \cite[Chapter III, Theorem 111.2.1]{Kassel1995}, this compatibility is satisfied if and only if one of the following two equivalent conditions holds:
    \begin{itemize}
        \item[(a)] The coalgebra maps, $\Delta$ and $\varepsilon$, are morphisms of algebras.
        \item[(b)] The algebra maps, $\mu$ and $\eta$, are morphisms of coalgebras.
    \end{itemize}
\end{enumerate}
\label{ex:bialgebra}
\end{example}

The operadic framework is perfectly suited to describe algebraic structures built from multi-ary operations with a single output. For instance, an associative algebra, whose structure is defined by the binary product $\mu: V \otimes V \to V$, is precisely an algebra over the operad $\mathsf{As}$, as we have seen in Subsection \ref{subsec: as}.

It is also true that operads can describe dual structures. Through the notion of a \textit{$\mathcal{P}$-co-algebra} — defined by maps into the \emph{coendomorphism operad} $\mathrm{CoEnd}_V$ — a coassociative coalgebra can be characterized as an $\mathsf{As}$-co-algebra. See \cite[Section 5.2.17]{lodayvallette}.

The limitation of the operadic formalism becomes apparent, however, when we consider a \textit{bialgebra}. This structure is not merely an algebra or a co-algebra, but both simultaneously, linked by a crucial compatibility condition between product and coproduct (see Example \ref{ex:bialgebra}). A single operad cannot govern both an algebra structure (a map to $\mathrm{End}_V$) and a co-algebra structure (a map to $\mathrm{CoEnd}_V$) on the same space. To capture this richer structure, a more general formalism is required.

\emph{PROPs} (an acronym for \emph{PROducts and Permutations category}) are the mathematical objects designed to describe this type of structure. Although more general, they appeared before operads, in a work by Mac Lane in 1965 \cite{MacLane1965}.

Intuitively, a PROP $\mathcal{P}$ is a collection of spaces of ``abstract operations'' $\mathcal{P}(m,n)$, where $n$ is the number of inputs and $m$ is the number of outputs. These operations can be composed in two ways:
\begin{enumerate}
    \item \emph{Vertical composition} ($\circ$): The output of one operation becomes the input of another (the standard composition).
    \item \emph{Horizontal composition} ($\otimes$): Two operations are ``placed next to each other'' to form a single, combined operation.
\end{enumerate}
The paradigmatic example is once again the endomorphism operad, which in this context becomes the \emph{endomorphism PROP}. For a vector space $V$, this is defined as the collection
\begin{equation}
 \End_V := \{\End_V(m, n) = \Hom(V^{\otimes n}, V^{\otimes m})\}_{m,n \ge 0} 
\end{equation}
with the tensor product of linear maps as horizontal composition and the usual composition of multilinear maps as vertical composition.

PROPs, however, are combinatorially complex, in part because their horizontal composition allows for the description of operations that decompose into disconnected components. For many applications, it is sufficient to consider a more restricted version: \emph{properads} (see, e.g. \cite{markl08}).

The fundamental difference, informally, is that properads only allow compositions along \emph{connected} graphs (see, e.g. \cite{vallette14}). This condition precludes the possibility of composing operations on disjoint sets of inputs and outputs.

Every properad is also a PROP, but not vice versa.

\subsection{Notes on Cyclic Operads}
\label{sec:cyclic_operads}

For some operads, there is no clear distinction between ``inputs'' and ``outputs''. \emph{Cyclic operads}, introduced by Getzler and Kapranov \cite{GetzlerKapranov1995}, formalize this phenomenon. Informally, cyclic operads are operads equipped with an extra symmetry that allows the exchange of the output with one of the inputs, see e.g. \cite{markl08}.

To graphically represent the abstract operations of cyclic operads, \emph{cyclic trees} are used. A cyclic tree with $n+1$ \emph{legs} is an unrooted operadic tree whose half-edges are labeled with the elements of the set $\{0, 1, \dots, n\}$ (see Figure~\ref{fig:cyclic_tree}). Since we do not assume the choice of a root, it no longer makes sense to speak of input and output vertices \cite{markl08}.

\begin{figure}
\centering
\begin{tikzpicture}[scale=1.5, thick]
    \node[circle, fill, inner sep=1pt] (c) at (0,0) {};

    \node[circle, fill, inner sep=1pt] (v1) at (-0.8, 0.8) {};
    \node[circle, fill, inner sep=1pt] (v2) at (0.8, 0.8) {};
    \node[circle, fill, inner sep=1pt] (v3) at (-1.1, -0.4) {};
    \node[circle, fill, inner sep=1pt] (v4) at (1.1, -0.6) {};
    \node[circle, fill, inner sep=1pt] (v5) at (-0.3, -0.9) {};

    \draw (c) -- (v1);
    \draw (c) -- (v2);
    \draw (c) -- (v3);
    \draw (c) -- (v4);
    \draw (c) -- (v5);
    \draw (c) -- ++(-90:1.2cm) node[below] {$0$};
    \draw (c) -- ++(-40:1.4cm) node[below] {$3$};
    
    \draw (v1) -- ++(135:0.7cm) node[above left] {$8$};
    \draw (v1) -- ++(45:0.6cm) node[above right] {$1$};

    \draw (v2) -- ++(110:0.7cm) node[above] {$9$};
    \draw (v2) -- ++(20:0.8cm) node[right] {$2$};
    
    \draw (v3) -- ++(160:0.8cm) node[left] {$4$};
    \draw (v3) -- ++(-135:0.6cm) node[below left] {$5$};
    
    \draw (v4) -- ++(0:0.9cm) node[right] {$6$};
    \draw (v4) -- ++(-60:0.8cm) node[below right] {$7$};
\end{tikzpicture}
\caption{Example of a labeled cyclic tree. Adapted from \cite[Figure 8]{markl08}.}
\label{fig:cyclic_tree}
\end{figure}
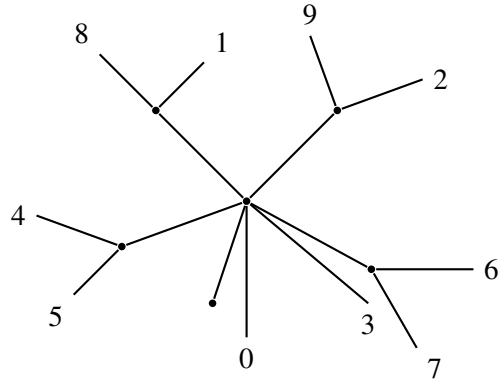

The natural algebraic context in which cyclic operads emerge is provided by vector spaces equipped with a symmetric, non-degenerate bilinear form. A basic example is provided by \emph{symmetric Frobenius algebras}. A symmetric Frobenius algebra is a unital associative algebra $F$ equipped with a symmetric non-degenerate bilinear form such that $\langle ab, c \rangle = \langle a, bc \rangle$, for all $a,b,c \in F$ \cite{vallette14}.


\begin{thebibliography}{99}

\bibitem[ATL18]{AppelLaredo2018}
\textsc{Appel, A.; Toledano Laredo, V.:} \newblock A 2-categorical extension of Etingof–Kazhdan quantisation.
\newblock \emph{Sel. Math. New Ser.}, \textbf{24}(4):3529–3617, 2018.
\newblock DOI: \texttt{10.1007/s00029-017-0381-z}.

\bibitem[BV73]{boardmanvogt}
\textsc{Boardman, J. M.; Vogt, R. M.:} \newblock \emph{Homotopy Invariant Algebraic Structures on Topological Spaces}.
\newblock Lecture Notes in Mathematics, Vol. 347. Springer-Verlag, Berlin, 1973.
\newblock DOI: \texttt{10.1007/bfb0068547}.

\bibitem[BM19]{BremnerMarkl2019}
\textsc{Bremner, M.; Markl, M.:} \newblock Distributive laws between the Three Graces.
\newblock \emph{Theory Appl. Categ.}, \textbf{34}:1317--1342, 2019.
\newblock URL: \url{www.tac.mta.ca/tac/volumes/34/41/34-41abs.html}.

\bibitem[C05]{Calaque2005}
\textsc{Calaque, D.:} \newblock Formality for Lie algebroids.
\newblock \emph{Commun. Math. Phys.}, \textbf{257}(3):563--578, 2005.
\newblock DOI: \texttt{10.1007/s00220-005-1350-5}.

\bibitem[CW15]{CalaqueWillwacher2015}
\textsc{Calaque, D.; Willwacher, T.:} \newblock Triviality of the higher formality theorem.
\newblock \emph{Proc. Am. Math. Soc.}, \textbf{143}(12):5181--5193, 2015.
\newblock DOI: \texttt{10.1090/proc/12670}.

\bibitem[DM69]{DeligneMumford1969}
\textsc{Deligne, P.; Mumford, D.:} \newblock The irreducibility of the space of curves of given genus.
\newblock \emph{Publ. Math., Inst. Hautes Étud. Sci.}, \textbf{36}:75--109, 1969.
\newblock DOI: \texttt{10.1007/BF02684599}.

\bibitem[Die17]{diestel17}
\textsc{Diestel, R.:} \newblock \emph{Graph Theory}.
\newblock 5th edition. Graduate Texts in Mathematics, Vol. 173. Springer, Berlin, 2017.
\newblock DOI: \texttt{10.1007/978-3-662-53622-3}.

\bibitem[D05]{Dolgushev2005a}
\textsc{Dolgushev, V. A.:} \newblock Covariant and equivariant formality theorems.
\newblock \emph{Adv. Math.}, \textbf{191}(1):147--177, 2005.
\newblock DOI: \texttt{10.1016/j.aim.2004.02.001}.

\bibitem[ED21]{esposito:2021}
\textsc{Esposito, C.; de Kleijn, N.:} \newblock $L_\infty$-resolutions and twisting in the curved context.
\newblock \emph{Rev. Mat. Iberoam.}, \textbf{37}(4):1581--1598, 2021.
\newblock DOI: \texttt{10.4171/rmi/1239}.

\bibitem[ENST25]{esposito:2025}
\textsc{Esposito, C.; Nest, R.; Schnitzer, J.; Tsygan, B.:} \newblock Quantization of the Momentum Map via $\mathfrak{g}$-adapted Formalities.
\newblock \emph{Preprint}, arXiv:2502.18295 [math.QA], 2025.
\newblock URL: \url{https://arxiv.org/abs/2502.18295}.

\bibitem[EGNO15]{EtingofGNO2015}
\textsc{Etingof, P.; Gelaki, S.; Nikshych, D.; Ostrik, V.:} \newblock \emph{Tensor Categories}.
\newblock Mathematical Surveys and Monographs, Vol. 205. American Mathematical Society, Providence, RI, 2015.
\newblock DOI: \texttt{10.1090/surv/205}.

\bibitem[ES98]{EtingofSchiffman1998}
\textsc{Etingof, P.; Schiffmann, O.:} \newblock \emph{Lectures on Quantum Groups}.
\newblock International Press, Boston, MA, 1998.

\bibitem[F09]{fresse2009}
\textsc{Fresse, B.:} \newblock \emph{Modules over Operads and Functors}.
\newblock Lecture Notes in Mathematics, Vol. 1967. Springer, Berlin, 2009.
\newblock DOI: \texttt{10.1007/978-3-540-89056-0}.

\bibitem[FTW18]{FresseTurchinWillwacher18}
\textsc{Fresse, B.; Turchin, V.; Willwacher, T.:}
\newblock The homotopy theory of operad subcategories.
\newblock \emph{J. Homotopy Relat. Struct.}, \textbf{13}(4):689--702, 2018.
\newblock DOI: \texttt{10.1007/s40062-018-0198-2}.

\bibitem[F17]{fresse2017}
\textsc{Fresse, B.:} \newblock \emph{Homotopy of Operads and Grothendieck-Teichmüller Groups. Part 1: The Algebraic Theory and Its Topological Background}.
\newblock Mathematical Surveys and Monographs, Vol. 217. American Mathematical Society, Providence, RI, 2017.
\newblock DOI: \texttt{10.1090/surv/217.1}.

\bibitem[Ger64]{Gerstenhaber1964}
\textsc{Gerstenhaber, M.:} \newblock On the deformation of rings and algebras.
\newblock \emph{Ann. Math. (2)}, \textbf{79}(1):59--103, 1964.
\newblock DOI: \texttt{10.2307/1970484}.

\bibitem[Get94]{Getzler1994}
\textsc{Getzler, E.:} \newblock Batalin-Vilkovisky algebras and two-dimensional topological field theories.
\newblock \emph{Commun. Math. Phys.}, \textbf{159}(2):265--285, 1994.
\newblock DOI: \texttt{10.1007/BF02102639}.

\bibitem[Get95]{Getzler1995}
\textsc{Getzler, E.:} \newblock Operads and moduli spaces of genus 0 Riemann surfaces.
\newblock In: R. Dijkgraaf, C. Faber, G. van der Geer (eds), \emph{The Moduli Space of Curves}. Progress in Mathematics, Vol. 129, pp. 199--230. Birkhäuser, Basel, 1995.
\newblock DOI: \texttt{10.1007/978-1-4612-4264-2\_9}.

\bibitem[GK95]{GetzlerKapranov1995}
\textsc{Getzler, E.; Kapranov, M. M.:} \newblock Cyclic operads and cyclic homology.
\newblock In: S.-T. Yau (ed), \emph{Geometry, Topology and Physics for Raoul Bott}. Conference Proceedings and Lecture Notes in Geometry and Topology, Vol. 4, pp. 167--201. International Press, Cambridge, MA, 1995.
\newblock DOI: \texttt{10.4310/CP.1995}.

\bibitem[GK94]{GinzburgKapranov1994}
\textsc{Ginzburg, V.; Kapranov, M. M.:} \newblock Koszul duality for operads.
\newblock \emph{Duke Math. J.}, \textbf{76}(1):203--272, 1994.
\newblock DOI: \texttt{10.1215/S0012-7094-94-07608-4}.

\bibitem[HP11]{hilgerponcin}
\textsc{Hilger, P.; Poncin, N.:} \newblock \emph{Lectures on Algebraic Operads}.
\newblock University of Luxembourg, 2011.
\newblock URL: \url{https://orbilu.uni.lu/handle/10993/14381}.

\bibitem[Kad80]{Kadeishvili1980}
\textsc{Kadeishvili, T. V.:} \newblock On the homology theory of fibre spaces.
\newblock \emph{Russ. Math. Surv.}, \textbf{35}(3):231--238, 1980.
\newblock DOI: \texttt{10.1070/RM1980v035n03ABEH001842}.

\bibitem[Kad89]{Kadeishvili1989}
\textsc{Kadeishvili, T.:} \newblock Structure of A($\infty$)-algebra and Hochschild and Harrison cohomology.
\newblock \emph{Tr. Tbilis. Mat. Inst. Razmadze Akad. Nauk Gruz. SSR}, \textbf{91}:19--27, 1989.
\newblock URL (arXiv): \texttt{https://arxiv.org/abs/math/0210331}.

\bibitem[Kas95]{Kassel1995}
\textsc{Kassel, C.:} \newblock \emph{Quantum Groups}.
\newblock Graduate Texts in Mathematics, Vol. 155. Springer-Verlag, New York, 1995.
\newblock DOI: \texttt{10.1007/978-1-4612-0783-2}.

\bibitem[Kel82]{Kelly1982}
\textsc{Kelly, G. M.:} \newblock \emph{Basic Concepts of Enriched Category Theory}.
\newblock London Mathematical Society Lecture Note Series, Vol. 64. Cambridge University Press, Cambridge, 1982.
\newblock URL (ristampa): \texttt{http://www.tac.mta.ca/tac/reprints/articles/10/tr10abs.html}.

\bibitem[KM95]{krizmay}
\textsc{Kriz, I.; May, J. P.:} \newblock \emph{Operads, algebras, modules and motives}.
\newblock Astérisque, Vol. 233. Société Mathématique de France, Paris, 1995.
\newblock URL: \url{http://www.numdam.org/issues/AST\_1995\_\_233\_/}.

\bibitem[Knu83]{Knudsen1983}
\textsc{Knudsen, F. F.:} \newblock The projectivity of the moduli space of stable curves, II: The stacks $M_{g,n}$.
\newblock \emph{Math. Scand.}, \textbf{52}:161--199, 1983.
\newblock DOI: \texttt{10.7146/math.scand.a-12001}.

\bibitem[Kon92]{Kontsevich1992}
\textsc{Kontsevich, M.:} \newblock Intersection theory on the moduli space of curves and the matrix Airy function.
\newblock \emph{Commun. Math. Phys.}, \textbf{147}(1):1--23, 1992.
\newblock DOI: \texttt{10.1007/BF02099526}.

\bibitem[Kon03]{Kontsevich2003}
\textsc{Kontsevich, M.:} \newblock Deformation quantization of Poisson manifolds.
\newblock \emph{Lett. Math. Phys.}, \textbf{66}(3):157--216, 2003.
\newblock DOI: \texttt{10.1023/B:MATH.0000027508.00421.bf}.

\bibitem[Lei04]{leinster04}
\textsc{Leinster, T.:} \newblock \emph{Higher Operads, Higher Categories}.
\newblock London Mathematical Society Lecture Note Series, Vol. 298. Cambridge University Press, Cambridge, 2004.
\newblock DOI: \texttt{10.1017/CBO9780511525896}.

\bibitem[Lei14]{leinster14}
\textsc{Leinster, T.:} \newblock \emph{Basic Category Theory}.
\newblock Cambridge Studies in Advanced Mathematics, Vol. 143. Cambridge University Press, Cambridge, 2014.
\newblock DOI: \texttt{10.1017/CBO9781107360068}.

\bibitem[LV12]{lodayvallette}
\textsc{Loday, J.-L.; Vallette, B.:} \newblock \emph{Algebraic Operads}.
\newblock Grundlehren der mathematischen Wissenschaften, Vol. 346. Springer, Berlin, 2012.
\newblock DOI: \texttt{10.1007/978-3-642-30362-3}.

\bibitem[Mac63]{maclane63}
\textsc{Mac Lane, S.:} \newblock Natural associativity and commutativity.
\newblock \emph{Rice Univ. Stud.}, \textbf{49}(4):28--46, 1963.
\newblock URL: \url{https://scholarship.rice.edu/handle/1911/62319}.

\bibitem[Mac65]{MacLane1965}
\textsc{Mac Lane, S.:} \newblock Categorical algebra.
\newblock \emph{Bull. Am. Math. Soc.}, \textbf{71}:40--106, 1965.
\newblock DOI: \texttt{10.1090/S0002-9904-1965-11234-4}.

\bibitem[Mac98]{maclane98}
\textsc{Mac Lane, S.:} \newblock \emph{Categories for the Working Mathematician}.
\newblock 2nd edition. Graduate Texts in Mathematics, Vol. 5. Springer, New York, 1998.
\newblock DOI: \texttt{10.1007/978-1-4757-4721-8}.

\bibitem[Mar08]{markl08}
\textsc{Markl, M.:} \newblock Operads and PROPs.
\newblock In: M. Hazewinkel (ed), \emph{Handbook of Algebra}, Vol. 5, pp. 87--140. North-Holland (Elsevier), Amsterdam, 2008.
\newblock DOI: \texttt{10.1016/S1570-7954(07)05002-4}.

\bibitem[MSS02]{markl_shnider_stasheff}
\textsc{Markl, M.; Shnider, S.; Stasheff, J.:}
\newblock \emph{Operads in Algebra, Topology and Physics}.
\newblock Mathematical Surveys and Monographs, Vol. 96. American Mathematical Society, Providence, RI, 2002.
\newblock DOI: \texttt{10.1090/surv/096}.

\bibitem[May72]{may_geometry}
\textsc{May, J. P.:}
\newblock \emph{The Geometry of Iterated Loop Spaces}.
\newblock Lecture Notes in Mathematics, Vol. 271. Springer, Berlin, 1972.
\newblock DOI: \texttt{10.1007/bfb0067491}.

\bibitem[MZZ25]{MayZhangZou2025}
\textsc{May, J. P.; Zhang, R.; Zou, F.:}
\newblock Unital operads, monoids, monads, and bar constructions.
\newblock \emph{Adv. Math.}, \textbf{461}, 110065, 2025.
\newblock DOI: \texttt{10.1016/j.aim.2024.110065}.

\bibitem[SS]{samchuckschnarch}
\textsc{Samchuck-Schnarch, S.:}
\newblock \emph{An Introduction to Operad Theory}. Lecture Notes, University of Ottawa.
\newblock URL: \url{https://alistairsavage.ca/pubs/Samchuck-Schnarch-Operads.pdf}.

\bibitem[S01]{smirnov}
\textsc{Smirnov, V. A.:} \newblock \emph{Simplicial and Operad Methods in Algebraic Topology}.
\newblock Translations of Mathematical Monographs, Vol. 198. American Mathematical Society, Providence, RI, 2001.
\newblock DOI: \texttt{10.1090/mmono/198}.

\bibitem[Sta63]{Stasheff1963}
\textsc{Stasheff, J. D.:} \newblock Homotopy associativity of H-spaces I, II.
\newblock \emph{Trans. Am. Math. Soc.}, \textbf{108}(2):275--312, 1963.
\newblock DOI (I): \texttt{10.2307/1993608}, DOI (II): \texttt{10.2307/1993609}.

\bibitem[Tam98]{Tamarkin1998}
\textsc{Tamarkin, D.:}
\newblock Another proof of M. Kontsevich formality theorem.
\newblock \emph{Preprint}, arXiv:math/9803025, 1998.
\newblock URL: \url{https://arxiv.org/abs/math/9803025}.

\bibitem[Val07]{Vallette2007}
\textsc{Vallette, B.:} \newblock A Koszul duality for PROPs.
\newblock \emph{Trans. Am. Math. Soc.}, \textbf{359}(10):4865--4943, 2007.
\newblock DOI: \texttt{10.1090/S0002-9947-07-04182-7}.

\bibitem[Val14]{vallette14}
\textsc{Vallette, B.:} \newblock Algebra + Homotopy = Operad.
\newblock In: T. Ratiu, A. Weinstein, S. Zelditch (eds), \emph{Symplectic, Poisson, and Noncommutative Geometry}, pp. 229--290. Cambridge University Press, Cambridge, 2014.
\newblock DOI: \texttt{10.1017/CBO9781139519419.009}.

\bibitem[Vit25]{Vitagliano2025}
\textsc{Vitagliano, L.:} \newblock \emph{Homology \& cohomology. A primer for undergraduates through applications}.
\newblock World Scientific, Singapore, 2025.
\newblock DOI: \texttt{10.1142/14116}.

\bibitem[W15]{Willwacher2015}
\textsc{Willwacher, T.:} \newblock M. Kontsevich’s graph complex and the Grothendieck–Teichmüller Lie algebra.
\newblock \emph{Invent. Math.}, \textbf{200}(3):671–760, 2015.
\newblock DOI: \texttt{10.1007/s00222-014-0528-x}.

\bibitem[Zha16]{zhang2016}
\textsc{Zhang, W.:} \newblock Group Operads and Homotopy Theory.
\newblock \emph{Homology Homotopy Appl.}, \textbf{18}(1):253--277, 2016.
\newblock DOI: \texttt{10.4310/HHA.2016.v18.n1.a12}.

\bibitem[Zwi93]{Zwiebach1993}
\textsc{Zwiebach, B.:} \newblock Closed string field theory: Quantum action and the Batalin-Vilkovisky master equation.
\newblock \emph{Nucl. Phys. B}, \textbf{390}(1):33--152, 1993.
\newblock DOI: \texttt{10.1016/0550-3213(93)90388-6}.
\end{thebibliography}
\end{document}